\documentclass{siamonline171218} 

\usepackage[utf8]{inputenc}

\usepackage{amstext,amsmath,amssymb,ifthen,lastpage,mathrsfs}
\allowdisplaybreaks

\usepackage{graphicx}
\usepackage{xcolor}
\usepackage{subfig}
\usepackage{scrextend}
\usepackage[colorinlistoftodos]{todonotes}

  \usepackage{hyperref}
   \hypersetup{
       colorlinks,
       linkcolor={blue!100!black},
       citecolor={blue!100!black},
       urlcolor={blue!100!black}
   }

\usepackage{cleveref}

\usepackage{enumitem}

\newcommand{\sC}{\mathscr{C}}

\newcommand{\sS}{\mathscr{S}}

\newcommand{\cA}{\mathcal{A}}

\newcommand{\cD}{\mathcal{D}}

\newcommand{\cF}{\mathcal{F}}

\newcommand{\cO}{\mathcal{O}}

\newcommand{\cS}{\mathcal{S}}
\newcommand{\cT}{\mathcal{T}}

\newcommand{\mN}{\mathbb{N}}
\newcommand{\mZ}{\mathbb{Z}}

\newcommand{\mR}{\mathbb{R}}
\newcommand{\mC}{\mathbb{C}}

\newcommand{\mG}{\mathbb{G}}
\newcommand{\mS}{\mathbb{S}}

 \newcommand{\be}{\ensuremath{\boldsymbol{e}}}
  
 \newcommand{\bb}{\ensuremath{\boldsymbol{b}}}
 \newcommand{\bt}{\ensuremath{\boldsymbol{t}}}
 
 \newcommand{\bj}{\ensuremath{\boldsymbol{j}}}
 \newcommand{\bx}{\ensuremath{\boldsymbol{x}}}

 \newcommand{\by}{\ensuremath{\boldsymbol{y}}}
 \newcommand{\ba}{\ensuremath{\boldsymbol{a}}}
 
  \newcommand{\bo}{\ensuremath{\boldsymbol{o}}}

 \newcommand{\bn}{\ensuremath{\boldsymbol{n}}}

 \newcommand{\bxi}{\ensuremath{\boldsymbol{\xi}}}
 \newcommand{\bdelta}{\ensuremath{\boldsymbol{\delta}}}

\newcommand{\tsigma}{\ensuremath{\tilde \sigma}}

\newcommand{\Text}[1]{\text{\textnormal{#1}}}

\newcommand{\D}{\Text{d}}
\newcommand{\I}{\Text{i}}
\newcommand{\E}{\Text{e}}
\newcommand{\MTEXT}[1]{\;\;\;\;\;\text{#1}\;\;\;\;\;}

\newcommand{\FA}{\MTEXT{for all}}
\newcommand{\closure}[2][3]{%
  {}\mkern#1mu\overline{\mkern-#1mu#2}}

\newcommand{\cc}[1]{{#1}^\ast}

\newcommand{\compl}[1]{#1^\textup{c}}

\newcommand{\F}{\lower0.9ex\hbox{$\mathchar'26$}\mkern-8mu \tNF}

\newcommand{\tNF}{\mathfrak{f}} 
\newcommand{\tNFa}[1]{\mathfrak{f}_{#1}} 
\newcommand{\mF}{m_{\mathfrak{f}}}
\newcommand{\mFa}[1]{m_{\mathfrak{f},#1}}
\newcommand{\nF}{n_{\mathfrak{f}}}
\newcommand{\kF}{k_{\mathfrak{f}}}
\newcommand{\sF}{\tilde \theta}
\newcommand{\eFa}[1]{\tilde \eta_{#1}}
\newcommand{\eFd}{\eFa{\tNFa {\Delta}}}

\newcommand{\iFd}{\tilde \iota_{\tNFa {\Delta}}}
\newcommand{\iFda}[1]{\tilde \iota_{\tNFa {\Delta}, #1}}

\newcommand{\SquareDom}{[-\frac 1 2; \frac 1 2]^m}
\newcommand{\SquareDomMin}[1]{[-\frac 1 2 + #1; \frac 1 2 - #1]^m}
\newcommand{\HIGHLIGHT}[1]{{#1}}	

\renewcommand{\Re}{\textup{Re}}

\newcommand{\tot}{{\textup{tot}}}
\newcommand{\sym}{{\textup{sym}}}
\newcommand{\obs}{{\textup{obs}}}

\newcommand{\leak}{{\textup{leak}}}
\newcommand{\stab}{{\textup{stab}}}

\newcommand{\low}{\textup{low}}
\newcommand{\band}{\textup{band}}

\newcommand{\real}{\textup{real}}

\newcommand{\BSpacea}[4]{\mathfrak{B}_{#1,#2,#3}^{#4}}
\newcommand{\BSpacem}[1]{\BSpacea{k}{r}{\bo}{#1}}
\newcommand{\BSpace}{\BSpacem{m}}

\newcommand{\CBand}{C_{\band}}

\newcommand{\CStab}{C_{\stab}}

\newcommand{\CRiesz}[1]{C_{\textup{Riesz},#1}}
\newcommand{\Csym}{C_{\textup{sym}}}
\newcommand{\Csyma}[1]{C_{\textup{sym}, #1}}

\newcommand{\OneD}{\textup{(1d)}}

\newcommand{\sfrac}[2]{{\textstyle \frac{#1}{#2}}}

\newcommand{\ssum}{{\textstyle \sum}}

\renewcommand{\L}{\ell}
\newcommand{\R}{{\textup r}}

\newcommand{\norm}[1]{\| #1 \|} 
\newcommand{\abs}[1]{| #1 |}
\newcommand{\parens}[1]{(#1 )}

\newcommand{\ceil}[1]{\lceil#1\rceil}

\newcommand{\ip}[2]{\langle #1, #2\rangle} 

\newcommand{\Norm}[1]{\left\| #1 \right\|} 
\newcommand{\Abs}[1]{\left| #1 \right|}
\newcommand{\Parens}[1]{\left(#1 \right)}

\newcommand{\Ip}[2]{\left\langle #1, #2\right\rangle} 

\newcommand{\bnorm}[1]{\big\| #1 \big\|} 
\newcommand{\babs}[1]{\big| #1 \big|}
\newcommand{\bparens}[1]{\big(#1 \big)}

\newcommand{\bip}[2]{\big\langle #1, #2\big\rangle} 

\newcommand{\Babs}[1]{\Big| #1 \Big|}
\newcommand{\Bparens}[1]{\Big(#1 \Big)}

\newcommand{\bbabs}[1]{\bigg| #1 \bigg|}
\newcommand{\bbparens}[1]{\bigg(#1 \bigg)}

\newcommand{\secref}[1]{$\S$\ref{#1}}

\newcommand{\nnl}{\nonumber \\}

\DeclareMathOperator*{\supp}{supp}

\DeclareMathOperator*{\dist}{dist}

\DeclareMathOperator*{\erfc}{erfc}
\DeclareMathOperator*{\sinc}{sinc}

\newtheorem{MyIP}{Inverse Problem}
\newtheorem{cor}[theorem]{Corollary}
\newtheorem{rem}[theorem]{Remark}
\newtheorem{example}[theorem]{Example}
\newtheorem{result}[theorem]{Result}

\numberwithin{theorem}{section}
\numberwithin{equation}{section}

\newcommand{\mylabel}[2]{#2\def\@currentlabel{#2}\label{#1}}

\deffootnote{1.5em}{1em}{(\thefootnotemark)\ }
\deffootnotemark{\textsuperscript{(\thefootnotemark)}}

\crefname{MyIP}{IP}{IPs}
\Crefname{MyIP}{IP}{IPs}

\begin{document}

\title{{Locality estimates for Fresnel-wave-propagation and stability of X-ray \HIGHLIGHT{phase contrast imaging} with finite detectors}\thanks{Submitted to the editors of Inverse Problems on May 16, 2018.\funding{This work was funded by Deutsche Forschungsgemeinschaft DFG through Project C02
of SFB 755 - Nanoscale Photonic Imaging.}}}

\author{Simon Maretzke\thanks{Institute for Numerical and Applied Mathematics, University of Goettingen, Lotzestrasse 16-18, 37083 G\"ottingen, Germany (\email{simon.maretzke@googlemail.com}).}} 

\headers{Locality estimates for Fresnel-wave-propagation}{Simon Maretzke}



\maketitle

\begin{abstract}
 Coherent wave-propagation in the \HIGHLIGHT{near-field Fresnel-regime} is the underlying contrast-mechanism to \HIGHLIGHT{(propagation-based) X-ray  phase contrast imaging} (\HIGHLIGHT{XPCI}), an emerging lensless technique that enables 2D- and 3D-imaging of biological soft tissues and other light-element samples down to nanometer-resolutions. Mathematically, propagation is described by the Fresnel-propagator, a convolution with an arbitrarily non-local kernel. As real-world detectors may only capture a finite field-of-view, this non-locality implies that the recorded diffraction-patterns are necessarily incomplete. This raises the question of stability of image-reconstruction from the truncated data -- even if the complex-valued wave-field, and not just its modulus, could be measured. Contrary to the latter restriction of the acquisition, known as the phase-problem, the finite-detector-problem has not received much attention in literature.
 The present work therefore analyzes locality of Fresnel-propagation in order to establish stability of \HIGHLIGHT{XPCI} with finite detectors. Image-reconstruction is shown to be severely ill-posed in this setting -- even without a phase-problem. However, quantitative estimates of the leaked wave-field reveal that Lipschitz-stability holds down to a sharp resolution limit that depends on the detector-size and varies within the field-of-view. The smallest resolvable lengthscale is found to be $\approx \! 1/\F$ times the detector's aspect length, where $\F$ is the Fresnel number associated with the latter scale.
 The stability results are extended to phaseless imaging in the linear contrast-transfer-function regime. 
\end{abstract}



\section{Introduction}

State-of-the-art high-resolution imaging techniques are a driving force behind current biomedical- and material science. Among such, \HIGHLIGHT{ (propagation-based) X-ray phase contrast imaging} (\HIGHLIGHT{XPCI}), also known as near-field holography, stands out as it yields two- or three-dimensional images down to nanometer-resolutions with high penetration-depths at relatively low radiation-dose and sample-preparation requirements  \cite{Wilkins1996,Pogany1997noninterferometric,Paganin1998,Cloetens1999holotomography,Bartels2015,MaretzkeEtAl2016OptExpr,Krenkel2017SingleCellHolotomography,Hagemann2017ResolutionExposureHoloVSCDI}.

 	\begin{figure}[hbt!]
 	 \centering
 	 \includegraphics[width=.8\textwidth]{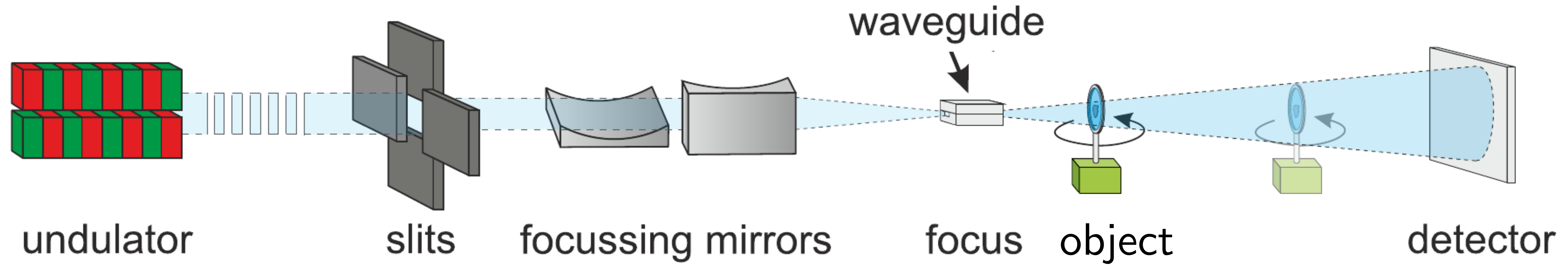}
 	 \caption{Exemplary experimental setup for (propagation-based) X-ray phase contrast imaging (XPCI) at synchrotrons (sketch of the GINIX-experiment \cite{Salditt2015GINIX} at P10-beamline, DESY).}
 	 \label{fig:setup_real}
 	\end{figure}

The  setup of \HIGHLIGHT{XPCI} is appealingly simple, see the example sketched in \cref{fig:setup_real}: essentially, it boils down to a coherent X-ray beam illuminating an unknown object and a detector that records the resulting near-field diffraction pattern, also termed \emph{hologram}, at a finite distance behind the sample. 
The coherent wave-propagation from the sample to the detector, described by the Helmholtz equation in the \emph{paraxial-} of \emph{Fresnel}-approximation 
\cite{Jonas2004TwoMeasUniquePhaseRetr,PaganinXRay}, is essential as it enables \emph{phase-contrast}: it partially encodes phase-shifts in the complex-valued X-ray wave-field $\Psi$ induced by refraction within the sample into measurable wave intensities $\propto |\Psi|^2$, thereby circumventing the well-known \emph{phase problem}, i.e.\ the inability to measure the phase of $\Psi$ directly. 
This permits imaging of biological soft tissues and other light-element samples, for which the absorption of X-rays -- but not refraction -- is negligible \cite{Pogany1997noninterferometric}.

To obtain an interpretable image of the sample, the induced phase-shifts (and absorption) have to be reconstructed from the measured hologram(s), i.e.\ an inverse problem has to be solved. By the limitation of the data to the squared modulus $|\Psi|^2$, this   requires to recover the missing phase-information. For the present setting, however, this task is comparably well-understood by now and routinely solved using data from multiple sample-detector-distances  along with a linearization of the contrast known as the contrast-transfer-function (CTF) model \cite{Cloetens1999holotomography,Turner2004FormulaWeakAbsSlowlyVarPhase,Krenkel2014BCAandCTF,Krenkel2017SingleCellHolotomography} and/or additional a priori knowledge on the recovered images \cite{Bartels2012,Bartels2015,PeinEtAl2016HoloWithShearlets,MaretzkeEtAl2016OptExpr}. Indeed, it is shown in previous work that the mild assumption of a known compact support of the image ensures well-posedness of the reconstruction in the linear CTF-regime \cite{MaretzkeHohage2017SIAM}.

What is typically tacitly ignored, however, is the data-incompleteness arising from the finiteness of the field-of-view captured by the detector due to the de-localizing action of (Fresnel-)wave-propagation: existing theory mostly assumes data within the complete infinite detector-plane and most reconstruction methods implicitly assume periodic detector-boundaries, possibly combined with artificial extension of the data by padding. While this produces reasonable results in practice, theoretical understanding for the effects of a finite detector and of the associated heuristic corrections is lacking.

This work aims to close this gap of theory by deriving rigorous estimates on the locality of information-transport by  wave-propagation in the Fresnel-regime with the ultimate goal of extending existing stability estimates for \HIGHLIGHT{XPCI} to settings with finite detectors. In particular, the focus is on the question of \emph{resolution}:
\vspace{.5em}
\begin{addmargin}[2em]{2em}
\emph{Given an \HIGHLIGHT{XPCI} setup, what is the size of the smallest sample-features that can be stably reconstructed from the measured data?}
\end{addmargin}
\vspace{.5em}
\HIGHLIGHT{In physics literature \cite{Nugent2010coherent,LatychenskaiaEtAl2012_HoloMeetsCDI(withHoloResolution)}, authors typically refer to Abbe's diffraction limit \cite{BornOptics,LipsonEtAl2010_OpticalPhysics}, defining the resolution via the {numerical aperture} associated with the detector-size. However, the underlying reasoning is heuristic and motivated by \emph{far-field} optics, despite the  {near-field} setting of the imaging technique. Rigorous theory is thus necessary to supplement physical intuition.}

The manuscript is organized as follows: \secref{S:Background} introduces the mathematical setting and notation as well as some preliminary insights  on the finite-detector problem. In \secref{S:GaussObservations}, the relation between resolution and detector-size is assessed by the study of Gaussian wave-packets, yielding \emph{best-case} estimates in some sense. These are then complemented by \emph{worst-case} estimates on stability of image-reconstruction derived in \secref{S:BoundsComplex}, \secref{S:Splines} and \secref{S:Real} under different a priori assumptions on the unknown objects. Having derived all of these results under the simplifying assumption that also the phase of the data is measured, the obtained locality- and stability estimates are then extended to the phaseless case of linearized \HIGHLIGHT{XPCI} in \secref{S:Phaseless}. \secref{S:Conclusions} concludes this work.

Despite the focus on \HIGHLIGHT{XPCI}, note that the derived estimates may be extended to a wide range of wave-propagation problems from classical physics and quantum-mechanics.

\vspace{.5em}

\section{Background} \label{S:Background}

\subsection{Basic setting} \label{SS2.0}

\subsubsection{Fresnel-propagation} \label{SSS:MathParaxProp} We consider the problem of reconstructing a function $h : \mR^m \to \mC$  from 
partial knowledge of  Fresnel-data:
\begin{align}
   \cD (h) &:= \cF \left( \mF \cdot \cF^{-1} ( h ) \right) \MTEXT{with}    \mF(\bxi) := \exp \left( - {\I \bxi^2 }/{(2 \tNF)}  \right), \quad \tNF > 0, \, \bxi \in \mR^m. \label{eq:DefFresnelProp} %
\end{align}
 $\cF (f)(\bxi) := (2\pi)^{-m/2} \int_{\mR^m} \exp(-\I \bxi \cdot \bx) f(\bx) \, \D \bx$ denotes the $m$-dimensional Fourier transform.
The \emph{Fresnel-propagator} $\cD$ models the free-space propagation of time-harmonic wave-fields $ \Psi(\bx, z) =  \tilde \Psi(\bx, z) \exp( \I \tNF z )$ with slowly varying envelope $\tilde \Psi$ within the regime of the paraxial Helmholtz-equation (see  e.g.\ \cite{PaganinXRay} for details):
\begin{align}
 \begin{cases} 
  \Parens{ 2 \I \tNF \partial_z + \Delta _{\bx} } \tilde \Psi(\bx, z) = 0  &(\bx,z) \in \mR^m \times (0;1) \\
    \tilde \Psi(\bx, 0) = h(\bx) &\bx \in \mR^m 
 \end{cases}
\quad \Rightarrow \quad \tilde \Psi(\cdot, 1) = \cD \Parens{ \tilde \Psi(\cdot, 0) }  \label{eq:ParaxialHelmholtz}
\end{align}

 	\begin{figure}[hbt!]
 	 \centering
 	 \includegraphics[width=.8\textwidth]{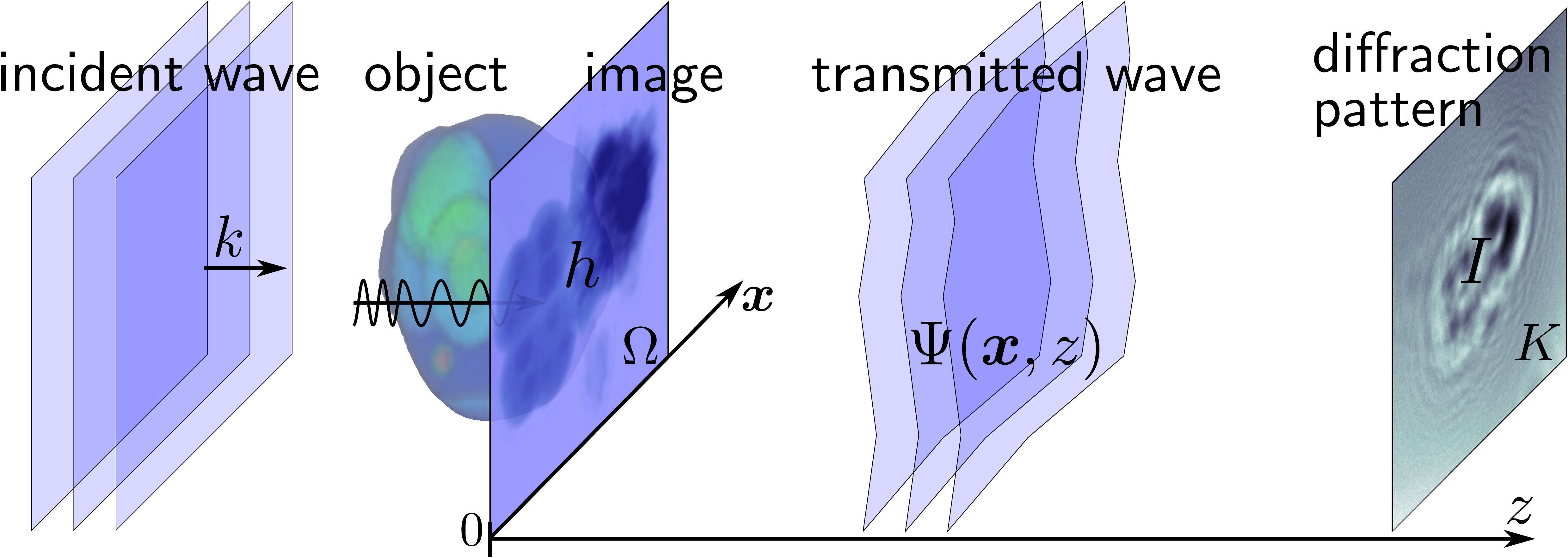}
 	 \caption{Schematic model of \HIGHLIGHT{propagation-based XPCI}: incident plane waves are scattered by a sample, imprinting phase-shifts and absorption $h =-\I \phi - \mu$ upon the transmitted wave-field $\Psi(\cdot,0)$ within the \emph{object-domain} $\Omega$. The intensity of the resulting near-field diffraction pattern $I= |\Psi(\cdot, d)|^2 $ is recorded within the \emph{detection-domain} $K$ at some distance behind the object.} 
 	 \label{fig:setup_sketch}
 	\end{figure}

\subsubsection{Forward models} \label{SSS:MathForwardModels}  
As detailed in \cite{Pogany1997noninterferometric,Jonas2004TwoMeasUniquePhaseRetr,PaganinXRay,MaretzkeHohage2017SIAM}, Fresnel-diffraction data arises in X-ray \HIGHLIGHT{ phase contrast imaging} (\HIGHLIGHT{XPCI}): if the incident beam in \cref{fig:setup_real} is modeled by a plane wave as sketched in \cref{fig:setup_sketch}, the wave-field in the object's exit-plane $z = 0$ is 
  \begin{align}
  \tilde \Psi(\cdot,0) &=   \exp(h)  \MTEXT{with} h = -\I \phi - \mu =- \I k \int_{\mR} \big( \delta - \I \beta \big)  \D z,  \label{eq:PhysModel}
\end{align}
 (within some standard approximations of X-ray optics \cite{Jonas2004TwoMeasUniquePhaseRetr}), where $n(\bx,z) = 1- \delta(\bx,z) + \I \beta(\bx, z)$ is the spatially varying refractive index of the sample. The complex-valued \emph{image} $h$ is thus a projection of the sample-characterizing quantities $\delta, \beta$. As the wave-field in the detector-plane relates to $\tilde \Psi(\cdot,0)$ via Fresnel-propagation, the detected intensities are given by
\begin{align}
 I(\bx) = \babs{ \cD\bparens{ \tilde \Psi(\cdot,0) } (\bx) }^2 = \Abs{ \cD\bparens{ \exp(h) } (\bx) }^2.  \label{eq:NLModel}
\end{align}

Under the additional assumption that the object is sufficiently weakly scattering for the image to be ``small'' in a suitable sense, \eqref{eq:NLModel} may be linearized:
\begin{align}
  I = 1 + \cT(h) + \cO( h ^ 2 ) \MTEXT{with} \cT(h) = 2 \Re \parens{ \cD( h ) }, \label{eq:LinModel}
\end{align}
where $\Re$ is the pointwise real part. In Fourier-space, the contrast in the phase-shifts $\phi$ and attenuation $\mu$ is then described by oscillatory \emph{contrast-transfer-functions} (CTF):
\begin{align}
 \cF\bparens{ \cT(-\I\phi-\mu) } (\bxi ) = -2 \sin \left( |\bxi|^2 / (2\tNF) \right) \cF(\phi)(\bxi )  -2 \cos \left( |\bxi|^2 / (2\tNF) \right) \cF(\mu)(\bxi )
\end{align}
for all $\bxi \in \mR^m$. Therefore, the linearized \HIGHLIGHT{XPCI}-model is also termed CTF-model.

Furthermore, it is often  assumed \cite{Paganin2002simultaneous,Turner2004FormulaWeakAbsSlowlyVarPhase} that the object is \emph{homogeneous}, in the sense that 
refraction and absorption
are proportional: $h =  -  \mu - \I \phi = -\I \E^{-\I \alpha} \varphi$
  for some $\alpha \in [0;\pi)$ and a  \emph{real-valued} function $\varphi$. In the linearized case, this yields a modified CTF-model: 
  \begin{align}
  \cS_\alpha ( \varphi) :=  \cT(h) =  -2 \cF^{-1} \left( s_\alpha \cdot \cF (\varphi) \right), \quad s_\alpha(\bxi ) := \sin \left(|\bxi|^2 / (2\tNF)  + \alpha \right)\label{eq:fwModel2}
\end{align}
 The case $\alpha = 0$ corresponds to \emph{pure phase objects} which induce negligible absorption $\mu \approx 0$. 
  
  \vspace{.5em}
\subsubsection{Object- and detection-domains} \label{SSS:MathObjDetDomains} We assume that the approximate size of the imaged object is known a priori. Then there exists a bounded  \emph{object-domain} $\Omega \subset \mR^m$ such that the unknown image $h$ satisfies $\supp(h) = \closure{ \{ \bx \in \mR^m : h(\bx) \neq 0 \} } \subset \Omega$, where the overbar denotes set-closure. We consider the $L^2$-functions  satisfying this \emph{support-constraint}:
\begin{align}
 h \in L^2 (\Omega) := \big\{ h: \mR^m \to \mC : \supp(h ) \subset \Omega, \, {\textstyle \int_{\mR^m} |h|^2\, \D \bx } < \infty \big\}.
\end{align}
Throughout this work, $\ip f g := \int_{\mR^m} f(\bx) g(\bx)^\ast \, \D \bx$ and $\norm{h} := \ip h h ^{1/2}$ refer to the inner-product and norm in the space of square-integrable functions $L^2(\mR^m)$.

Contrary to most previous work, we account for the fact that real-world detectors may only record data 
within a bounded \emph{detection-domain} $K \subsetneq \mR^m$, also referred to as \emph{field-of-view} (\emph{FoV}) or simply \emph{detector}. 
Thus, only \emph{restrictions} $I|_K$ of the intensity-data in \eqref{eq:PhysModel}, defined by $I|_K(\bx) = I(\bx)$ for $\bx \in K$ and $I|_K(\bx) = 0$ otherwise, are available. By considering continuous measurements, however, we neglect that detectors are composed of \emph{discrete} pixels.

For the \HIGHLIGHT{XPCI}-setting, a two-dimensional square detector $K = [-\frac 1 2; \frac 1 2]^2$ is certainly of highest practical relevance. By the analysis in \cite{Kostenko2013_AllAtOncePCTWithTV,Ruhlandt2016_RadonPCICommute}, however, also the case $m = 3$ is of interest as it arises in a linearized model of \emph{tomographic} imaging. Moreover, $m = 3$ is also the natural dimension for an alternate application from quantum-mechanics:
\vspace{.5em}
\begin{rem}[Application in quantum mechanics] \label{rem:Schroedinger}
 The paraxial Helmholtz equation in \eqref{eq:ParaxialHelmholtz} is equivalent to the time-dependent Schr\"odinger-equation for a free electron if $z$ is identified with the time-dimension. Accordingly, all results of this work can be interpreted in view of the question how much probability-mass of a \HIGHLIGHT{quantum-mechanical} wave-function, initially localized in $\Omega \subset \mR^m$, leaks out of some domain $K \subset \mR^m$ upon time-propagation.
\end{rem}
\vspace{.5em}
\noindent Therefore, the analysis is carried out independently of the dimension $m$ as far as possible.

\vspace{.5em}
\subsubsection{Fresnel number(s)} \label{SSS:MathFresnelNumber} The dimensionless parameter $\tNF$ in \eqref{eq:DefFresnelProp} is the (modified) Fresnel number of the imaging setup (related to the classically defined Fresnel number $\F$ by a convenient $2\pi$-factor: $\tNF= 2\pi \F$). It is defined as $\tNF = k b^2/d$, where $k$ is the wavenumber of the incident plane-wave in \cref{fig:setup_sketch}, $d$ is the distance between object- and detector-plane and $b$ is the physical length that corresponds to unity in the dimensionless coordinates $\bx$. The value of $\tNF$ determines how strongly structures of lengthscale 1 in an object $h$ are distorted upon Fresnel-propagation $h \mapsto \cD(h)$: for $\tNF \gg 1$, structures are essentially preserved whereas $\tNF \ll 1$ corresponds to full far-field diffraction.

The FoV will typically be taken as the unit-square, $K = \SquareDom$. Thereby, the Fresnel number $\tNF$ is implicitly defined with $b$ as the detector's physical aspect length. Typical values are then in the range $10^3 \lesssim \tNF \lesssim 10^5$ for high-resolution \HIGHLIGHT{XPCI}-experiments at synchrotrons. By the freedom in choosing $b$, however, one can also associate a Fresnel number with any other lateral scale: if $\sigma$ is a dimensionless length, $\tNFa \sigma := \sigma^2 \tNF$ is the Fresnel number that describes diffraction on the physical scale corresponding to $\sigma$.

%
\vspace{.5em}
\subsubsection{Inverse problems} \label{SSS:MathInvProblems} In order to study \HIGHLIGHT{XPCI} with a finite FoV, we consider image-reconstruction problems both with complex- and phaseless Fresnel-data ($\bdelta$: data-errors):

\vspace{.5em}
 \begin{MyIP}[Reconstruction of complex-valued images] \label{IP:Complex}
  For $\Omega, K \subset \mR^m$, reconstruct a \emph{complex-valued} $h \in L^2(\Omega)$ from either of the following data: \vspace{.25em}
  \begin{enumerate}[leftmargin=2em]
   \item[\textup{(a)}] $g^\obs_{\textup{(a)}}  =  \cD(h)|_K  + \bdelta$
   \item[\textup{(b)}] $g^\obs_{\textup{(b)}}  =  \cT(h)|_K + \bdelta$ 
   \item[\textup{(c)}] $g^\obs_{\textup{(c)}}  = \babs{\cD(\exp(h))}^2 |_K + \bdelta$
  \end{enumerate} 
 \end{MyIP}
 \vspace{.5em}
   \begin{MyIP}[Reconstruction of real-valued (homogeneous) images] \label{IP:Real}
  For  $\Omega, K \subset \mR^m$ and $\alpha \in [0;\pi)$, reconstruct a \emph{real-valued} $\varphi \in L^2(\Omega, \mR)$ from either of the following data: \vspace{.25em}
  \begin{enumerate}[leftmargin=2em]
   \item[\textup{(a)}] $g^\obs_{\textup{(a)}}  =  \cD(\varphi)|_K  + \bdelta$
   \item[\textup{(b)}] $g^\obs_{\textup{(b)}}  =  \cS_\alpha(\varphi) |_K + \bdelta$	
    \item[\textup{(c)}] $g^\obs_{\textup{(c)}}  = \babs{\cD(\exp( -\I \E^{-\I \alpha} \varphi ))}^2 |_K + \bdelta$
  \end{enumerate}  
 \end{MyIP}
\vspace{.5em}
 
It should be emphasized that reconstructions in the setting of \cref{IP:Real}(b),(c), 
are currently standard in \HIGHLIGHT{XPCI}, 
whereas solving \cref{IP:Complex}(b),(c) is typically considered too unstable due to the larger number of unknowns to be recovered. Yet, it is not at all obvious that \cref{IP:Complex} and \cref{IP:Real} also exhibit different effects due to a finite FoV, \HIGHLIGHT{i.e.\ that real-valuedness\footnotemark[1] is relevant for the present study}. Surprisingly, however, this indeed turns out to be the case.
  
To identify the effects of a finite FoV,
we will mostly consider the non-phaseless problems \cref{IP:Complex}(a) and \cref{IP:Real}(a). By the richness of measured data, however, the problems (b) and (c) are clearly harder to solve than the variants (a) and \cref{IP:Real}(a)  is easier to solve than any of the others. In particular, this ``hierarchy-of-difficulties'' means that any instabilities in \cref{IP:Complex}(a) and \cref{IP:Real}(a)  will necessarily also be present in the phaseless problems.

\footnotetext[1]{\HIGHLIGHT{Although we will refer to ``real-valued'' signals throughout the work, note that all the results obtained for such trivially extend to signals given by real-functions multiplied by a global complex phase.}}



\vspace{.5em}
\subsection{Properties of the Fresnel propagator} \label{SS:FresnelProps}
As a preparation for the subsequent analysis, we summarize some basic properties of the Fresnel propagator, see also \cite{PaganinXRay,LieblingEtAl2003Fresnelets,HomannDiss2015}:
\vspace{.5em}
\begin{enumerate} 
 \item[(P1)] \emph{Unitary operator:} The map $\cD: L^2(\mR^m) \to L^2(\mR^m)$ defines a linear \emph{isometry}:
 \begin{align}
  \norm{\cD(f)} = \norm{f}  \FA  f  \in L^2(\mR^m).  \label{eq:FresnelPropUnitary} \tag{P1}  
 \end{align} 
 This also implies that $\cT, \cS_\alpha\!: L^2(\mR^m) \to L^2(\mR^m)$ are bounded with $\norm{\cT}, \norm{\cS_\alpha} \leq 2$.  \vspace{.75em}
 \item[(P2)] \emph{Convolution form:} As a Fourier-multiplier, $\cD$ can be alternatively written as a convolution: for all $f  \in L^1(\mR^m) \cap L^2(\mR^m)$, it holds that
  \begin{align}
  \cD(f)(\bx) &=  \left( \nF \ast f \right)(\bx) = \int_{\mR^m} \kF(\bx  - \by) f(\by) \D \by \FA \bx \in \mR^m \nnl
  \kF &= u_0 \Parens{ \tNF /( 2 \pi ) }^{\frac m 2 } \cdot  \nF, \quad \nF(\bx) = \exp\left( \I \tNF \bx^2 / 2 \right), \quad   u_0 = \exp\left(-{\I m \pi}/4 \right)    \label{eq:FresnelPropConvForm} \tag{P2}
 \end{align} \vspace{-.75em}
 \item[(P3)] \emph{Alternate form:} By rearranging the convolution-formulation \eqref{eq:FresnelPropConvForm}, the following alternate form of the Fresnel propagator can be obtained: 
   \begin{align}
  \cD(f)(\bx) &= u_0 \tNF^{\frac m 2 } \nF(\bx) \cdot \cF\left( \nF \cdot f \right)(\tNF \bx) \FA \bx \in \mR^m \label{eq:FresnelPropAltForm}  \tag{P3}
 \end{align} \vspace{-.75em}
 \item[(P4)] \emph{Separability:} $\mF$ factorizes into a product of functions of a single coordinate:
  \begin{align}
  \mF(\bxi) &= \exp\left( - \frac{\I  \bxi^2}{2\tNF} \right) =  \exp\bbparens{ - \frac{\I }{2\tNF} \sum_{j=1}^m \xi_j^2 } = \prod_{j=1}^m \exp\left( - \frac{\I  \xi_j^2}{2\tNF} \right)  = \prod_{j=1}^m \mFa j (\bxi), \nnl
  \mFa j (\bxi) &:= \mF(\xi_j)   \FA \bxi = (\xi_1,\ldots, \xi_m) \in \mR^m.  \nonumber
 \end{align}
 Consequently, $\cD$ factorizes into a commuting product of quasi-1D Fresnel-propagators acting along the different dimensions:
 \begin{align}
  \cD(f) &= \cD_1 \ldots \cD_m(f) = \cD_m \ldots \cD_1(f) \label{eq:FresnelPropSeparable}  \tag{P4} \\
  \cD_j(f) &:= \cF^{-1} \left( \mFa j \cdot \cF(f) \right) =  \cF_j^{-1} \left( \mFa j \cdot \cF_j(f) \right). \nonumber
 \end{align}
  $\cF_j: L^2(\mR^m) \to L^2(\mR^m)$ is the 1D-Fourier transform along the $j$th dimension. \vspace{.75em}
 
 \item[(P5)]  \emph{Isotropy and translation invariance:} As a convolution operator, $\cD$ is translation invariant, i.e.\ commutes with coordinate-shifts. As $\mF$ is invariant under orthogonal transformations, i.e.\ $\mF(A\bxi) = \mF(\bxi)$ for all $\bxi \in \mR^m, \, A \in O(m)$, $\cD$ also commutes with orthogonal coordinate transforms, i.e.\ acts \emph{isotropically} along all dimensions:
 \begin{align}
   \cD \cA =  \cA \cD \MTEXT{for all} \cA : L^2(\mR^m) \to L^2(\mR^m); \; &\cA(f)(\bx) = f(A\bx + \ba) \label{eq:FresnelPropIsotropy}  \tag{P5} \\
   &A \in O(m), \ba\in \mR^m. \nonumber
 \end{align} \vspace{-.75em}
 \item[(P6)]  \emph{Extension to distributions:} $\cD$ can be extended to tempered distributions $\sS(\mR^m)'$, i.e.\ to the dual space of smooth and rapidly decaying Schwartz-functions $\sS(\mR^m)$:
 \begin{align}
   (\cD (T))(u) := T( \cD(u) ) \FA T \in \sS(\mR^m)', u \in \sS(\mR^m). \label{eq:FresnelPropDistributions} \tag{P6}
 \end{align}
 In particular, one has $\cD(1) = 1$ for the constant $1$-function. Moreover, by continuity of $\cD, \cF: \sS(\mR^m)' \to \sS(\mR^m)'$, \eqref{eq:FresnelPropAltForm} remains valid in a distributional sense.
\end{enumerate}

\vspace{.5em}
\subsection{Preliminary results}	

 We aim to characterize the ill-posedness of inverse problems \cref{IP:Complex} and \cref{IP:Real}. Let us first note that Fresnel-propagation is, in principle, arbitrarily non-local:
 
   \vspace{.5em}
  \begin{theorem}[Arbitrary non-locality of Fresnel-propagation] \label{thm:FresnelNonlocal}
  Let $0 \neq h \in L^2(\mR^m)$ have compact support. Then $\cD(h)$ is supported within the whole $\mR^m$, $\supp(\cD(h)) = \mR^m$.
 \end{theorem}
 \vspace{.5em}
 \begin{proof}
  By \eqref{eq:FresnelPropAltForm} and the Paley-Wiener-Schwartz-theorem, $\cD(h)$ is an entire analytic function. Thus, $\cD(h)$ is non-zero almost everywhere in $\mR^m$.
 \end{proof}
 \vspace{.5em}
 
 Accordingly, measuring diffraction-data only within a finite FoV will always result in some information-\emph{leakage}. One might think that this ultimately introduces \emph{non-uniqueness} of the reconstruction. This is however not the case, as has been shown in previous work:
 
  \vspace{.5em}
  \begin{theorem}[Uniqueness \text{\cite{Maretzke2015IP}}] \label{thm:Unique}
  Let $\Omega \subset \mR^m$ bounded and let $K \subset\mR^m$ contain an open set. Then \cref{IP:Complex} and \cref{IP:Real} are uniquely solvable (up to periodicity of the exponential in (c)).
 \end{theorem}
 \vspace{.5em}
 
  \Cref{thm:Unique} means that the question, whether a small detection-domain $K$ raises issues, admits no simple yes-no-answer. Indeed, it implies that the effects of the size of $K$ can only be understood by studying \emph{stability}. We recall that -- for infinite detectors -- the \emph{linear} inverse problems \cref{IP:Complex}(a),(b) and \cref{IP:Real}(a),(b) are Lipschitz-stable, i.e.\ \emph{well-posed}:

   \vspace{.5em}
  \begin{theorem}[Well-posedness for infinite detectors and compact supports \text{\cite{MaretzkeHohage2017SIAM}}] \label{thm:WellPosed}
  Let $\Omega \subset \mR^m$ be bounded and let $K = \mR^m$. Then \cref{IP:Complex}(a),(b) and \cref{IP:Real}(a),(b) are well-posed, i.e.\ if $T: L^2(\Omega) \to L^2(\mR^m)$ denotes the corresponding forward operator, then there exist constants $C_{\textup{stab}}^{\textup{IP}\ast} > 0$, depending on $\tNF, m, \Omega$ (and $\alpha$), such that 
  \begin{align}
   \Norm{ T ( h) } \geq C_{\textup{stab}}^{\textup{IP}\ast}\norm{h} \MTEXT{for all} h \in L^2(\Omega). \label{eq:thm-WellPosed-1}
  \end{align} 
 \end{theorem}
 \vspace{.5em}
 \begin{proof}
  For \cref{IP:Complex}(a) and \cref{IP:Real}(a), the result is due to the unitarity of the Fresnel propagator \eqref{eq:FresnelPropUnitary} and one has $C_{\textup{stab}}^{\textup{IP1(a)}}  = C_{\textup{stab}}^{\textup{IP2(a)}}  = 1$. For \cref{IP:Complex}(b) and \cref{IP:Real}(b), the general statement along with estimates of the constants $C_{\textup{stab}}^{\textup{IP}\ast}$ is proven in \cite{MaretzkeHohage2017SIAM}.
 \end{proof}
 \vspace{.5em}
 
 The point of Lipschitz-stability estimates of the form \eqref{eq:thm-WellPosed-1} is that they are necessary and sufficient for the operator $T$ to have a bounded (pseudo-)inverse $T^\dagger$ and thereby ensure that data-errors $\bdelta$ induce only \emph{bounded} deviations $\leq \parens{ C_{\textup{stab}}^{\textup{IP}\ast} }^{-1} \norm{\bdelta}$ in the reconstructions. Clearly, one would like to have similar results for finite detectors $K \subsetneq \mR^m$. 
 However, the following theorem shows that stability may deteriorate dramatically due to a finite FoV:

  \vspace{.5em}
 \begin{theorem}[Severe ill-posedness for bounded detectors] \label{thm:Illposed}
  Let $\Omega, K \subset \mR^m$ be bounded with non-empty interior. Then \cref{IP:Complex} and \cref{IP:Real} are severely ill-posed.
 \end{theorem}
 \vspace{.5em}
 \begin{proof}
  By the hierarchy-of-difficulty discussed in \secref{SS2.0}, it is sufficient to prove the claim for \cref{IP:Real}(a). Accordingly, we have to consider the singular values of the forward operator $T: L^2(\Omega, \mR) \to L^2(K); \, h \mapsto \cD(h)|_K$. Thus, we compute $T^\ast T$. Using the convolution-form \eqref{eq:FresnelPropConvForm}, it can be shown that, for arbitrary $h \in L^2(\Omega, \mR)$,
  \begin{align*}
   T^\ast T(\bx) 
   &= \int_{\Omega} \bbparens{ \underbrace{ \int_{K} \Re\Parens{  \kF(\bx - \by) \cdot  \cc{\kF(\by- \by') } } \, \D \by }_{k (\bx, \by')} } h(\by') \, \D \by' \MTEXT{for all} \bx \in \Omega.
  \end{align*}
  Accordingly, $T^\ast T$ is given by an integral-operator with kernel $k $. Since $\kF$ is bounded and infinitely smooth, so is $k $ and $k  \in L^2(\Omega \times \Omega)$ by boundedness of $\Omega$. In total, this implies that $T^\ast T$ is an infinitely smoothing compact integral-operator so that its eigenvalues, the squared singular values of $T$, decay super-algebraically. This shows that \cref{IP:Real}(a) and hence all considered inverse problems are severely ill-posed.
 \end{proof}
  \vspace{.5em}
  
  Importantly, the severe ill-posedness arises \emph{independently} of the phase-problem, i.e.\ also for reconstructions from seemingly complete Fresnel-data $\cD(h)|_K$.
  In practice, the result means that there will always be a large number of image-modes that cannot be recovered from finite detector data at any realistically achievable noise-levels.
  This prediction is in contradiction to the stable reconstructions achieved in practical \HIGHLIGHT{XPCI} and thus necessitates a deeper analysis of the nature  of the found ill-posedness.

\vspace{.5em}
\section{Assessment by Gaussian wave-packets} \label{S:GaussObservations}

In the following, we aim to assess stability of \cref{IP:Complex} and \cref{IP:Real} by considering Gaussian wave-packets as a special class of object-signals $h$, for which Fresnel-propagation may be computed analytically. The theory is completely analogous to the textbook-example of wave-packets for the time-dependent Schr\"odinger-equation.

\vspace{.5em}
\subsection{The Gaussian-beam solution} \label{SS:GaussianBeam}

We consider centered Gaussians of  width $\sigma > 0$:  
\begin{align}
 p_{\sigma}(\bx) =   \left( 2 \pi \sigma^2 \right)^{-m/2} \exp \left( -   \frac{\bx^2}{2 \sigma^2}   \right) \MTEXT{for all} \bx   \in \mR^m \label{eq:FoV-2}
\end{align}
Owing to the Gaussian form,  $\cD( p_{\sigma} )$ can be computed explicitly. It constitutes an exact solution to the paraxial Helmholtz equation \eqref{eq:ParaxialHelmholtz} known as the \emph{Gaussian beam}, see e.g.\ \cite[Sec.\ 3.1]{Teich1991Photonics}. With a certain unitary factor $c_0$, it can be written in the form
\begin{align}
   \cD( p_{\sigma} )(\bx)   &= \frac{ \tsigma^{m/2} c_0 }{ \sigma^{m/2} } \exp\left(\frac{\I \bx^2}{2 \eta^2} \right)  p_{\tsigma}(\bx), \qquad   \eta^2 := \frac{ 1 + \sigma^4 \tNF^2 }{\tNF},  \quad \tsigma ^2  := \frac{ \eta^2 }{ \sigma^2 \tNF }. \label{eq:GaussianBeam}
\end{align}
Accordingly, $ \cD( p_{\sigma} ) $ is again of Gaussian shape, yet modulated by a unitary oscillatory factor.

Consider the limit $\sigma \to 0$ of a more and more localized peak. Then the \emph{propagated width} $\tsigma$ tends to infinity according to \eqref{eq:GaussianBeam}, i.e.\ the propagated Gaussian $\cD( p_{\sigma} )$ becomes arbitrarily \emph{delocalized}. 
Indeed, it holds that
\begin{align}
 \lim_{\sigma \to 0}  \norm{\cD( p_{\sigma} )|_{K}}/\norm{ p_{\sigma}} = 0  \label{eq:FoV-2-limit}
\end{align}
for any bounded detection-domain $K \subset \mR^m$. The example indicates that, asymptotically, the sharper a feature in the object the less contrast it induces in the diffraction data on a finite detector $K$. Accordingly, a finite FoV limits the achievable \emph{resolution}. 

\vspace{.5em}
\subsection{Gaussian wave-packets}\label{SS:GaussianWavePackets}

In order to further investigate the relation between the detection-domain $K$ and resolution, we study the propagation of Gaussian wave-packets, given by a Gaussian peak that is modulated by a sinusoidal oscillation:
\begin{align}
 h _{\bxi, \ba}(\bx) :=  \exp\bparens{ \I \bxi  \cdot (\bx - \ba) } p_\sigma(\bx - \ba), \quad \bxi, \ba \in \mR^m. \label{GaussWavePacket}
\end{align}
Analytical propagation of such signals is
enabled by the following lemma:

\vspace{.5em}
\begin{lemma}[Fresnel propagation under frequency shifts] \label{lem:FresnelPropFreqShift}
 For $\bxi, \bb \in \mR^m$,  $\be_{\bxi }(\bx) := \exp(\I \bxi \cdot \bx)$ denote the Fourier mode to the frequency $\bxi$ and $T_{\bb}: f \mapsto f((\cdot) + \bb)$ the translation by $\bb$. Then it holds for all $f \in L^2(\mR^m)$ that
   \begin{align}
  \cD( \be_{\bxi} \cdot f ) &=   \mF(\bxi) \cdot \be_{\bxi}  \cdot T_{-\bxi/\tNF}\left( \cD  (f)  \right)  \label{eq:FresnelPropFreqShift}
 \end{align}
 where $\mF$ is the Fresnel factor from \eqref{eq:DefFresnelProp}.
\end{lemma}
\vspace{.5em}
 
 \Cref{lem:FresnelPropFreqShift} is proven in \cref{App:FresnelPropFreqShift}. It states that the Fresnel propagator partly translates frequency-shifts into spatial shifts. By applying \eqref{eq:FresnelPropFreqShift} to the Gaussian-beam \eqref{eq:GaussianBeam}, we obtain an analytical formula for the propagation of Gaussian wave-packets:
\begin{subequations} \label{GaussWavePacketProp} 
 \begin{align}
 \cD\bparens{ h _{\bxi, \ba} } (\bx) =  \nu_{\bxi}\Parens{\bx - \ba - \bxi/\tNF}  p_{\tilde \sigma}(\bx - \ba - \bxi/\tNF)  \label{GaussWavePacketProp-a} \\
  \nu_{\bxi}\Parens{\bx} = \frac{ \tsigma^{m/2} c_0 }{ \sigma^{m/2} }  \exp  \Parens{ \I  \bbparens{ \frac{\bx^2 }{2 \eta^2} + \bxi \cdot \bx + \frac{\bxi^2}{2 \tNF} } }.  \label{GaussWavePacketProp-b} 
\end{align}
\end{subequations}
The oscillatory factor $\nu_{\bxi}$ has constant modulus. Hence, the envelope $\babs{\cD\bparens{ h _{\bxi, \ba} }} \propto p_{\tilde \sigma}((\cdot) - \ba - \bxi/\tNF)$ is again a Gaussian of width $\tsigma$, whose center is  \emph{shifted} by $\bxi/\tNF$ with respect to that of the original wave-packet $h_{\bxi, \ba}$. Accordingly, wave-packets propagate \emph{laterally} within the field-of-view upon action of the Fresnel-propagator.

\vspace{.5em}
\subsection{Resolution estimates via Gaussian wave-packets} \label{SS:GaussianResoEstimates}


We aim to use the analytical propagation formula \eqref{GaussWavePacketProp} for Gaussian wave-packets to derive \emph{upper}  bounds the achievable resolution in the reconstruction for \cref{IP:Complex} and \cref{IP:Real}. 
Since uniqueness always holds, see \cref{thm:Unique}, the only reasonable way to define resolution is via \emph{stability}: if we claim that the reconstruction has a resolution $1/r$, i.e.\ that features of the object down to a size $r > 0$ are faithfully recovered, then the reconstruction should be \emph{stable}  to perturbations of the object $h$ by any function $\tilde h $ that varies on lengthscales $r$, i.e.\ the induced contrast in the data
%
should be sufficiently large compared to $\norm{ \tilde h }$. By the hierarchy-of-difficulty of the considered inverse problems and linearity of $\cD$, a necessary condition for this to hold is that $\norm{ \cD( h + \tilde h)|_K - \cD( h)|_K }/\norm{ \tilde h }  =  \norm{ \cD(\tilde h)|_K }/\norm{ \tilde h }$ 
is non-negligible.  

Gaussian wave-packets $\tilde h = h_{\bxi, \bx_0}$ of frequency $|\bxi| \leq \pi/r$ constitute \emph{special} perturbations varying on lengthscales $\gtrsim r$. 
Thus, we can  derive \emph{upper}, i.e.\ possibly optimistic bounds on the achievable resolution $1/r$ by identifying parameter-regimes, for which $\norm{ \cD( h_{\bxi, \bx_0})|_K } / \norm{ h_{\bxi, \bx_0} }$ is negligibly small.


\vspace{.5em}
\subsubsection{Resolution for complex-valued images} \label{SSS:GaussianResoEstimatesComplex}

We study \cref{IP:Complex}(a) for a square detection-domain $K :=  [-\frac 1 2; \frac 1 2]^m$, $\Omega \subset K$.
In this setting, the unknown image $h \in L^2(\Omega)$ is complex-valued so that Gaussian wave-packets $\tilde h = h_{\bxi, \bx_0}$ of the form \eqref{GaussWavePacket}  centered at some point $\bx_0 \in \Omega$ constitute admissible perturbations\footnotemark[2].
 As seen from \eqref{GaussWavePacketProp}, the center of the Gaussian is then shifted to the point $\bx_{\textup{prop}} := \bx_0 + \bxi / \tNF $ upon Fresnel-propagation. Accordingly, if we consider wave-packets of larger and larger frequency $|\bxi|$, then the propagated wave-packet will eventually leave the detection domain, as visualized in \cref{fig:WavepacketPropSequenceComplex}.
More quantitatively, upon defining the path-length from a point $\bx $ to the detector-boundary $\partial K$  along a direction $\bn$,
\begin{align}
 \dist\!_{\bn} (\bx, \partial K) &= \inf \left\{y \geq 0: \bx + y \bn \notin K \right\} \MTEXT{for} \bx, \bn \in \mR^m: | \bn | = 1,
\end{align}
 the propagated center $\bx_{\textup{prop}}$ is inside  $K$ if and only if $\dist\!_{\bxi/|\bxi|} (\bx_0, \partial K) \leq |\bxi| / \tNF $. If $\bx_{\textup{prop}} \in K$, then the induced data-contrast is non-negligible: 
\begin{align} 
 \norm{\cD(h_{\bxi, \bx_0})|_K} \geq 2^{-m/2} \norm{ h_{\bxi, \bx_0} } \MTEXT{if} \bx_{\textup{prop}} \in K. \label{eq:WavePacketContrastEstimateInner}
\end{align}
On the contrary, if $\bx_{\textup{prop}} \notin K$ with distance $\dist(\bx_{\textup{prop}}, K) \gtrsim \tsigma $ greater than the propagated width $\tsigma$ of the wave-packet, then the contrast may be quite small:
\begin{align}
   \Norm{ \cD \parens{ h _{\bxi, \bx_0} }| _{K} }    \leq   \frac 1 2 \erfc \bbparens{ \frac{ \dist(\bx_{\textup{prop}}, K)   }{ \tsigma } }^{\frac 1 2}  \Norm{ h _{\bxi, \bx_0}  }.
  \label{eq:WavePacketContrastEstimateOuter}
\end{align}
As the complementary error function $\erfc(x)$ decays very fast for $x \gtrsim 1$, \eqref{eq:WavePacketContrastEstimateOuter} shows that the perturbation $h \mapsto h + h_{\bxi, \bx_0}$ is practically \emph{invisible} in the data $\cD(h + h_{\bxi, \bx_0})|_K$ if $|\bxi|$ is sufficiently large. In other words, oscillations at $\bx_0$ above a certain cutoff-frequency cannot be \emph{resolved}.

\begin{figure}[hbt!]
 \centering
 \includegraphics[width=\textwidth]{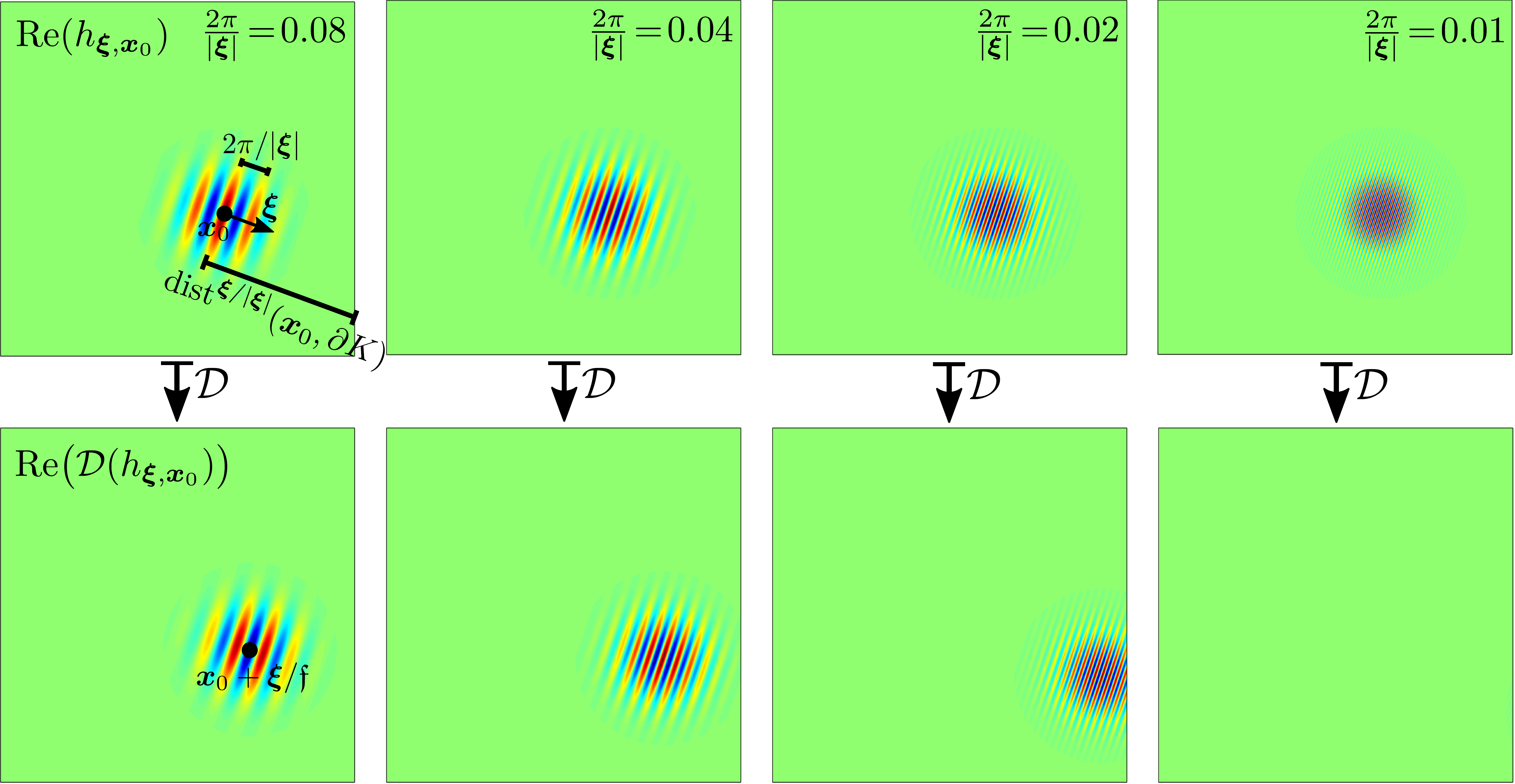}
 \caption{Propagation of a Gaussian wave-packet $h_{\bxi, \bx_0}$ for $m = 2$, $K = [-\frac 1 2; \frac 1 2]^2 $, $\tNF = 10^3$, $\sigma = 0.08$. Plotted are the real-parts of the complex-valued wave-packet (top row) and its propagated version $\cD(h_{\bxi, \bx_0})|_K$ (bottom row) computed via  \eqref{GaussWavePacketProp} for different frequencies $|\bxi|$. As $|\bxi|$ increases from left to right, the propagated wave-packet $\cD(h_{\bxi, \bx_0})$ is more and more  shifted with respect to $h_{\bxi, \bx_0}$ until it leaves the field of view $K$ (right-most column)  and is thus practically invisible to the considered imaging setup. The linear colorscale is identical in all images.} 
 \label{fig:WavepacketPropSequenceComplex}
\end{figure}

\footnotetext[2]{We ignore that the Gaussian wave-packet is technically not compactly supported and thus $h + h_{\bxi, \bx_0} \notin L^2(\Omega)$. Note, however, that $h_{\bxi, \bx_0}|_\Omega \approx h_{\bxi, \bx_0}$ up to a very small $L^2$-error given that $\bx_0$ is sufficiently far from the boundary of $\Omega$ in units of the Gaussian's width $\sigma$.}

The construction reveals that the local resolution $1/r(\bx_0)$ at a point $\bx_0$ is closely related to the distance to the detector-boundary $\dist  (\bx, \partial K)  = \min_{|\bn| = 1}  \dist\!_{\bn} (\bx, \partial K)$:
\vspace{.5em}
\begin{itemize}
 \item[$\boldsymbol 1$] For all Gaussian wave-packets $h_{\bxi, \bx_0}$ with $|\bxi| < \tNF \dist  (\bx_0, \partial K)$, the propagated center $\bx_{\textup{prop}}$ lies within the detection-domain $K$  \vspace{.25em}
 \item[$\boldsymbol 2$] For all frequencies $\xi > \tNF \dist  (\bx, \partial K)$, there exists a wave-packet $h_{\bxi, \bx}$ with $|\bxi| = \xi$, such that the propagated center $\bx_{\textup{prop}}$ lies outside $K$ 
\end{itemize}
\vspace{.5em}
As wave-packets leaving the field-of-view $K$ correspond to non-resolvable lengthscales, these observations translate into a resolution estimate:

\vspace{.5em}
\begin{result}[Resolution limit for complex-valued image reconstruction] \label{res:ComplexReso}
 For $K$ convex and $\Omega\subset K$, stable reconstruction in \cref{IP:Complex} can only be achieved down to a \emph{local resolution limit}
 \begin{align}
  1/ r(\bx) \lesssim  \frac {\tNF \dist(\bx, \partial K)} \pi \MTEXT{for all} \bx \in \Omega, \label{eq:GaussResoEstimateComplex}
 \end{align}
 where $r(\bx)$ denotes the smallest resolvable feature-size of the image $h$ at position $\bx$.
 
 In particular, for $K = \SquareDom$, the global maximum resolution is  bounded by the \emph{classical} Fresnel number  of the imaging setup (see \secref{SSS:MathFresnelNumber}): $\max_{\bx \in K} 1/ r(\bx) \lesssim  {\tNF}/{(2\pi)} = \F$. 
\end{result}
\vspace{.5em}

The resolution limit stated in \cref{res:ComplexReso} is \emph{isotropic} -- the resolution for features along a specific direction may be higher. 
\Cref{fig:WavepacketResolution}(a) shows the spatially varying resolution according to the estimate \eqref{eq:GaussResoEstimateComplex} for the exemplary setting $m = 2$, $K = [-\frac 1 2 ; \frac 1 2 ]^2$, $\tNF = 10^4$. \HIGHLIGHT{Note that the maximum resolution $\max_{\bx\in K} 1/r(\bx) = \F$ coincides with predictions according to Abbe's diffraction limit if the detector-size defines the numerical aperture, compare \cite{Nugent2010coherent,LatychenskaiaEtAl2012_HoloMeetsCDI(withHoloResolution)}. 
 Interestingly, however, the resolution only attains this optimum in the very center of the FoV as it decreases towards the detector-edges.}

\begin{figure}[hbt!]
 \centering
 \includegraphics[width=\textwidth]{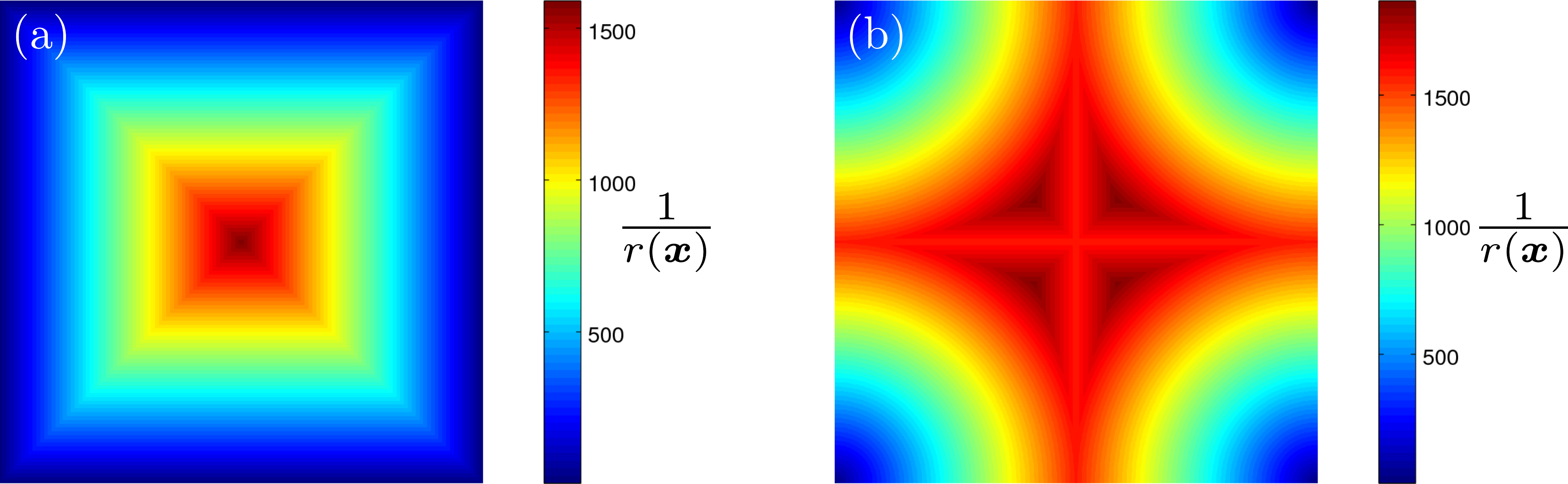}
 \caption{(a) Upper bound on the stably reconstructible local resolution $1/r(\bx)$ in \cref{IP:Complex} (complex-valued images) for $m = 2$, $K = [-\frac 1 2; \frac 1 2]^2 $, $\tNF = 10^4$ according to the estimate \eqref{eq:GaussResoEstimateComplex}. (b) Same plot for \cref{IP:Real}, i.e.\ for real-valued image-reconstruction, according to the estimate \eqref{eq:GaussResoEstimateReal}.}
 \label{fig:WavepacketResolution}
\end{figure}

\vspace{.5em}
\subsubsection{Resolution for real-valued images} \label{SSS:GaussianResoEstimatesReal}

In the case of \cref{IP:Real}(a), real-valued images are to be reconstructed so that complex-valued Gaussian wave-packets are no longer admissible perturbations. Accordingly, we study \emph{real-valued} wave-packets. Such signals are given by a superposition of  two Gaussian wave-packets with wavevectors $\bxi$ and $-\bxi$:
 \begin{align}
  h_{\bxi, \ba}^{\textup{real}}(\bx) &:=  \cos\bparens{ \bxi  \cdot (\bx - \ba) + \beta } p_\sigma(\bx - \ba) \nnl
  &= \Re \Parens{ \E^{\I \beta} h _{\bxi, \ba}(\bx)  } = \sfrac 1 2 \Parens{\E^{\I \beta} h _{\bxi, \ba}(\bx) +  \E^{-\I \beta}h _{-\bxi, \ba}(\bx)}
 \end{align}
 for $\bx, \bxi, \ba \in \mR^m, \beta \in [0; 2\pi)$.
Using \eqref{GaussWavePacketProp} and linearity of the Fresnel-propagator,  an analytical propagation formula is obtained for $h_{\bxi, \ba}^{\textup{real}}$:
\begin{align}
 \cD\bparens{h_{\bxi, \ba}^{\textup{real}}}(\bx) &= \frac{ \E^{\I \beta} } 2 \cD\Parens{h_{\bxi, \ba} }(\bx)   + \frac{ \E^{-\I \beta} } 2 \cD\Parens{h_{-\bxi, \ba} }(\bx) \nnl
 &= \frac{ \E^{\I \beta} } 2   \nu_{\bxi}\Parens{\bx - \ba - \bxi/\tNF}  p_{\tilde \sigma}(\bx - \ba - \bxi/\tNF) +  \frac{ \E^{-\I \beta} } 2   \nu_{-\bxi}\Parens{\bx - \ba + \bxi/\tNF}  p_{\tilde \sigma}(\bx - \ba + \bxi/\tNF) \label{eq:GaussWavePacketRealProp}
\end{align}

\begin{figure}[hbt!]
 \centering
 \includegraphics[width=\textwidth]{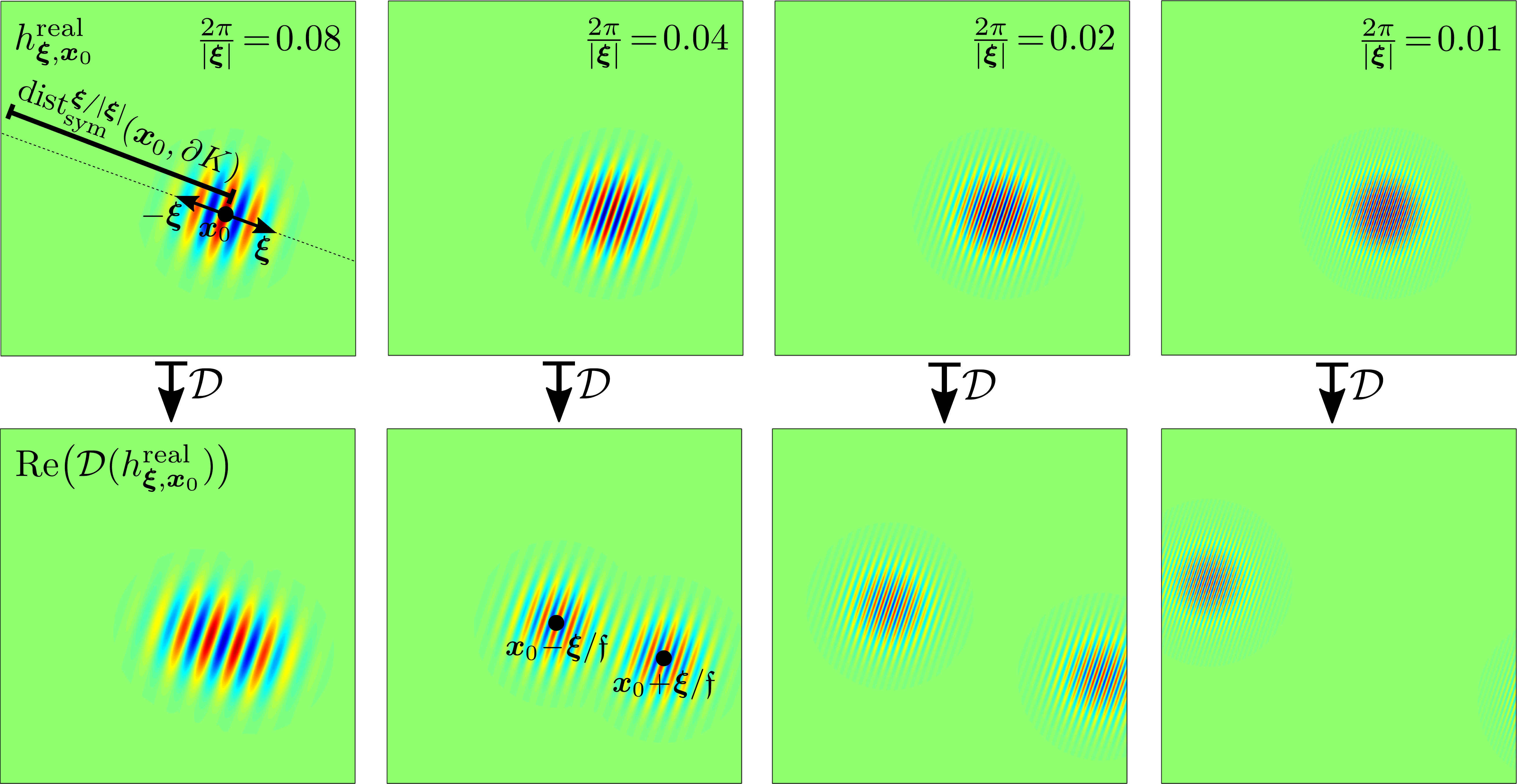}
 \caption{Analogue of \cref{fig:WavepacketPropSequenceComplex} for a real-valued wave-packet $h_{\bxi, \bx_0}^{\textup{real}}$. Upon Fresnel-propagation, such split up into two wave-packets that laterally propagate along opposite directions $\pm \bxi$. Consequently, the induced data contrast $\cD(h_{\bxi, \bx_0}^{\textup{real}})|_K$ is non-negligible under milder conditions than in the complex-valued case. For details, see text.}
 \label{fig:WavepacketPropSequenceReal}
\end{figure}

The analytical solution \eqref{eq:GaussWavePacketRealProp} reveals surprising features of the propagated signal:
upon propagation, the wave-packet splits up into two packets propagating into opposite directions $\pm \bxi$ as visualized in \cref{fig:WavepacketPropSequenceReal}. This has important consequences in terms of stability: if an object $h \in L^2(\Omega, \mR)$ is perturbed by a real-valued wave-packet $h_{\bxi, \bx}^{\textup{real}}$ at some point $\bx \in K$, then this perturbation manifests non-negligibly in the data $\cD(h + h_{\bxi, \bx}^{\textup{real}})|_K$ as long as \emph{either} of the two wave-packets remains within the field-of-view $K$.
For a point $\bx \in K$ and a direction $\bn$, we therefore introduce the following distance-measure:
\begin{align}
 \dist\! _{\textup{sym} }^{\bn} (\bx, \partial K) &= \inf \left\{y \geq 0: \bx + y \bn \notin K \textup{ and }  \bx - y \bn \notin K \right\} \\
 \dist\! _{\textup{sym} } (\bx, \partial K) &= \inf_{|\bn| = 1} \dist\! _{\textup{sym} }^{\bn}  (\bx, \partial K ).
\end{align}
$ \dist\! _{\textup{sym} }^{\bn} (\bx, \partial K)$ gives the larger length of the two line-segments $\{\bx \pm y \bn: y \geq 0 \} \cap K$, which connect $\bx$ with the boundary of $\partial K$ along $\bn$.
In view of wave-packets, the interpretation is simple: for $\bx, \bxi \in \mR^m$, the centers of \emph{both} propagating wave-packets forming $\cD\parens{h_{\bxi, \bx}^{\textup{real}}}$ lie outside $K$ if and only if $\dist\! _{\textup{sym} }^{\bxi/|\bxi|} (\bx, \partial K) < |\bxi|/\tNF$. Hence, the following relations hold true:
 \vspace{.5em}
\begin{itemize}
 \item[$\boldsymbol 1$] For all wave-packets $h_{\bxi, \bx}^{\textup{real}}$ with $|\bxi| < \tNF \dist\! _{\textup{sym} } (\bx, \partial K)$, the center of one of the propagating wave-packets lies within $K$. \vspace{.25em}
 \item[$\boldsymbol 2$] For all frequencies $\xi > \tNF \dist\! _{\textup{sym} } (\bx, \partial K)$, there exists a wave-packet $h_{\bxi, \bx}^{\textup{real}}$ with $|\bxi| = \xi$, such that the center of both  wave-packets  lie outside of $K$.  
\end{itemize}
 \vspace{.5em}
Accordingly, the quantity $\dist\! _{\textup{sym} } (\bx, \partial K)$ yields an upper bound for the local resolution in the real-valued setting:

\vspace{.5em}
\begin{result}[Resolution limit for real-valued image reconstruction]  \label{res:RealReso}
 For $K$ convex and $\Omega\subset K$, stable reconstruction in \cref{IP:Real} can only be achieved down to a \emph{local resolution limit}
 \begin{align}
  1/ r(\bx) \lesssim  \frac {\tNF \dist\! _{\textup{sym} }(\bx, \partial K)} \pi \MTEXT{for all} \bx \in \Omega, \label{eq:GaussResoEstimateReal}
 \end{align}
 where $r(\bx)$ denotes the smallest resolvable feature-size of the image $\varphi$ at position $\bx$. 
\end{result}
\vspace{.5em}
For $m = 2$, $K = [-\frac 1 2 ; \frac 1 2 ]^2$, $\dist\! _{\textup{sym} }$ can be evaluated analytically:
\begin{align}
 \dist\! _{\textup{sym} } \bparens{ (x_1,x_2), \partial K } = \min \Big\{ \Bparens{ \bparens{\sfrac 1 2 - |x_1|}^2 + \bparens{\sfrac 1 2 - |x_2|}^2 }^{\frac 1 2}, \min_{j \in \{1,2\}} \max_{   \pm}  \sfrac 1 2 \pm x_j \Big\}
\end{align}
The resulting spatially varying resolution for $\tNF = 10^4$ is plotted in \cref{fig:WavepacketResolution}(b). Notably, the maximum resolution is attained slightly off-center and is higher than in complex-valued case, compare \cref{fig:WavepacketResolution}(a). Moreover, a high resolution $1/r \geq 1000$ is obtained within a much larger subdomain of the field of view $K$. Most prominently, the resolution in \cref{fig:WavepacketResolution}(b) even remains large near the detector boundary -- except for the  corners of $K$.
Yet, the maximum resolution remains  essentially bounded by $\max_{\bx\in K} 1/ r(\bx) \lesssim \F$.

\vspace{.5em}
\section{Locality estimates for complex-valued objects} \label{S:BoundsComplex}

The goal of the subsequent sections is to complement the (potentially) \emph{optimistic} resolution estimates from \secref{S:GaussObservations} with \emph{worst-case} bounds. Accordingly, we aim to prove that stable image reconstruction can indeed be achieved down to a certain resolution. Note that this is necessarily more involved than the preceding analysis because stability has to be proven with respect to \emph{general} perturbations instead of considering just a special class like Gaussian wave-packets. 

\vspace{.5em}
\subsection{Basic idea and preliminaries} \label{SS:BoundsComplex-Prelim}

The principal difficulty in proving stability-estimates for bounded detection domains $K \subset \mR^m$ lies in the pronounced \emph{non-locality} of the Fresnel-propagator: according to \eqref{eq:FresnelPropConvForm}, it is given by a convolution with a kernel $\kF(\bx) \propto \exp(\I \tNF \bx / 2)$ that shows \emph{no spatial decay whatsoever!} Hence, Fresnel-propagation may transport object-information over arbitrary lateral distances in principle, i.e. features of the imaged object $h \in L^2(\Omega)$ with $\Omega \subset K$ may manifest far outside the field-of-view $K$ in the diffraction data $\cD(h)$. In addition to this non-locality in real-space, any restriction to $K \subsetneq \mR^m$ breaks the translational invariance of  $\cD$ and thus its diagonality, i.e.\ locality, in Fourier-space.

On the other hand, it has been seen in \secref{S:GaussObservations} that the distance, by which object-information is transported laterally, depends on the spatial frequencies of the signal. 
Accordingly, locality might be established by restricting to lower frequencies, i.e.\ to sufficiently smooth objects.

The principal idea of the subsequent analysis is to decompose the convolution kernel $\kF$ into an inner, local part, and an outer non-local part:
\begin{align}
 \kF = \kF | _{P} + \kF|_{\compl P} \MTEXT{for some} P \subset \mR^m.
\end{align}
For an object $h \in L^2(\Omega)$ supported in $\Omega \subset K \subset \mR^m$ and a suitably chosen $P$, the wave-field leaked outside $K$ depends only on the outer part: $\cD(h)|_{\compl K} = \parens{ \kF \ast h }|_{\compl K} = \parens{ \kF|_{\compl P} \ast h }|_{\compl K}$. This implies estimates of the form $\norm{\cD(h)|_{\compl K}} \leq \norm{\kF|_{\compl P} \ast h}$, which are diagonal in Fourier-space and thus simple to interpret as the norm of a \emph{filtered} object.

\paragraph{Notation: indicator functions}  For a set $ A \subset \mR^m$, let $\boldsymbol 1 _A: \mR^m \to \mR$ be defined by $\boldsymbol 1 _A(\bx) = 1$ if $\bx \in A$ and $\boldsymbol 1 _A(\bx) = 0$ otherwise.  


\vspace{.5em}
\subsection{Principal leakage estimates} \label{SS:BoundsComplex-ErrorBounds} 

Our principal leakage estimate is based on the insight that the frequency response of a restricted Fresnel-kernel $\kF | _{P}$ is readily computable:

\vspace{.5em}
\begin{lemma}[Frequency response of a restricted Fresnel-kernels]   \label{lem:LocalFresnel}
 Let $P \subset \mR^m$ be a measurable set such that  $\cD(  \boldsymbol 1 _{P} ) \in L^{\infty}(\mR^m)$ is well-defined and bounded. Let $\kF$ denote the convolution-kernel of the Fresnel-propagator. Then it holds for all $h \in L^2(\mR^m)$ that
      \begin{align}
      \cF \Parens{ \kF|_P \ast h }(\bxi) &=  \mF(\bxi) \cdot \cD(\boldsymbol 1_P) (\bxi/\tNF) \cdot  \cF \Parens{ h }(\bxi) \MTEXT{for almost all} \bxi \in \mR^m    \label{eq:lem-LocalFresnel-1}
      \end{align}
      and in particular for any measurable set $A \subset \mR^m$:
      \begin{align}
      \Norm{ (\kF|_P \ast h)|_A } &\leq   \Norm{\cD(\boldsymbol 1_P) (\cdot /\tNF) \cdot  \cF \Parens{ h }}.   \label{eq:lem-LocalFresnel-2} 
    \end{align} 
 \end{lemma}
\vspace{.5em}
\begin{proof}
 By the assumption $\cD(  \boldsymbol 1 _{P} ) \in L^{\infty}(\mR^m)$, both sides of the equation \eqref{eq:lem-LocalFresnel-1} are continuous in $h$ with respect to the $L^2$-norm. Hence, it is sufficient to prove the claim for Schwartz-functions $h \in \sS(\mR^m)$ by denseness of these in $L^2(\mR^m)$.
 
 For $h \in \sS(\mR^m)$, the convolution $\kF|_P \ast h $ is well-defined in a pointwise sense but can also be regarded as convolution between a Schwartz-function and a tempered distribution $\kF|_P \in \sS(\mR^m)'$. Accordingly, the convolution theorem holds, i.e.\
 \begin{align}
  \cF \Parens{ (\boldsymbol 1 _{P} \cdot \kF) \ast h } = (2\pi)^{-m/2} \cF(\boldsymbol 1 _{P} \cdot \kF) \cdot \cF(h) \label{eq:lem-LocalFresnel-pf1} 
 \end{align}
 in a distributional sense. Recalling that the alternate form of the Fresnel-propagator \eqref{eq:FresnelPropAltForm} remains valid for tempered distributions, we get
 \begin{align}
   \cF( \boldsymbol 1 _P \cdot \kF ) = u_0 \tNF^{\frac m 2 } \cF( \boldsymbol 1 _{P} \cdot \nF )  &= \frac 1 { \nF\Parens{\cdot/\tNF }}  \Parens{ u_0 \tNF^{\frac m 2 } \nF\Parens{\cdot/\tNF }  \cdot \cF( \boldsymbol 1 _P \cdot \nF ) \Parens{ \tNF \Parens{\cdot/\tNF } }  } \nnl
  &\stackrel{\eqref{eq:FresnelPropAltForm}}= \mF \cdot \cD(  \boldsymbol 1 _P ) (\cdot /\tNF )   \label{eq:lem-LocalFresnel-pf2} 
 \end{align}
 Inserting \eqref{eq:lem-LocalFresnel-pf2} into \eqref{eq:lem-LocalFresnel-pf1} yields \eqref{eq:lem-LocalFresnel-1}.
 The inequality \eqref{eq:lem-LocalFresnel-2} now follows by using unitarity of the Fourier transform along with the observations that $\mF$ has constant modulus 1 and that the restriction-operation $f \mapsto f|_A$ is non-increasing in the $L^2$-norm:
 \begin{align}
   \Norm{ (\kF|_P \ast h)|_A } &\leq  \Norm{ \kF|_P \ast h } = \Norm{\cF( \kF|_P \ast h ) } \stackrel{\eqref{eq:lem-LocalFresnel-1}}= \Norm{\mF \cdot \cD(\boldsymbol 1_P) (\cdot /\tNF) \cdot  \cF \Parens{ h }} \nnl
   &= \Norm{ \cD(\boldsymbol 1_P) (\cdot /\tNF) \cdot  \cF \Parens{ h }}. 
 \end{align}
\end{proof}
\vspace{.5em}

A surprising feature of \cref{lem:LocalFresnel} is that $\cD(\ldots)$ occurs as a factor in \emph{Fourier-space}. Similar as \cref{lem:FresnelPropFreqShift}, this reveals an interesting real-space-Fourier-space-duality of the Fresnel-propagator.
Using \cref{lem:LocalFresnel}, we may derive leakage estimates as outlined in \secref{SS:BoundsComplex-Prelim}:

\vspace{.5em}
\begin{theorem}[Principal leakage estimate] \label{thm:PrincipalLeakage}
 Let $K, \Omega, P_{\leak} \subset \mR^m$ be measurable sets such that the boundary $\partial K$ has Lebesgue-measure zero and $\Omega + P_{\leak} = \{ \bx + \by : \bx \in \Omega , \by \in P_{\leak} \} \subset K$. Moreover, let $\cD(  \boldsymbol 1 _{\compl P_{\leak}} ) \in L^{\infty}(\mR^m)$. 
 Then it holds for all $h \in L^2(\Omega)$
      \begin{align}
      \Norm{  \cD \Parens{ h }|_{\compl K} } \leq  \bnorm{ \hat p^{\leak} \cdot  \cF \Parens{ h } }, \qquad  p^{\leak}( \bxi ) := \Abs{ \cD(  \boldsymbol 1 _{\compl P_{\leak}} ) (\bxi /\tNF ) }.   \label{eq:thm-PrincipalLeakage-1} 
    \end{align}
 \end{theorem}
\vspace{.5em}
\begin{proof}
 By a similar continuity argument as in \cref{lem:LocalFresnel}
 it is sufficient to prove the claim for Schwartz-functions $h \in L^2(\Omega) \cap \sS(\mR^m)$. Then the convolution-form \eqref{eq:FresnelPropConvForm} of the  Fresnel-propagator may be used. Hence, we have
 \begin{align}
    \cD \Parens{ h }|_{\compl K} &=   \Parens{ \kF \ast h } |_{\compl K} =   \Parens{ \Parens{ \boldsymbol 1 _{P_{\leak}} \cdot \kF } \ast h } |_{\compl K} + \Parens{     \Parens{   \boldsymbol 1 _{\compl P_{\leak}} \cdot \kF } \ast h } |_{\compl K}    \label{eq:thm-PrincipalLeakage-pf1} 
 \end{align}
 According to standard results on the support of convolutions, it holds that	
\begin{align}
 \supp \Parens{ \Parens{ \boldsymbol 1 _{P_{\leak}} \cdot \kF } \ast h } &\subset \closure{ \supp \Parens{  \boldsymbol 1 _{P_{\leak}} \cdot \kF   } +  \supp \Parens{  h  } } \subset \closure{ P_{\leak} + \Omega } \subset \closure{ K }.  \label{eq:thm-PrincipalLeakage-pf2}
\end{align}
 \eqref{eq:thm-PrincipalLeakage-pf2} implies that $ \Parens{ \Parens{ \boldsymbol 1 _{P_{\leak}} \cdot \kF } \ast h } |_{\compl K}$ vanishes except for possibly the boundary $\partial K$. As $\partial K$ is a set of measure zero, $ \Parens{ \Parens{ \boldsymbol 1 _{P_{\leak}} \cdot \kF } \ast h } |_{\compl K} = 0$ holds in an $L^2$-sense. Thus, \eqref{eq:thm-PrincipalLeakage-pf1} yields
 \begin{align}
     \Norm{  \cD \Parens{ h }|_{\compl K}  }   &= \Norm{\Parens{  \Parens{   \boldsymbol 1 _{\compl P_{\leak}} \cdot \kF } \ast h } |_{\compl K}    }   \label{eq:thm-PrincipalLeakage-pf3}
 \end{align}
 Applying the bound \eqref{eq:lem-LocalFresnel-2}  to the right-hand side of \eqref{eq:thm-PrincipalLeakage-pf3} now yields the assertion.
\end{proof}
\vspace{.5em}

 \Cref{thm:PrincipalLeakage} bounds the leaked wave-field $\cD \parens{ h }|_{\compl K}$ in terms of a \emph{filtering-operation}. In order to predict in which cases leakage is small or large, we need to understand the nature of the filter-response $\hat p^{\leak}$ that weights the Fourier-components of $h$ in \eqref{eq:thm-PrincipalLeakage-1}. If  $\Omega \subset K \subset \mR^m$, then the largest admissible set $P$ in \cref{thm:PrincipalLeakage} is some bounded domain containing 0, where the exact size of $P$ depends on the distance between $\Omega$ and $\partial K$. Let us assume that the size of $P$ is much larger than $1/\tNF^{\frac 1 2}$ as will be the typical case in the following. Then the indicator-function $\boldsymbol 1_{\compl P} $ is essentially preserved  upon Fresnel-propagation, i.e. $\cD(\boldsymbol 1_{\compl P} ) \approx \boldsymbol 1_{\compl P} = 1- \boldsymbol 1_{P}$ up to some oscillations near the boundary of $P$. Accordingly, $\hat p^{\leak} = \cD(\boldsymbol 1_{\compl P} )(\cdot/\tNF)$ acts as a \emph{high-pass filter}, essentially damping out all Fourier-frequencies within the domain $\tNF \cdot \boldsymbol 1_{P}$. Thus, the right-hand side of \eqref{eq:thm-PrincipalLeakage-1} is small for sufficiently \emph{smooth} objects $h$.
 
%
%

\vspace{.5em}
\subsection{Explicit leakage bounds for rectangular domains} \label{SS:BoundsComplex-ExplicitFilters} 

For general sets $P_{\leak}$, $\hat p^\leak$ from \cref{thm:PrincipalLeakage} 
cannot be computed explicitly. An exception is given by rectangular domains owing to the known Fresnel-transform of the Heaviside-function \HIGHLIGHT{$\theta =  \boldsymbol 1 _{\mR_{\geq 0}}$} in terms of \emph{Fresnel-integrals} \cite{LieblingEtAl2003Fresnelets}:
\begin{align}
 \cD( \theta ) (x) 
 &= \frac 1 2 - \frac{1- \I} 2 \bbparens{ \textup{C} \bbparens{ \frac{ -\tNF^{\frac 1 2 } x }{ \pi^{\frac 1 2}}  } + \I \textup{S} \bbparens{ \frac{ -\tNF^{\frac 1 2 } x }{ \pi^{\frac 1 2}}  } }  =: \tilde \theta( \tNF^{\frac 1 2 } x ) \nnl
 \textup{C}(x) &:= \int_{0}^{x} \cos\Parens{\frac \pi 2 t^2} \, \D t, \quad \textup{S}(x) := \int_{0}^{x} \sin\Parens{\frac \pi 2 t^2} \, \D t  \MTEXT{for all} x \in \mR   \label{eq:FresnelPropEdge}
\end{align}
Note that $\sF$ is an entire analytic function and bounded with $\max_{x\in \mR} |\sF(x)| \leq 1.171$. 

By the separability- and isotropy-properties of the Fresnel-propagator, \eqref{eq:FresnelPropSeparable} and \eqref{eq:FresnelPropIsotropy}, the explicit solution generalizes to half-spaces $H_{a, \bn} := \{ \bx \in \mR^m: \bn \cdot \bx \geq a \}$ in arbitrary dimensions with $a \in \mR $ and $\bn \in \mS^{m-1} = \{\bx\in \mR^m: |\bn| = 1\} $:
\begin{align}
 \cD( \boldsymbol 1 _{H_{a, \bn}} ) (\bx) =  \sF \bparens{ \tNF^{\frac 1 2 } ( \bn \cdot \bx - a) } \MTEXT{for all} \bx \in \mR^m. \label{eq:FresnelPropHalfspace}
\end{align}
Using linearity of $\cD$, \eqref{eq:FresnelPropEdge} furthermore yields the Fresnel-transform of intervals:
\begin{align}
 \cD( \boldsymbol 1 _{[-\Delta; \Delta]} ) (x) &=  \cD( \boldsymbol 1 _{[\Delta; \infty)} ) (x) -  \cD( \boldsymbol 1 _{[-\Delta; \infty)} ) (x) = \sF \Parens{ \tNF^{\frac 1 2 } ( x - \Delta) } - \sF \bparens{ \tNF^{\frac 1 2 } ( x + \Delta) }  \nnl
 &=  \sF \bparens{  \tNF^{\frac 1 2 }  x -  \tNF_\Delta^{\frac 1 2 }  }- \sF \bparens{  \tNF^{\frac 1 2 }  x +  \tNF_\Delta^{\frac 1 2 }  } =: \iFd (\tNF^{\frac 1 2} x) \label{eq:FresnelPropInt}
\end{align}
Here, we introduced the Fresnel number associated with the lateral lengthscale $\Delta > 0$, $\tNFa {\Delta} := \Delta^2 \tNF$, compare \secref{SSS:MathFresnelNumber}.
Again by the separability of the Fresnel-propagator, this generalizes to stripe-shaped domains $S_{\Delta, \bn} := \{ \bx \in \mR^m: -\Delta \leq \bn \cdot \bx \leq \Delta \}$ and squares:
\begin{align}
\cD( S_{\Delta, \bn} ) (\bx) &=  \iFd \Parens{ \tNF^{\frac 1 2} (\bn\cdot \bx) } =: \iFda {\bn} (\bx), \qquad \iFda j :=  \iFda {\be_j}  \label{eq:FresnelPropStripe} \\
\cD( \boldsymbol 1 _{[-\Delta; \Delta]^m} ) (\bx) &=  \prod_{j = 1}^m \cD( \boldsymbol 1 _{[-\Delta; \Delta]} ) (x_j) = \prod_{j = 1}^m \iFd (\tNF^{\frac 1 2} x_j)  = \prod_{j = 1}^m \iFda j (\tNF^{\frac 1 2} \bx )\label{eq:FresnelPropSquare}
\end{align}
for all $\bx = (x_1,\ldots, x_m) \in \mR^m $, where $\be_j\in \mR^m$ denotes the $j$-th unit normal vector.
Finally, indicator functions of \emph{complements} are simple to propagate using linearity and $\cD(1) = 1$:
\begin{align}
  \cD(\boldsymbol 1_{\compl A} ) = \cD(1 - \boldsymbol 1_{ A} ) = \cD(1)  -  \cD(\boldsymbol 1_{ A} ) = 1 -  \cD(\boldsymbol 1_{ A} ) \MTEXT{for} A \subset \mR^m. \label{eq:FresnelPropCompl}
\end{align}
Using the formulas derived above, we can explicitly write down the filter from \cref{thm:PrincipalLeakage} for the important special case of square domains:

\vspace{.5em}
\begin{cor}[Leakage bound for square domains]   \label{cor:LeakageSquare}
Let $ K = \SquareDom $ and $\Omega = \SquareDomMin \Delta$ for some $0 < \Delta < \frac 1 2$. 
Then it holds for all $h \in L^2(\Omega)$
\begin{align}
\Norm{  \cD \Parens{ h }|_{\compl K} } \leq  \bnorm{\hat p^{\leak}_{\Box, \tNF, \tNFa {\Delta}} \cdot \cF(h)}, \qquad  \hat p^{\leak}_{\Box, \tNF, \tNFa {\Delta}}(\bxi) &:= \bbabs{ 1 - \prod_{j = 1}^m \iFda j ( \bxi /  \tNF^{\frac 1 2} ) }, \quad \bxi \in \mR^m \label{eq:cor-LeakageSquare-1}
\end{align}
 \end{cor}
\begin{proof}
 If we set $P_{\leak} := [-\Delta; \Delta]^m$, the assumptions of \cref{thm:PrincipalLeakage} are satisfied.  
 The expression for the filter in \eqref{eq:cor-LeakageSquare-1} follows by using \eqref{eq:FresnelPropSquare} along with \eqref{eq:FresnelPropCompl}.
\end{proof}
\begin{figure}[hbt!]
 \centering
 \includegraphics[width=\textwidth]{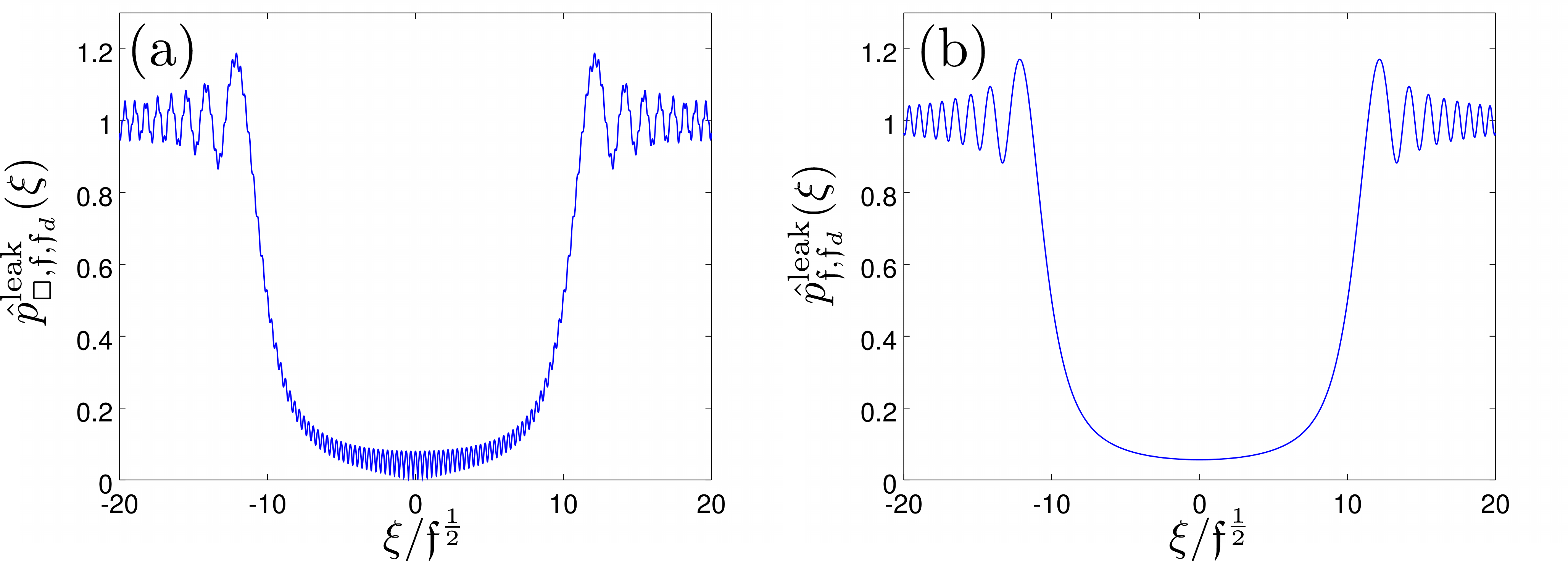}
 \caption{Plot of the leakage-filters for square-domains from (a) \cref{cor:LeakageSquare} and (b) \cref{thm:LeakageSquareSimple} for $m = 1$ dimensions and $\tNFa {\Delta} = 100$.}
 \label{fig:Filters}
\end{figure}

The filter-response $\hat p^{\leak}_{\Box, \tNF, \tNFa {\Delta}}$ from \cref{cor:LeakageSquare} is plotted in \cref{fig:Filters}(a) for an exemplary 1D-setting. It can be seen to be a \emph{high-pass} filter of width $\approx 2 (\tNF \cdot \tNFa {\Delta})^{1/2} = 2 \tNF \Delta$  in Fourier-space. Note that this width is in perfect agreement with the expected cut-off frequency from the Gaussian wave-packet analysis in \secref{SSS:GaussianResoEstimatesComplex} for the considered distance $\dist(\Omega, \partial K) = \Delta$ between object-domain and detector-boundary. However, $\hat p^{\leak}_{\Box, \tNF, \tNFa {\Delta}}$ is heavily oscillatory on fine scales and is \emph{not} everywhere $\leq 1$, although this would be reasonable by unitarity of the Fresnel-propagator. Another drawback of the filter-response  in \eqref{eq:cor-LeakageSquare-1} is that it varies 
in a non-trivial manner in higher dimensions. 

Both the oscillatory behavior and the complicated high-dimensional structure can be resolved by exploiting the simple rectangular geometry to obtain an alternative filter: 

\vspace{.5em}
\begin{theorem}[Leakage bound for square domains with simplified filter]   \label{thm:LeakageSquareSimple}
Within the setting of \cref{cor:LeakageSquare}, it holds for all $h \in L^2(\Omega)$
\begin{align}
  \Norm{  \cD \Parens{ h }|_{\compl K} }  &\leq  \bnorm{\hat p^{\leak}_{\tNF, \tNFa {\Delta}} \cdot \cF(h)}  , \qquad   \hat p^{\leak}_{\tNF, \tNFa {\Delta}}(\bxi)  := \bbparens{ \sum_{  j = 1}^m  \Abs{ \eFd \Parens{  \tNF^{-\frac 1 2} \be_j \cdot \bxi }  }^2 }^{\frac 1 2} \label{eq:thm-LeakageSquareSimple-1} \\
 \eFd(x) &:= \bparens{ \babs{\sF(x -\tNFa {\Delta}^{\frac 1 2})}^2  + \babs{\sF(-x -\tNFa {\Delta}^{\frac 1 2})}^2  }^{\frac 1 2 }  \label{eq:thm-LeakageSquareSimple-2}  
\end{align}
where $\be_j$ denotes the unit normal vector along the $j$-th dimension.
 \end{theorem}
\vspace{.5em}
\begin{proof}
 Let $h \in L^2(\Omega)$.
  If we define the half-spaces $H_{j,\pm} :=  \mR^{j-1} \times \pm[-\frac 1 2 ; \infty) \times \mR^{m-j}$, then it holds that $\compl K = \bigcup_{j = 1, \pm }^m \compl{H_{j,\pm}}$. Thus,
  \begin{align}
   \Norm{  \cD \Parens{ h }|_{\compl K} } ^2 \leq \sum_{\pm, j = 1}^m  \Norm{ \cD(h)|_{H_{j,\pm}} }^2  \label{eq:thm-LeakageSquareSimple-pf1}
  \end{align}
  Upon setting $\tilde K := H_{j,\pm}$ and $P_{\leak} := \mR^{j-1} \times \pm [-\Delta ; \infty) \times \mR^{m-j}$,  \cref{thm:PrincipalLeakage} is applicable so that each of the squared norms on the right-hand side \eqref{eq:thm-LeakageSquareSimple-pf1} can be bounded via \eqref{eq:thm-PrincipalLeakage-1}:
  \begin{align}
     \Norm{ \cD(h)|_{H_{j,\pm}} }^2 &\leq \bip{\cF(h)}{ |\hat p^{\leak}_{\tNF, \tNFa {\Delta}, j, \pm}|^2 \cdot \cF(h)} \nnl
     \hat p^{\leak}_{\tNF, \tNFa {\Delta}, j, \pm}(\bxi) &= \cD( \boldsymbol 1_{\compl{P_{\leak}}}) (\bxi/\tNF) \stackrel{\eqref{eq:FresnelPropHalfspace}}=  \sF \bparens{ \pm \tNF^{-\frac 1 2} \be_j \cdot \bxi - \tNFa {\Delta}^{\frac 1 2} }. \label{eq:thm-LeakageSquareSimple-pf2}
  \end{align}
  Substituting \eqref{eq:thm-LeakageSquareSimple-pf2} into \eqref{eq:thm-LeakageSquareSimple-pf1} and using sesqui-linearity yields the assertion.
\end{proof}
\vspace{.5em}

\Cref{fig:Filters}(b) plots the alternate filter-response $\hat p^{\leak}_{\tNF, \tNFa {\Delta}}$ for the same 1D-setting as in \cref{fig:Filters}(a). The plot reveals that the filter-profiles are almost identical except that the oscillations are eliminated from the low-frequency regime. This makes it easier to derive bounds for $|\hat p^{\leak}_{\tNF, \tNFa {\Delta}}|$ in a given frequency-interval compared to the original filter $\hat p^{\leak}_{\Box, \tNF, \tNFa {\Delta}}$.

\vspace{.5em}
\subsection{Stability estimates}

By unitarity of the Fresnel-propagator $\cD$, \emph{upper} bounds on the wave-field leaked into $\compl K$ induce \emph{lower} bounds on the contrast  {within} the field-of-view $K$, 
i.e.\ stability estimates for the reconstruction of $h$ from  data $\cD(h)|_K$:

\vspace{.5em}
\begin{cor}[Stability estimates for square domains] \label{cor:StabilitySquare}
Let $ K = \SquareDom$ and $\Omega = \SquareDomMin \Delta$ for some $0 < \Delta < \frac 1 2$. 
Then it holds for all $h \in L^2(\Omega)$
\begin{align}
   \Norm{  \cD \Parens{ h }|_{ K} }^2 \geq  \bip{\cF(h)}{ \bparens{ 1 - \babs{ \hat p^{\leak} }^2 }  \cdot \cF(h)} \label{eq:cor-StabilitySquare-1}
\end{align}
with  $\hat p^{\leak} \in \{ \hat p^{\leak}_{\tNF, \tNFa {\Delta}}, \hat p^{\leak}_{\Box, \tNF, \tNFa {\Delta}} \}$ 
as defined in \cref{thm:LeakageSquareSimple,cor:LeakageSquare}.
\end{cor}
\vspace{.5em}
\begin{proof}
 The claim follows from \cref{thm:LeakageSquareSimple,cor:LeakageSquare} as $  \Norm{  \cD \Parens{ h }|_{ K} }^2 =   \Norm{  \cD \Parens{ h }  }^2 -   \Norm{  \cD \Parens{ h }|_{ \compl K} }^2$ and $\Norm{  \cD \Parens{ h }  }^2 = \Ip{  \cF \Parens{ h }  }{  \cF \Parens{ h }  }$ as $\cD$ and $\cF$ are unitary.
\end{proof}
\vspace{.5em}

 \Cref{cor:StabilitySquare} gives lower- and upper bounds on the contrast on a square detector in terms of \emph{filtering} operations with explicitly known profile in Fourier-space. 
 It is tempting to interpret the bound as the norm of a low-pass-filtered version of $h$:
 \begin{align}
   \bip{\cF(h)}{ \bparens{ 1 - \babs{ \hat p^{\leak} }^2 }  \cdot \cF(h)} \textup{``=''} \bnorm{ \bparens{ 1 - \babs{ \hat p^{\leak} }^2 }^{\frac 1 2}  \cdot \cF(h)}^2
 \end{align}
 However, this is technically not correct because $|\hat p^{\leak}(\bxi)|$ typically attains values greater than 1 at frequencies above the cut-off $|\bxi| \geq \tNF \Delta$, see \cref{fig:Filters}. This means that the bound \eqref{eq:cor-StabilitySquare-1} indeed permits \emph{negative} contrast in certain Fourier-frequencies. While this is certainly counter-intuitive from a physical point-of-view, one has to cope with this peculiarity in order to make sense of the stability estimates.
 
 Since $\abs{ \hat p^{\leak} }$ is typically much smaller than 1 for low frequencies (compare \cref{fig:Filters}), the right-hand side of \eqref{eq:cor-StabilitySquare-1} will be positive for objects $h$ whose Fourier-transform $\cF(h)$ is sufficiently localized at low frequencies. Accordingly, a natural candidate for a class of functions that can be stably recovered from $\cD \Parens{ h }|_{ K}$ would be \emph{band-limited} ones, such that $\cF(h)$ vanishes above a certain maximum frequency, including all parts of the Fourier domain where $1 - \abs{ \hat p^{\leak} }^2$ is negative. Importantly, however, \cref{cor:StabilitySquare} also assumes $h$ to have compact support in real-space 
 so that $\cF(h)$ is an entire function and thus cannot vanish in any open set $U \subset \mR^m$ unless $\cF(h)$, and hence $h$, is identically zero.
 
 \HIGHLIGHT{In general, we see that determining a stable class of objects naturally involves the classical problem of finding functions that are well-localized in real-space and Fourier-space at the same time, governed by so-called \emph{uncertainty principles}. See e.g.\ \cite{Slepian1983FourierCompactSpectrum,Folland1997Uncertainty} for reviews on this topic. 
 As a solution,} we will restrict to objects given by B-splines, which may have a compact support \emph{and} will be shown to be \emph{essentially} band-limited in a suitable quantitative manner. 

\vspace{.5em}
\section{Stability estimates for spline objects} \label{S:Splines}

In the following, we derive stability results for objects given by multi-variate B-splines, which can be regarded as  images of \emph{finite resolution}.
Such a restriction also makes sense from an experimentalist's point-of-view as the finite number of detector-pixels introduce a natural discretization in any real-world \HIGHLIGHT{XPCI} setup.

\vspace{.5em}
\subsection{Multi-variate B-splines} \label{SS:SplinesBasicNotions}

As a model for discretized, i.e.\ \emph{pixelated} images, we consider spaces of $j$-th order multi-variate B-splines: for a fixed \emph{resolution} $1/r$ with $r > 0$ and origin $\bo \in [0;1)^{m}$, we arrange nodes on a uniform Cartesian grid in $\mG_{r, \bo} ^m := \left\{ \bo + r \bj  : \bj \in \mZ^m \right\} \subset \mR^m$:
Now we define objects as linear combinations of basis-splines centered at these nodes:
\begin{subequations}   \label{eq:DefSplineSpace}
\begin{align}
 \BSpace &:= \bigg\{h: \bx \mapsto \sum_{\bj \in \mZ^m } b_{\bj} B_{k}^m(\bx/r - \bj - \bo ): (b_{\bj}) \in \ell^2(\mZ^m)  \bigg\}  \label{eq:DefSplineSpace-a} \\
 B_{k}^m(x_1, \ldots, x_m) &:=   \prod_{j = 1}^m  B_{k} (x_j ), \quad  B_{k}  = \begin{cases} B_{0}  \ast B_{k-1}  &\textup{for } k \in \mN  \\
  \boldsymbol 1 _{[-\frac 1 2 ; \frac 1 2 )} &\textup{for } k = 0   \end{cases}  \label{eq:DefSplineSpace-b}
\end{align}
\end{subequations} 
For details and explicit formulas of B-splines, see for example \cite{UnserEtAl1991FastBSplineTransform,UnserEtAl1992_SplinesToGaussConvergence}. For our purposes here it is sufficient to note that $\supp(B_{k}^m) = [-\frac{ k+1} 2; \frac{ k+1} 2]^m$ and $B_k^m \in \sC^{k-1}(\mR^m)$ for $k \geq 1$. 

\vspace{.5em}
 \subsubsection{Approximation properties} \label{SSS:SplinesApprox} Splines interpolate values assigned on the grid nodes: for any sequence  $(y_{\bj}) \in \ell^2(\mZ^m)$ and $k \in \mN_0$, there exists a unique spline $h \in \BSpace$ such that $h(r\bj + \bo) = y_{\bj}$ for all  $\bj \in \mZ^m$ and the map $(y_{\bj}) \mapsto h$ is continuous from $\ell^2(\mZ^m)$ to $L^2(\mR^m)$.
This is related to the fact that B-splines form a \emph{Riesz sequence} \cite{Christensen2010bFramesOfBSplines}:
\begin{align}
 r^{m/2} \CRiesz{k}^m \Norm{(b_{\bj} )}_{\ell^2(\mZ^m)} \leq \Norm{h} \leq r^{m/2}  \Norm{(b_{\bj} )}_{\ell^2(\mZ^m)}
 \label{eq:SplineRieszSeq}
\end{align}
for some constants $\CRiesz{k} > 0$ and all $h =  \sum_{\bj \in \mZ^m } b_{\bj} B_{k}^m(\cdot/r - \bj - \bo ) \in \BSpace$. The Riesz-sequence-property ensures {stability} of the approximation of functions by B-splines.

\vspace{.5em}
 \subsubsection{Separability} \label{SSS:SplinesSeparable} According to their definition in \eqref{eq:DefSplineSpace}, B-splines exhibit a separable structure: for any $h \in \BSpace$, $1 \leq j \leq m$ and $\bx_{<j} \in \mR^{j-1},\bx_{>j} \in \mR^{m-j}$ fixed, it holds that $h_{(\bx_{<j}, \bx_{>j})}: x_j \mapsto h(\bx_{<j}, x_j, \bx_{>j}) \in \BSpacem 1$. In other words, multi-variate B-splines are one-dimensional B-splines along each coordinate dimension.

 \vspace{.5em}
\subsection{Quasi-band-limitation of B-splines} \label{SS:SplinesBandlimit}


Our interest in B-splines is mainly due to their property of being quasi band-limited. As the following estimate of this quasi-band-limitation is slightly off-topic and lengthy to derive, its proof is given in \cref{App:SplineQuasiBandlimit}.

\vspace{.5em}
\begin{theorem}[Quasi-band-limitation of univariate B-splines] \label{thm:SplineQuasiBandlimit}
 Let $k \in \mN_0$, $r > 0$, $\Xi_ r := [-\frac \pi r ; \frac \pi r]$ and $\nu  \geq 1$. Then it holds that 
 \begin{align}
  \Norm{ \cF(h) |_{\compl{(\nu \Xi_ r)}}}  &\leq \CBand(k, \nu)\Norm{ \cF(h) } \MTEXT{for all} h \in \BSpacem 1, \label{eq:thm-SplineQuasiBandlimit-1}
  \end{align}
  where the constant $\CBand(k, \nu) < 1$ is given by
  \begin{align}
  \CBand(k, \nu)^2 &= \frac{c_\band(k,\nu)}{1 + c_\band(k,\nu) }, \quad c_\band(k,\nu) = c_{\band,0}(k,\nu) + \! \sum_{n =  \ceil{(\nu-1)/2}}^\infty \frac 2 { \Parens{ 2n + 1 }^{2(k+1)}} \nonumber \\
  c_{\band,0}(k,\nu) &= \max \bigg\{  \frac{\max\{\tilde \nu, 0\}^{ 2(k+1)}}{ \Parens{ \nu + 2\tilde \nu } ^{ 2(k+1)} }+   \frac{\max\{\tilde \nu, 0\}^{ 2(k+1)}}{   \nu   ^{ 2(k+1)} } - \frac 1 {\bar \nu ^{ 2(k+1)}} , \; 0  \bigg\} \label{eq:thm-SplineQuasiBandlimit-2}
 \end{align}
 where $\ceil{\cdot}$ is the ``round up''-operation, $\bar \nu := 1+ 2 \ceil{(\nu-1)/2}$ and $\tilde \nu := \bar \nu - \nu - 1$.
 
 Conversely, for any $\nu < 1$, there exists an $h \in \BSpacem 1$ such that  $ \cF(h) |_{\compl{(\nu \Xi_ r)}} = \cF(h)$, i.e. no estimate of the form \eqref{eq:thm-SplineQuasiBandlimit-1} can hold true for any constant $\CBand(k, \nu) < 1$.
\end{theorem}
 
 \begin{figure}[hbt!]
 \centering
 \includegraphics[width=.55\textwidth]{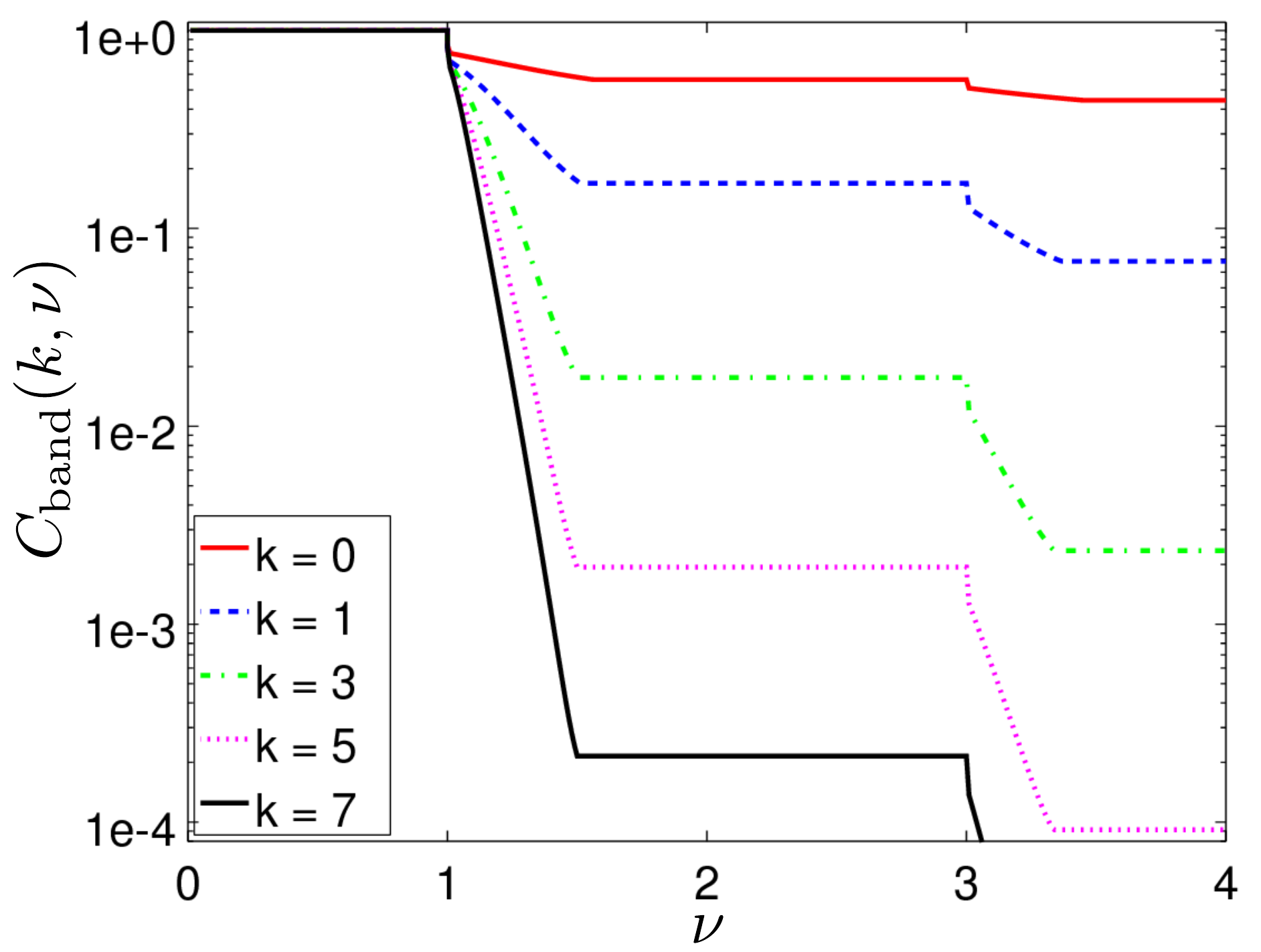} 
 \caption{Semi-logarithmic plot of the quasi-band-limitation constant $\CBand(k, \nu)$ from \cref{thm:SplineQuasiBandlimit} for different spline-orders $k$.}
 \label{fig:CBand}
\end{figure}

 The constant $\CBand(k, \nu)$ in \eqref{eq:thm-SplineQuasiBandlimit-1} may be readily evaluated by computing the infinite series in \eqref{eq:thm-SplineQuasiBandlimit-2} via known analytical formulas. In \cref{fig:CBand}, $\CBand(k, \nu)$ is plotted against $\nu$ for different spline-orders $k = 0,1,3,5,7$. It can be seen that the bound $\CBand(k, \nu)$ drops \emph{discontinuously} from 1 to $\approx 2^{-1/2}$ at $\nu = 1$ and then decreases exponentially until $\nu \approx 1.5$, where the decrease is sharper for higher spline-orders $k$. For $\nu \in [1.5;3]$, the value of  $\CBand(k, \nu)$ stagnates before it continues to decrease within the interval $[3;3.5]$ and so on.
 
 By exploiting the separable structure of B-splines discussed in \secref{SSS:SplinesSeparable}, the 1D-result in \cref{thm:SplineQuasiBandlimit} may be easily generalized to higher dimensions:
\vspace{.5em}
\begin{theorem}[Quasi-band-limitation of multivariate B-splines] \label{thm:SplineQuasiBandlimitND}
 Let $k \in \mN_0$, $r > 0$, $\nu  \geq 1$, $\Xi_ r := [-\frac \pi r ; \frac \pi r]^m$ and $\Xi_{r,j} := \mR^{j-1} \times [-\frac \pi r ; \frac \pi r] \times \mR^{m-j}$. Then it holds that 
 \begin{align}
  \bnorm{ \cF(h) |_{\compl{(\nu \Xi_{r,j})}}}  &\leq \CBand  (k, \nu)\Norm{ \cF(h) } \qquad \qquad \qquad \quad \MTEXT{for all} 1 \leq j \leq m, \label{eq:thm-SplineQuasiBandlimitND-1} \\
  \bnorm{ \cF(h) |_{\compl{(\nu \Xi_ r)}}}  &\leq \Parens{ 1 - \Parens{1 - \CBand(k, \nu)^2 }^m }^{\frac 1 2} \Norm{ \cF(h) } \MTEXT{for all} h \in \BSpace \label{eq:thm-SplineQuasiBandlimitND-2}
  \end{align}
\end{theorem}
\vspace{.5em}
\vspace{.5em}

 \vspace{.5em}
 \subsection{Stability estimates} \label{SS:SplinesStability}
 \HIGHLIGHT{In the language of regularization theory, the transition to finitely sampled B-spline objects corresponds to imposing a (very strong)  \emph{source condition}. Similarly as proven in \cite{AlessandriniEtAl2005LipschitzStabHelmholtzScat,BerettaEtAl2016LipschitzStabHelmholtzScat} for other severely ill-posed problems, such a ``finite-resolution'' source condition enables Lipschitz-stability estimates for image-reconstruction from truncated Fresnel-data. This is seen by combining the quasi-band-limitation results from \secref{SS:SplinesBandlimit} with the leakage estimates from \secref{SS:BoundsComplex-ExplicitFilters}:} 

\vspace{.5em}
\begin{theorem}[Stability estimate for spline-objects] \label{thm:StabSpline}
Let $ K = \SquareDom$ and $\Omega = \SquareDomMin \Delta$ for $0 < \Delta < \frac 1 2$. Let $\tNFa {\Delta} := \Delta^2\tNF$ and $\tNFa r := r^2\tNF$ denote the Fresnel numbers associated with the length-scales $\Delta$ and $r$, respectively (compare \secref{SSS:MathFresnelNumber}). Furthermore, let $\nu \geq 1$ and $\Xi := [-\nu \pi/\tNFa r^{1/2}; \nu \pi/\tNFa r^{1/2}]$. Then it holds that 
 \begin{align}
  \Norm{ \cD(h) | _{K} }  &\geq \CStab (\tNFa {\Delta}, \tNFa r, k, \nu )^m \norm{h } \FA h \in \BSpace \cap L^2(\Omega). \label{eq:thm-StabSpline-1}
 \end{align}
 With $\eFd$ as defined in \cref{thm:LeakageSquareSimple}, the constant is given by 
  \begin{align}
    \CStab (\tNFa {\Delta}, \tNFa r, k, \nu )^2  &= 1 -  C_{\low}^2 - \CBand  (k, \nu )^2 \Parens{ C_{\tot}^2 - C_{\low}^2  }  \label{eq:thm-StabSpline-2}  \\
    C_{\low}  &:= \max_{x \in \Xi} \eFd(x), \qquad C_{\tot} := \max_{x \in \mR}   \eFd(x)  \nonumber
  \end{align}
  \end{theorem}
\vspace{.5em}
 \begin{proof}
  We first prove the claim for $m = 1$, i.e.\ let $h \in \BSpacem 1 \cap L^2(\Omega)$ for $K = [-\frac 1 2; \frac 1 2]$ and $\Omega = [-\frac 1 2 + \Delta; \frac 1 2- \Delta]$.
  By \cref{thm:LeakageSquareSimple}, it then holds that
  \begin{align}
   \Norm{  \cD \Parens{ h }|_{\compl K} }^2 &\leq  \bip{\cF(h)}{ |\hat p^{\leak}_{\tNF, \tNFa {\Delta}}|^2  \cdot \cF(h)}, \qquad \hat p^{\leak}_{\tNF, \tNFa {\Delta}}(\xi) = \eFd (\xi/\tNF^{\frac 1 2} ), \quad \xi \in \mR\label{eq:thm-StabSpline-pf1}
\end{align}
 Let $\Xi_r := (\tNF^{\frac 1 2} / \nu) \cdot \Xi = [-\pi/r; \pi/r]$. Then we have by definition of the constants $C_\low, C_\tot$
 \begin{align}
  \max_{\xi \in (\nu\Xi_r)} | \hat p^{\leak}_{\tNF, \tNFa {\Delta}}(\xi) | &= \max_{\xi \in (\nu\Xi_r)} \eFd (\xi/\tNF^{\frac 1 2} ) = \max_{x \in \Xi} \eFd (x ) = C_{\low}  \label{eq:thm-StabSpline-pf2} \\
  \max_{\xi \in \mR} | \hat p^{\leak}_{\tNF, \tNFa {\Delta}}(\xi) | &= \max_{x \in \mR} \eFd (x ) = C_{\tot}  \label{eq:thm-StabSpline-pf3}  
 \end{align}
  By combining these bounds with the estimate \eqref{eq:thm-StabSpline-pf1}, we obtain
  \begin{align}
   \bnorm{\cD(h)|_{\compl K}}^2 &- C_{\low}^2 \norm{\cF(h)}^2 \leq \bip{\cF(h)}{ \underbrace{ \bparens{  | p^{\leak}_{\tNF, \tNFa {\Delta}} |^2 -  C_{\low}^2} }_{ \leq 0 \textup{ in } \nu \Xi_r } \cdot \cF(h)} \nnl
   &\leq  \bip{\cF(h)|_{\compl{(\nu\Xi_r)}}}{  \bparens{ | p^{\leak}_{\tNF, \tNFa {\Delta}} | ^2  -  C_{\low}^2}   \cdot \cF(h)|_{\compl{(\nu\Xi_r)}}} \leq   \Parens{ C_{\tot}^2 - C_{\low}^2 } \bnorm{  \cF(h)|_{\compl{(\nu\Xi_r)}}}^2 \label{eq:lem-StabSplineStripe-pf1}
  \end{align}
 Since $h \in \BSpacem 1$ and $\nu \geq 1$, $\bnorm{  \cF(h)|_{\compl{(\nu\Xi_r)}}}$ can be bounded via the quasi-band-limitation \cref{thm:SplineQuasiBandlimit}: $\bnorm{  \cF(h)|_{\compl{(\nu\Xi_r)}}} \leq \CBand (k, \nu) \norm{  \cF(h) }$.
  Inserting this bound into \eqref{eq:lem-StabSplineStripe-pf1} yields 
  \begin{align}
     \Norm{\cD(h)|_{  K}}^2 &= \Norm{\cD(h)}^2 - \bnorm{\cD(h)|_{\compl K}}^2 = \Norm{\cF(h)}^2 - \bnorm{\cD(h)|_{\compl K}}^2 \nnl 
     &= \Norm{\cF(h)}^2  - C_{\low}^2 \norm{\cF(h)}^2 - \bparens{ \bnorm{\cD(h)|_{\compl K}}^2 - C_{\low}^2 \norm{\cF(h)}^2 }  \nnl
     &\geq  \Parens{ 1 - C_{\low}^2 - \CBand (k, \nu)^2 \bparens{ C_{\tot}^2 - C_{\low}^2  } } \Norm{\cF(h)}^2 =  \CStab(\tNFa {\Delta} , \tNFa r, k, \nu)^2  \norm{h }^2. \label{eq:lem-StabSplineStripe-pf3}
  \end{align}

 \vspace{.5em}
\paragraph{Extension to $m > 1$:} The result may be generalized to higher dimensions by exploiting the separability of the Fresnel-propagator \eqref{eq:FresnelPropSeparable}, of multi-variate B-splines and of the domains $\Omega$ and $K$: if we set $K_j := \mR^{j-1} \times [-\frac 1 2; \frac 1 2] \times \mR^{m-j}$ and $\Omega_j :=  \mR^{j-1} \times [-\frac 1 2 + \Delta; \frac 1 2 - \Delta] \times \mR^{m-j}$ and factorize the Fresnel propagator $\cD = \cD_m \ldots \cD_1$, then we have $K = \bigcap_{j=1}^m K_j$, $\Omega = \bigcap_{j=1}^m \Omega_j$ and the restriction to $K_j$ commutes with $\cD_i$ for any $i\neq j$. Thus, 
\begin{align}
 \cD(h) |_K &= \bparens{  \ldots \bparens{ \cD_m \ldots \cD_1(h)|_{K_1}}|_{K_2} \ldots }|_{K_m} = \cD_m\bparens{ \ldots \cD_2\bparens{ \cD_1(h)|_{K_1} }|_{K_2} \ldots }|_{K_m} \nnl
 &=   \cD_m\bparens{ h_{m} }|_{K_m} \MTEXT{with} h_j := \cD_{j-1}\bparens{ \ldots \cD_2\bparens{ \cD_1(h)|_{K_1} }|_{K_2} \ldots }|_{K_{j-1}}.
\end{align}
 Moreover, as the operators $T_j: f \mapsto \cD_{j}(f)|_{K_j}$ act only along the $j$-th coordinate dimension and since $h \in \BSpace$ with $\supp(h) \subset \Omega$, it holds that $\supp(h_j) \subset \Omega_j$ for all $1 \leq j \leq m$ and $h_j$ is a 1D B-spline when restricted to the $j$-th coordinate dimension (compare \secref{SSS:SplinesSeparable}). This implies that we may bound expressions of the form $\cD_j\bparens{ h_{j} }_{K_j}$ using the bound for $m = 1$ dimensions derived above. By recursive application of this argument, we arrive at
 \begin{align}
  \Norm{ \cD(h) |_K } &= \bnorm{ \cD_m\bparens{ h_{m} }|_{K_m}  }  \geq  \CStab (\tNFa {\Delta}, \tNFa r, k, \nu ) \Norm{ h_m  } \nnl
  &= \CStab (\tNFa {\Delta}, \tNFa r, k, \nu ) \bnorm{ \cD_{m-1}\bparens{ h_{m-1} }|_{K_{m-1}}  } \geq \!\ldots \! \geq \CStab (\tNFa {\Delta}, \tNFa r, k, \nu )^m    \Norm{ h }.  
 \end{align}
 \end{proof}

\vspace{.5em}
\subsection{Application: resolution estimates} \label{SS:SplinesResolutionEstimates}

The stability estimate in \cref{thm:StabSpline} can be used to verify that an imaging setup allows for a certain resolution at a realistic noise level within the setting of inverse problem \cref{IP:Complex}(a).
We can address to types of questions:
\vspace{.5em}
\begin{itemize}
 \item[$\boldsymbol 1$] For a fixed (spline-)resolution $1/r$, how stable is the reconstruction within a square object-domain $\Omega$ depending on its distance $\Delta$ to the detector boundary $\partial K$? \vspace{.25em}
 \item[$\boldsymbol 2$] If we require a stability estimate $\norm{\cD(h)|_K} \geq C \norm{h}$ with some minimal contrast $C \in (0;1)$, what resolution $1/r$ can be achieved depending on $d$?
\end{itemize}
\vspace{.5em}
We illustrate this for an exemplary setting in $m = 2$ dimensions with square detector $K = [-\frac 1 2; \frac1 2]^2$, Fresnel number $\tNF =  10^4$ and splines of order $k = 7$. 

For setting $\boldsymbol 1$, let us examine how stably features of size $r = 1/500$ 
can be reconstructed. We compute values of the stability-constant $\CStab(\tNFa {\Delta}, \tNFa r, k, \nu)$ for different $0 < \Delta < \frac 1 2$, $\tNFa {\Delta} = \Delta^2\tNF$ and a suitable $\nu$ (here, we choose $\nu = 1.2$ fixed but note that, in principle, one could optimize over this parameter as the bound \eqref{eq:thm-StabSpline-1} holds for \emph{all} $\nu \geq 1$). For any point $\bx \in K$, we can then express the \emph{local} stability of the reconstruction at a point $\bx$ as
\begin{align}
 c_{\stab, r}(\bx) &:= \sup \left\{C_\stab  ( \tNFa {\Delta}, r ^2 \tNF , k, \nu) : \bx \in [-\sfrac 1 2 + \Delta; \sfrac 1 2- \Delta]^m \right\} \nnl
 &= \CStab \bparens{ \dist(\bx, \partial K)^2 \tNF, r^2 \tNF , k, \nu } \qquad (\partial K\textup{: detector-boundary}) \label{eq:CStabLocal}
\end{align}
The resulting values of $c_{\stab, r}(\bx)$ are plotted in \cref{fig:StabilityResolutionComplex}(a). It can be seen that $c_{\stab, r}(\bx) = 0$, indicating instability, holds true up to  $\dist(\bx, \partial K) \gtrsim \pi / (\tNF r)$ and then increases very quickly to a value close $1$ for larger distances to the detector-boundary.  These results are in very good agreement with the resolution estimates from the analysis of Gaussian wave-packets in \secref{S:GaussObservations}.
 
 \begin{figure}[hbt!]
 \centering
 \includegraphics[width=\textwidth]{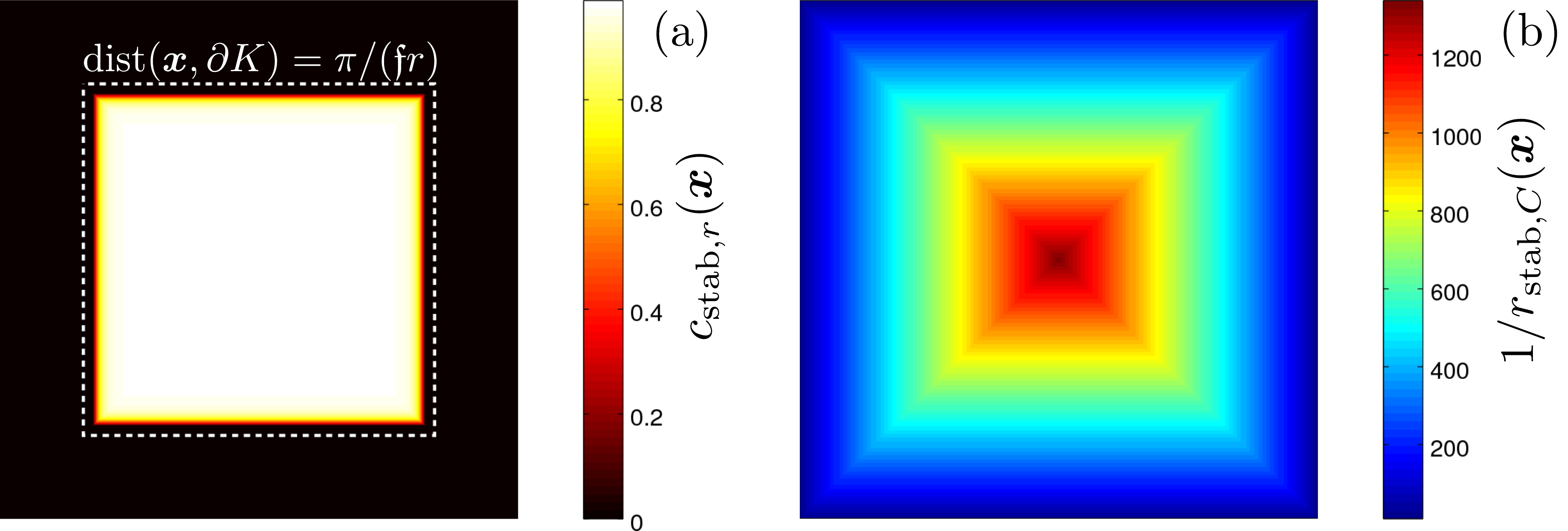} 
 \caption{(a) Local stability of the reconstruction in \cref{IP:Complex}(a) for a square FoV $K = [-\frac 1 2; \frac 1 2]^2$ and $\tNF = 10^4$, according to \eqref{eq:CStabLocal} and \cref{thm:StabSpline} for $k = 7$, $\nu = 1.2$. The dashed line bounds the region that is expected to be stable according to the wave-packet-analysis in \secref{SSS:GaussianResoEstimatesComplex}. (b) Plot of the stably achievable local resolution computed via \eqref{eq:RStabLocal} for $C = 1/4$.}
 \label{fig:StabilityResolutionComplex}
\end{figure}

For problems of the type $\boldsymbol 2$, we can use \eqref{eq:CStabLocal} 
to express the stably achievable resolution:
\begin{align}
 r_{\stab, C}(\bx) := \inf \left\{r > 0 : c_{\stab, r}(\bx) \geq C \right\} \label{eq:RStabLocal}
\end{align}
Numerically computed values of $1/r_{\stab, C}(\bx)$ for $C = 1/4$ are plotted in \cref{fig:StabilityResolutionComplex}(b). The plot turns out to be practically identical to \ref{fig:WavepacketResolution}(a), up to slightly lower resolutions by a global factor of about $1.2$. In other words, the \emph{worst-case} resolution estimates of the present section are very close to the possibly optimistic bounds derived in \secref{SSS:GaussianResoEstimatesComplex}.

\vspace{.5em}
\section{Improved estimates for real-valued objects} \label{S:Real}

\vspace{.5em}
\subsection{Quasi-symmetric propagation principle} \label{SS:RealSymmProp}

In the preceding sections, we have derived locality- and stability estimates for Fresnel-propagation in terms of essentially two ingredients: smoothness, i.e.\ a finite resolution, and distance to the detector-boundary. Moreover, as both best-case- and worst-case-stability has been considered, these ingredients have been shown to be both \emph{necessary and sufficient}!
In \secref{SSS:GaussianResoEstimatesReal}, however, it has been found that the reconstruction of  \emph{real-valued} images  is subject to much less severe resolution limits, based on the observation that real-valued Gaussian wave-packets propagate \emph{symmetrically} upon Fresnel propagation.

Clearly, the observed behavior of wave-packets could  be just a peculiarity of the considered, very special class of functions. Yet, \emph{quasi}-symmetric propagation of real-valued signals turns out to be a general principle, that is closely related to the characteristic symmetry-properties of their Fourier transforms: for any $\varphi: \mR^m \to \mR$, $\cF(\varphi)$ is a \emph{Hermitian} function, i.e.\ it holds that $\cF(\varphi)(-\bxi) = \cc{\cF(\varphi)(\bxi)}$ for all $\bxi \in \mR^m$. We use this property via the following lemma:

\vspace{.5em}
 \begin{lemma} \label{lem:FiltersRealSym}
 Let $\varphi \in L^2(\mR^m,\mR)$ be real-valued and $\hat p \in L^\infty( \mR^m )$. Then it holds that
 \begin{align}
  \Norm{ \hat p \cdot \cF(\varphi)} = \Norm{ \sym\!\Parens{ \hat p } \cdot \cF(\varphi)},
 \end{align}
 where $\sym \Parens{ \hat p }$ is   defined by $\sym \parens{ \hat p }(\bxi) := 2^{- 1/2 } \parens{ | \hat p (\bxi) |^2 + |  \hat p (-\bxi) |^2  } ^{1/2}$ for all $\bxi \in \mR^m$.
\end{lemma}
\vspace{.5em}
\begin{proof}
 As $\varphi$ is real-valued, $|\cF(\varphi)(-\bxi)| = \babs{ \cc{\cF(\varphi)(\bxi)} } = \abs{ \cF(\varphi)(\bxi) } $ for all $\bxi \in \mR^m$. Thus,
\begin{align}
  2 \norm{ \hat p \cdot \cF \parens{ \varphi }  }^2 &= 2 \int_{\mR^m} |\hat p(\bxi)|^2 | \cF(\varphi)(\bxi)|^2  \, \D \bxi =  \int_{\mR^m} |\hat p(\bxi)|^2 \Parens{ | \cF(\varphi)(\bxi)|^2 + | \cF(\varphi)(-\bxi)|^2 }  \, \D \bxi \nnl
  &=  \int_{\mR^m} |\hat p(\bxi)|^2 | \cF(\varphi)(\bxi)|^2 \, \D \bxi  +  { \int_{\mR^m} |\hat p(\bxi)|^2 | \cF(\varphi)(-\bxi)|^2  \, \D \bxi } \nnl 
  &=  \int_{\mR^m} \bparens{ |\hat p(\bxi)|^2 + |\hat p(-\bxi)|^2 }  | \cF(\varphi)(\bxi)|^2 \, \D \bxi = 2 \int_{\mR^m}  |\sym (\hat p)(\bxi)|^2   | \cF(\varphi)(\bxi)|^2 \, \D \bxi \nnl
  &= 2 \norm{ \sym(\hat p) \cdot \cF \parens{ \varphi }  }^2.  \label{eq:lem-FiltersRealSym-pf1}  
\end{align}
\end{proof}
\vspace{.5em}

Despite its simplicity, \cref{lem:FiltersRealSym} enables us to prove a surprisingly general result on the propagation of real-valued signals:

\vspace{.5em}
 \begin{theorem}[Quasi-symmetric propagation of real-valued signals] \label{thm:SymmetryBound}
For $a \in \mR$ and $\bn \in \mS^{m-1}$, let $K := \{ \bx \in \mR^m: \bn \cdot \bx \geq a \} \subset \mR^m$ be a half-space. Then it holds that
\begin{align}
 \norm{\cD(\varphi)|_{\compl{K}}} \leq \Csym  \norm{\varphi} \MTEXT{for all \emph{real-valued}} \varphi \in L^2(K, \mR) \label{eq:thm-SymmetryBound-1}
\end{align}
with a universal constant $\Csym < 1$, independent of $\tNF$, $m$, $a$, $\bn$ and $\varphi$, that is bounded by
\begin{align}
 \Csym \leq \max_{x \in \mR} \, \sym( \sF ) (x)  \leq 0.837. \label{eq:thm-SymmetryBound-2}
\end{align}

On the contrary, for general, \emph{complex-valued} signals $ \varphi \in L^2(K, \mC)$, no bound of the form \eqref{eq:thm-SymmetryBound-1} may hold true for any $\Csym < 1$.
\end{theorem}
\vspace{.5em}
\begin{proof}
 Let $\varphi \in L^2(K, \mR)$ be real-valued.
 If we set $P_{\leak} := \{ \bx \in \mR^m: \bn \cdot \bx \geq 0 \}$, \cref{thm:PrincipalLeakage} is applicable and we obtain by \eqref{eq:thm-PrincipalLeakage-1}
   \begin{align}
      \Norm{  \cD \Parens{ \varphi }|_{\compl{K}} } \leq   \bnorm{ \hat p^{\leak}  \cdot  \cF \Parens{ \varphi } }  \label{eq:thm-SymmetryBound-pf1} 
   \end{align}
   where the filter is given by $\hat p^{\leak}  (\bxi) = \cD(  \boldsymbol 1 _{\compl P_{\leak}} ) (\bxi /\tNF ) = \sF(-\bn\cdot \bxi / \tNF^{1/2} )$ by the analytical propagation-formula  \eqref{eq:FresnelPropHalfspace} for half-spaces. By \cref{lem:FiltersRealSym}, it follows that
   \begin{align}
    \bnorm{ \hat p^{\leak}  \cdot  \cF \Parens{ \varphi } } &= \bnorm{ \sym( \hat p^{\leak}  )  \cdot  \cF \Parens{ \varphi } } \leq  \Bparens{ \max_{\bxi\in \mR^m} \, | \sym( \hat p^{\leak}  ) (\bxi) | } \cdot \norm{ \cF \Parens{ \varphi } }  \label{eq:thm-SymmetryBound-pf2}
   \end{align}
   The result can be further simplified by using that $\hat p^{\leak} $ varies only along the axis $\bn$:
   \begin{align}
  \max_{\bxi\in \mR^m} \, | \sym( \hat p^{\leak}  ) (\bxi) |  &= 2^{-\frac 1 2 } \max_{\bxi\in \mR^m} \bparens{  |\sF(-\bn\cdot \bxi / \tNF^{1/2} )|^2 + |\sF(\bn\cdot \bxi / \tNF^{1/2} )|^2  }^{\frac 1 2} \nnl
    &= 2^{-\frac 1 2 }\max_{x \in \mR} \bparens{  |\sF(-x)|^2 + |\sF(x)|^2  }^{\frac 1 2} =   \max_{x \in \mR} \, \sym( \sF ) (x).    \label{eq:thm-SymmetryBound-pf3}
   \end{align}
   Combining \eqref{eq:thm-SymmetryBound-pf1}, \eqref{eq:thm-SymmetryBound-pf2} and \eqref{eq:thm-SymmetryBound-pf3}  yields the first assertion.

Now let us drop the assumption of real-valuedness, i.e.\ let $0 \neq h \in L^2(K(, \mC))$ be arbitrary. By \cref{lem:FresnelPropFreqShift}, the propagated intensity $|\cD(h)|^2$ may be shifted in arbitrary directions and arbitrarily far by replacing $h$ with $\tilde h : \bx \mapsto \exp(\I \bxi \cdot \bx) h(\bx)$ for a suitable $\bxi \in \mR^m$, while one still has $\tilde h \in  L^2(K)$ with $\norm{\tilde h} = \norm{h}$. By this shifting-mechanism, one may thus construct $\tilde h$  for which $|\cD(\tilde h)|^2$ is arbitrarily concentrated in $\compl K$, i.e.\ $\norm{\cD(\tilde h)|_{\compl K}} / \norm{\tilde h}$ may be arbitrarily close to 1. Hence, no non-trivial bound of the form \eqref{eq:thm-SymmetryBound-1} may hold for complex-valued signals.
\end{proof}
\vspace{.5em}

\Cref{thm:SymmetryBound} states that -- independent of any smoothness constraints (!) -- only a limited fraction of a real-valued signal may propagate out of its support in a single direction. As is also stated in the theorem, this situation is unique to the real-valued case. We note that the analytical estimate for the constant $\Csym$ is not optimal:

\vspace{.5em}
\begin{rem}[Optimal value of $\Csym$] \label{rem:Csym}
 Numerical eigenvalue computations (not shown) indicate that the optimal value of the symmetric-propagation constant is given by $\Csym \approx 0.721$. Accordingly, at most a fraction of $0.721^2 \approx 0.52$ of the intensity of a real-valued signal may leak out of the field-of-view along a single direction.
\end{rem}
\vspace{.5em}

\vspace{.5em}
\subsection{Construction of improved leakage bounds}  \label{SS:RealLeakageBounds}

Next, we  extend the quasi-symmetric propagation bound in \cref{thm:SymmetryBound} from half-spaces to the more practically relevant case of square FoVs. In such a setting, the propagated signal may always leak out of the detection domain along two opposite directions so that (quasi-)symmetric propagation alone may not guarantee finite leakage.
Instead, we have to combine the latter principle with the detector-distance-based leakage estimates of the preceding sections.
The idea is simple: along each direction, we can decompose an object-signal into a part with support close to the detector-boundary $\partial K$, to be bounded by exploiting quasi-symmetric propagation, and another part that is concentrated far away from $\partial K$ and which thus can be bounded using the theory from \secref{S:BoundsComplex}. We first prove such a bound for half-spaces:

\vspace{.5em}

\begin{lemma} \label{lem:SymmetryBound} 
 For $\bn \in \mS^{m-1}$, $a \in \mR$ and $\Delta > 0$, let $K := \{ \bx \in \mR^m: \bn \cdot \bx \geq a \}$, $\Omega = K$ and $\Omega_{\leq \Delta} := \Omega \cap \{ \bx \in \mR^m:  \bn \cdot \bx \leq a +\Delta  \}$. Then it holds that
\begin{align}
 \bnorm{\cD(\varphi)|_{\compl{K}}} &\leq 2^{-\frac 1 2} \bnorm{  \hat p^{\leak}_{\tNF, \tNFa {\Delta}}  \bparens{ \bn\cdot(\cdot) }  \cdot \cF(\varphi) } + \Csyma{\tNFa {\Delta}} \bnorm{ \varphi|_{\Omega_{\leq \Delta} } }  \FA \varphi \in L^2(\Omega, \mR)   \label{eq:lem-SymmetryBound-1} 
\end{align}
with $\tNFa {\Delta} = \Delta^2 \tNF$ and $\hat p^{\leak}_{\tNF, \tNFa {\Delta}}$ as in \cref{thm:LeakageSquareSimple} and the constant $\Csyma{\tNFa {\Delta}}$ is given by
\begin{align}
  \Csyma{\tNFa {\Delta}} &:= \max_{x \in \mR} \, \sym\parens{\sF_{  \tNFa {\Delta}} }(x), \quad \sF_{  \tNFa {\Delta}} (x) := \sF(x) - \sF( x- \tNFa {\Delta} ^{\frac 1 2 }).   \label{eq:lem-SymmetryBound-3} 
\end{align}
\end{lemma}
\vspace{.5em}
\begin{proof}
 By separability \eqref{eq:FresnelPropSeparable} and isotropy \eqref{eq:FresnelPropIsotropy}, it is sufficient to prove the claim for the 1D-setting $m = 1$, $\bn = 1$, $a = 0$, $K = \Omega = [0; \infty)$ and $\Omega_{\leq \Delta} = [0; \Delta]$. 
 
 Thus, let $\varphi \in L^2(K,\mR)$ be arbitrary.
We follow a similar approach as in \cref{thm:PrincipalLeakage}:
using the convolution-form of the Fresnel-propagator \eqref{eq:FresnelPropConvForm}, we obtain
 \begin{align}
  \cD(\varphi)|_{\compl K} &= \bparens{ \kF \ast \varphi }|_{\compl K}  = \bparens{ \bparens{ \kF|_{(-\infty; -\Delta)} + \kF|_{[-\Delta;0]} + \kF|_{K}  } \ast \varphi }|_{\compl K}  \nnl
  &= \bparens{\kF|_{(-\infty; -\Delta)} \ast \varphi }|_{\compl K} + \bparens{\kF|_{[-\Delta;0]} \ast \varphi }|_{\compl K} + \bparens{\kF|_{K  }  \ast \varphi }|_{\compl K}  \label{eq:lem-SymmetryBound-pf2}
 \end{align}
  Now we decompose $\varphi$ into left-hand- and right-hand parts, $\varphi = \varphi_\L + \varphi_\R$ with $\varphi_\L:=\varphi|_{\Omega_{\leq \Delta}}$, $\varphi_\R:=\varphi|_{\compl {\Omega_{\leq \Delta}}}$.
  By standard results on the support of convolutions, we then have
 \begin{align}
  \supp(\kF|_{K} \ast \varphi) &\subset \supp(\kF|_{K}) + \supp(\varphi) \subset K + K = K \nnl
  \supp(\kF|_{[-\Delta;0]} \ast \varphi_\R) &\subset \supp(\kF|_{[-\Delta;0]}) + \supp(\varphi_\R) \subset [-\Delta;0] + [\Delta; \infty) = K \nonumber
 \end{align}
 Together with \eqref{eq:lem-SymmetryBound-pf2}, this implies that $ \cD(\varphi)|_{\compl K} = \bparens{\kF|_{(-\infty;-\Delta)} \ast \varphi }|_{\compl K} + \bparens{\kF|_{[-\Delta; 0]} \ast \varphi_\L }|_{\compl K} $.
  An application of the triangle inequality and \cref{lem:LocalFresnel} thus yields
   \begin{align}
  \Norm{ \cD(\varphi)|_{\compl K} } &\leq  \Norm{ \bparens{ \kF|_{(-\infty;-\Delta)} \ast \varphi }|_{\compl K} } + \Norm{ \bparens{\kF|_{[-\Delta; 0]} \ast \varphi_\L }|_{\compl K} } \nnl
  &\leq  \Norm{    \cD( \boldsymbol 1_ {(-\infty;-\Delta)} )(\cdot/\tNF) \cdot \cF \parens{ \varphi }  } + \Norm{ \cD( \boldsymbol 1_ {[-\Delta; 0]} )(\cdot/\tNF) \cdot \cF \parens{ \varphi_\L }  } \label{eq:lem-SymmetryBound-pf4}
 \end{align}
Using the exact propagation-formulas from \secref{SS:BoundsComplex-ExplicitFilters}, we get
\begin{align}
 \cD( \boldsymbol 1_ {(-\infty;-\Delta)} )(\xi/\tNF) &= \sF \bparens{ - \xi  / \tNF^{\frac 1 2} - \tNFa {\Delta}^{\frac 1 2} }  \nnl 
 \cD( \boldsymbol 1_ {[-\Delta; 0]} )(\xi/\tNF) &=  \cD( \boldsymbol 1_ {(-\infty;0)} )(\xi/\tNF) - \cD( \boldsymbol 1_ {(-\infty;-\Delta)} )(\xi/\tNF)  \nnl
 &= \sF \bparens{ - \xi  / \tNF^{\frac 1 2}  } - \sF \bparens{ - \xi  / \tNF^{\frac 1 2} - \tNFa {\Delta} ^{\frac 1 2} }  = \sF_{\tNFa {\Delta}} (- \xi / \tNF^{\frac 1 2})  \label{eq:lem-SymmetryBound-pf5}
\end{align}
Moreover, since $\varphi$ and thus $\varphi_\L$ are real-valued, \cref{lem:FiltersRealSym} is applicable. Thus,
\begin{align}
  \Norm{    \cD( \boldsymbol 1_ {(-\infty;-\Delta)} )(\cdot/\tNF) \cdot \cF \parens{ \varphi }  } &= \bnorm{  \sym \bparens{ \sF \bparens{ (- \cdot)  / \tNF^{\frac 1 2} - \tNFa {\Delta}^{\frac 1 2} } } \cdot \cF \parens{ \varphi }  }   \stackrel{\eqref{eq:thm-LeakageSquareSimple-1}}= 2^{-\frac 1 2} \Norm{   \hat p^{\leak}_{\tNF, \tNFa {\Delta}}  \cdot \cF(\varphi) } \nnl
   \Norm{ \cD( \boldsymbol 1_ {[-\Delta; 0]} )(\cdot/\tNF) \cdot \cF \parens{ \varphi_\L }  }  &=  \bnorm{ \sF_{\tNFa {\Delta}} (-\cdot /\tNF^{\frac 1 2})\cdot \cF \parens{ \varphi_\L }  } = \bnorm{ \sym( \sF_{\tNFa {\Delta}}  )(\cdot /\tNF^{\frac 1 2})\cdot \cF \parens{ \varphi_\L }  } \nnl
   &\leq \Bparens{ \max_{x\in \mR} \sym( \sF_{\tNFa {\Delta}}  ) (x) } \Norm{ \varphi_\L } = \Csyma {\tNFa {\Delta}} \Norm{ \varphi_\L }  \label{eq:lem-SymmetryBound-pf6}
\end{align}
Substituting \eqref{eq:lem-SymmetryBound-pf6} into \eqref{eq:lem-SymmetryBound-pf4} now yields the assertion.
\end{proof}
\vspace{.5em}

 Note that the constant $\Csyma{\tNFa {\Delta}}$ in \eqref{eq:lem-SymmetryBound-1}  attains almost the same values as $\Csym$ within the relevant regime $\tNFa {\Delta} \gg 1$.
 Next, we extend \cref{lem:SymmetryBound} to square detectors $K = \SquareDom$ by decomposing $\compl K$ into half-spaces. By far the strongest result is obtained for a 1D-case:

 \vspace{.5em}
\begin{theorem}[Leakage estimate for real-valued objects in 1D-intervals] \label{thm:LeakageRealInterval}
 Let $m = 1$, and $\Omega = K = [-\frac 1 2; \frac 1 2]$. Then it holds that
 \begin{align}
  \Norm{\cD(\varphi)|_{\compl K}} \! &\leq \bnorm{  \hat p^{\leak}_{\tNF, \tNF/4} \cdot \cF(\varphi) }\! + \Csyma{\tNF/4}\norm{\varphi} \MTEXT{for all} \varphi \in L^2(\Omega, \mR).  \label{eq:thm-LeakageRealInterval-1}
 \end{align}
\end{theorem}
\vspace{.5em}
\begin{proof}
Let $\varphi \in L^2(\Omega, \mR)$ be arbitrary.
 We decompose $\compl K$ into a left-hand and a right-hand part: $\compl K = L \cup R$ with $L = (-\infty; -\frac 1 2)$, $R = -L$. Then it holds that
 \begin{align}
  \bnorm{\cD(\varphi)|_{\compl K}}^2 = \bnorm{\cD(\varphi)|_{L}}^2 + \bnorm{\cD(\varphi)|_{R}}^2 \label{eq:thm-LeakageRealInterval-pf1}
 \end{align}
 If we set $a = -\frac 1 2$, $d= \frac 1 2$, $\bn = 1$, $K = \compl L$ and $\Omega_{\leq \Delta} = [-\frac 1 2; 0]$, it can be seen that the assumptions of \cref{lem:SymmetryBound} are satisfied. Thus, we obtain
 \begin{align}
   \bnorm{\cD(\varphi)|_{L}} &\leq  2^{-\frac 1 2} \bnorm{ \hat p^{\leak}_{\tNF, \tNF/4} \cdot \cF(\varphi) } + \Csyma{\tNF/4} \bnorm{ \varphi_\L }, \label{eq:thm-LeakageRealInterval-pf2}
 \end{align}
 where $\varphi_\L := \varphi|_{[-\frac 1 2; 0]} $ denotes the left-hand part of $\varphi$. If we define $\varphi_\R := \varphi|_{[0; \frac 1 2]} $, an analogous estimate can be obtained for the right-hand domain $R$:
 \begin{align}
   \bnorm{\cD(\varphi)|_{R}} &\leq  2^{-\frac 1 2} \bnorm{  \hat p^{\leak}_{\tNF, \tNF/4} (-\cdot) \cdot \cF(\varphi) } + \Csyma{\tNF/4} \bnorm{ \varphi_\R } \nnl
   &= 2^{-\frac 1 2} \bnorm{  \hat p^{\leak}_{\tNF, \tNF/4} \cdot \cF(\varphi) } + \Csyma{\tNF/4} \bnorm{ \varphi_\R }, \label{eq:thm-LeakageRealInterval-pf3}
 \end{align} 
 where it has been exploited that $\hat p _{\leak,  \tNF/4 }^{\real}$ is an even function by definition.
 
 Now we apply the estimates \eqref{eq:thm-LeakageRealInterval-pf2} and \eqref{eq:thm-LeakageRealInterval-pf3} to \eqref{eq:thm-LeakageRealInterval-pf1} and exploit that $\varphi_\L$ and $\varphi_\R$ are $L^2$-orthogonal so that $\norm{\varphi_\L}^2 + \norm{\varphi_\R}^2= \norm{\varphi}^2$ and $\norm{\varphi_\L} + \norm{\varphi_\R} \leq 2^{1/2} \norm{\varphi}$:
 \begin{align}
   \bnorm{\cD(\varphi)|_{\compl K}}^2 &\leq  \bnorm{ \hat p^{\leak}_{\tNF, \tNF/4} \cdot \cF(\varphi) }^2 +   \Csyma{\tNF/4}^2 \Parens{  \bnorm{ \varphi_\L }^2 + \bnorm{ \varphi_\R }^2 } \nnl
   &\quad + 2^{\frac 1 2} \Csyma{\tNF/4} \bnorm{ \hat p^{\leak}_{\tNF, \tNF/4} \cdot \cF(\varphi) } \Parens{  \bnorm{ \varphi_\L } + \bnorm{ \varphi_\R } } \nnl
   &\leq \bparens{ \bnorm{  \hat p^{\leak}_{\tNF, \tNF/4} \cdot \cF(\varphi) } +  \Csyma{\tNF/4}  \bnorm{ \varphi  }  }^2.  \nonumber 
 \end{align}
\end{proof}
\vspace{.5em}

 As seen in \secref{SS:BoundsComplex-ExplicitFilters}, $\hat p^{\leak}_{\tNF, \tNF/4}$ acts as high-pass filter with cutoff-frequency $|\xi| \approx \tNF/2$.  Provided that $\Csyma{\tNF/4} < 1$ is small enough, \eqref{eq:thm-LeakageRealInterval-1} thus guarantees positive contrast $\norm{\cD(\varphi)|_K} \geq ( 1 - \Csyma{\tNF/4}^2 - \varepsilon^2 )^{1/2} \norm{\varphi}$ for some small $\varepsilon > 0$ if $\cF(\varphi)$ is concentrated within the interval $[-\tNF/2;\tNF/2]$. This indicates that image-reconstruction is stable down to features of size $r \gtrsim 2 \pi / \tNF$, which is already the \emph{upper} limit for the achievable resolution by \secref{SSS:GaussianResoEstimatesComplex}. Moreover, as the object-domain is $\Omega = K$ in \cref{thm:LeakageRealInterval}, this optimal resolution can be obtained in the entire FoV!
 
 However, the surprisingly strong 1D-result does \emph{not} carry over to higher dimensions 
 because square detectors $K = \SquareDom$ for $m > 1$ have \emph{corners}, close to which image-reconstruction is unstable down to low spatial frequencies as found in \secref{SSS:GaussianResoEstimatesReal}. We have to exclude the considered objects from having support in these unstable regions: 
 
 \vspace{.5em}
\begin{theorem}[Leakage estimate for real-valued objects in square domains] \label{thm:LeakageRealSquare}
Let $K = [-\frac 1 2 ; \frac 1 2]^m$ and  $\Omega  := \bigcup_{j = 1}^m S_{\Delta,j}$ with $S_{\Delta,j} := \mR^{j-1} \times [-\frac 1 2 + \Delta ; \frac 1 2 - \Delta] \times \mR^{m-j}$ for $0 < \Delta < \frac 1 2$. Then it holds that
 \begin{align}
  \Norm{\cD(\varphi)|_{\compl K}} &\leq \bnorm{ \hat p^{\leak}_{ \tNF, \tNFa {\Delta} } \cdot \cF(\varphi) } + \Csyma{\tNFa {\Delta}}\norm{\varphi|_{\Omega_{\leq \Delta} } } \MTEXT{for all} \varphi \in L^2(\Omega, \mR),
   \end{align}
   where $\Omega_{\leq \Delta} := \Omega \setminus (-\frac 1 2 + \Delta ; \frac 1 2 - \Delta)^m$ denotes the part of $\Omega$ with distance less than $\Delta$ to $\partial K$. 
\end{theorem}
\vspace{.5em}
\begin{proof}
  Let $\varphi \in L^2(\Omega, \mR)$.
  If we define the half-spaces $K_{j,\pm} :=  \mR^{j-1} \times \pm[-\frac 1 2 ; \infty) \times \mR^{m-j}$, then it holds that $\compl K = \bigcup_{j = 1, \pm }^m \compl{K_{j,\pm}}$. Thus, we have
  \begin{align}
   \bnorm{\cD(\varphi)|_{\compl K}}^2 \leq \sum_{j = 1, \pm }^m \bnorm{\cD(\varphi)|_{\compl{K_{j,\pm}}}}^2. \label{eq:thm-LeakageRealSquare-pf1}
  \end{align}
 By construction, each of the squared norms on the right-hand side can be estimated via \cref{lem:SymmetryBound} (with parameters $a = -\frac 1 2$, $\bn = \pm \be_j$), yielding
   \begin{align}
    \bnorm{\cD(\varphi)|_{\compl{K_{j,\pm}}}} \leq  2^{-\frac 1 2} \bnorm{   \hat p ^{\leak}_{\tNF, \tNFa {\Delta} , j }  \cdot \cF(\varphi) } + \Csyma{\tNFa {\Delta}} \bnorm{ \varphi_{j,\pm} }, \label{eq:thm-LeakageRealSquare-pf2}
  \end{align}
  where we have defined $\hat p ^{\leak}_{\tNF, \tNFa {\Delta} , j }(\bxi) := \hat p^{\leak}_{\tNF, \tNFa {\Delta}}(\be_j \cdot \bxi) = \hat p^{\leak}_{\tNF, \tNFa {\Delta}}(-\be_j \cdot \bxi)$ and $\varphi_{j, \pm} := \varphi|_{ \Omega_{j, \pm} }$ with $\Omega_{j, \pm} := \Omega \cap \parens{ \mR^{j-1} \times \pm[-\frac 1 2 ; -\frac 1 2 + \Delta] \times \mR^{m-j} }$. Inserting \eqref{eq:thm-LeakageRealSquare-pf2} into \eqref{eq:thm-LeakageRealSquare-pf1} yields
  \begin{align}
   \bnorm{\cD(\varphi)|_{\compl K}}^2 &\leq   \frac 1 2 \sum_{j = 1, \pm }^m \bnorm{  \hat p ^{\leak}_{\tNF, \tNFa {\Delta} , j } \cdot \cF(\varphi) }^2  + \;  \Csyma{\tNFa {\Delta}}^2 \bbparens{ \sum_{j = 1, \pm }^m  \bnorm{ \varphi_{j,\pm} }^2 } \nnl
   &\quad + 2^{\frac 1 2} \Csyma{\tNFa {\Delta}} \sum_{j = 1, \pm }^m  \bnorm{  \hat p ^{\leak}_{\tNF, \tNFa {\Delta} , j } \cdot \cF(\varphi) }\bnorm{ \varphi_{j,\pm} } . \label{eq:thm-LeakageRealSquare-pf3}  
  \end{align}
  The last summand on the right-hand side of \eqref{eq:thm-LeakageRealSquare-pf3} can be regarded as a euclidean inner product in $\mR^{2m}$. By applying  Cauchy-Schwarz' inequality to this term and using that $ \sum_{j = 1, \pm }^m \norm{  \hat p ^{\leak}_{\tNF, \tNFa {\Delta} , j } \cdot \cF(\varphi) }^2 = 2 \norm{   \hat p ^{\leak}_{\tNF, \tNFa {\Delta}  }  \cdot \cF(\varphi) }^2 $ by \eqref{eq:thm-LeakageSquareSimple-1}, \eqref{eq:thm-LeakageRealSquare-pf3} becomes
  \begin{align}
       \bnorm{\cD(\varphi)|_{\compl K}}^2 &\leq  \bbparens{ \bnorm{   \hat p ^{\leak}_{\tNF, \tNFa {\Delta}  }  \cdot \cF(\varphi) }  +  \Csyma{\tNFa {\Delta}}  \bbparens{ \sum_{j = 1, \pm }^m  \bnorm{ \varphi_{j,\pm} }^2 }^{\frac 1 2} }^2 
  \end{align}

  Now the  choice of  $\Omega$ ensures that the sub-domains $\{ \Omega_{j, \pm} \}$ are mutually disjoint (up to intersections of measure zero). Hence, the $\{ \varphi_{j,\pm} \}$ are mutually $L^2$-orthogonal, which implies 
    \begin{align}
       \bbparens{ \sum_{j = 1, \pm }^m \bnorm{ \varphi_{j,\pm} }^2 }^{\frac 1 2}  &=  \bnorm{ \ssum_{j = 1, \pm }^m  \varphi_{j,\pm} } =  \bnorm{  \varphi|_{(\bigcup_{j,\pm} \Omega_{j,\pm})} }  =  \bnorm{  \varphi|_{\Omega_{\leq \Delta}} } . 
    \end{align}
  \end{proof}
\vspace{.5em}

\vspace{.5em}
\subsection{Stability estimates for spline objects}  \label{SS:RealStabSpline}

The leakage estimates from the preceding section may be used to derive stability estimates for spline objects analogously as in  \secref{SS:SplinesStability}. 

 \vspace{.5em}
\begin{theorem}[Stability estimate for real-valued splines in intervals] \label{thm:StabRealInterval}
 Let $\Omega = K = [-\frac 1 2; \frac 1 2]$, $k \in \mN_0$, $r > 0$, $\nu \geq 1$, $\tNFa r = r^2 \tNF$ and $\Xi  := [-\nu \pi/\tNFa r^{1/2}; \nu \pi/\tNFa r^{1/2}]$. Then it holds that
 \begin{align}
  \Norm{\cD(\varphi)|_{K}} &\geq C_\stab^{\real, \OneD} ( \tNF, \tNFa r, k, \nu) \Norm{ \varphi }   \MTEXT{for all} \varphi \in \BSpacem 1 \cap  L^2(\Omega, \mR)  \label{eq:thm-StabRealInterval-1}
 \end{align}
 where the constant $C_\stab^{\real, \OneD}  ( \tNF, \tNFa r, k, \nu)$ is given by
 \begin{align}
  &C_\stab^{\real, \OneD} ( \tNF, \tNFa r, k, \nu)^2   =   1 - \Bparens{ \Csyma {\tNF/4}  + \Parens{ C_{\low}^2 + \CBand(k, \nu)^2  \Parens{ C_{\tot}^2 - C_{\low}^2 } }^{\frac 1 2} }^2    \label{eq:thm-StabRealInterval-2} \\
  &C_{\low} := \max_{x\in \Xi } \eFa{\tNF/4}(x)   = \max_{\xi\in \tNF^{1/2} \Xi }  \hat p^{\leak}_{\tNF, \tNF/4 } (\xi),  \qquad C_{\tot} := \max_{x\in \mR } \eFa{\tNF/4}(x) = \max_{\xi\in \mR }  \hat p^{\leak}_{\tNF, \tNF/4 } (\xi) \nonumber
 \end{align}
\end{theorem}
\vspace{.5em}
\begin{proof}
 The proof is similar to that of \cref{thm:StabSpline}: the setting matches the assumptions of \cref{thm:LeakageRealInterval}. With $\Xi_r := (\tNF^{1/2}/\nu) \cdot \Xi = [-\pi/ r; \pi/ r]$, the leakage bound \eqref{eq:thm-LeakageRealInterval-1} yields
\begin{align}
   \bparens{  \Norm{\cD(\varphi)|_{\compl K}}   &- \Csyma {\tNF/4} \norm{\varphi} }^2-  C_{\low}^2 \Norm{\varphi}^2   \leq \bip{ \cF(\varphi)  }{ \bparens{  \abs{ \hat p ^{\leak}_{\tNF, \tNFa {\Delta}  } }^2 - C_{\low}^2 } \cdot \cF(\varphi) } \nnl
   &\leq \bip{ \cF(\varphi)|_{\compl{(\tNF^{1/2} \Xi )}}  }{ \bparens{ \abs{ \hat p ^{\leak}_{\tNF, \tNFa {\Delta}  } }^2  - C_{\low}^2 } \cdot \cF(\varphi)|_{\compl{(\tNF^{1/2} \Xi )}} } \nnl
   &\leq \bparens{ C_{\tot}^2 - C_{\low}^2 } \Norm{ \cF(\varphi)|_{\compl{(\nu\Xi_r)}} }^2  \leq \CBand(k, \nu)^2  \Parens{ C_{\tot}^2 - C_{\low}^2 } \Norm{\varphi}^2  \label{eq:thm-StabRealInterval-pf1}
   \end{align}
   for any $\varphi \in \BSpacem 1 \cap  L^2(\Omega, \mR) $, where the quasi-band-limitation \cref{thm:SplineQuasiBandlimit} has been applied in the final step. Rearranging \eqref{eq:thm-StabRealInterval-pf1} yields
   \begin{align}
   \Norm{\cD(\varphi)|_{\compl K}} &= \Csyma {\tNF/4} \norm{\varphi}  + \Parens{  \Norm{\cD(\varphi)|_{\compl K}}  - \Csyma {\tNF/4} \norm{\varphi} } \nnl 
   &\leq  \Parens{ \Csyma {\tNF/4}  +  \Parens{ C_{\low}^2 + \CBand(k, \nu)^2   \Parens{ C_{\tot}^2 - C_{\low}^2 } }^{\frac 1 2} } \Norm{\varphi} \nnl
   &= \bparens{ 1 -  C_\stab^{\real, \OneD} ( \tNF, \tNFa r, k, \nu)^2 }^{\frac 1 2}  \Norm{\varphi} \label{eq:thm-StabRealInterval-pf2}
\end{align}
 Since $ \Norm{\cD(\varphi)|_{K}} = ( \Norm{\varphi}^2 - \norm{\cD(\varphi)|_{\compl K}}^2 )^{1/2}$, \eqref{eq:thm-StabRealInterval-pf2} proves the assertion.
 \end{proof}
\vspace{.5em}

Once more, the remarkable aspect of the 1D stability result in \cref{thm:StabRealInterval} is that does not require any distance between the object-domain $\Omega$ and the boundary of $K$.
Analogously, we can obtain a stability estimate for the higher-dimensional case:

 \vspace{.5em}
\begin{theorem}[Stability estimate for real-valued splines in square domains] \label{thm:StabRealSquare}
 Within the setting of \cref{thm:LeakageRealSquare}, let  $\tNFa {\Delta} = \Delta^2 \tNF$, $\tNFa r = r^2 \tNF$ and $\Xi  := [-\nu \pi/\tNFa r^{1/2}; \nu \pi/\tNFa r^{1/2}]$. Then it holds that
 \begin{align}
  \Norm{\cD(\varphi)|_{K}} &\geq C_\stab^{\real,m} ( \tNFa {\Delta}, \tNFa r, k, \nu) \Norm{ \varphi }   \MTEXT{for all} \varphi \in \BSpacem 1 \cap  L^2(\Omega, \mR)  \label{eq:thm-StabRealSquare-1}
 \end{align}
 where the constant $ C_\stab^{\real,m} ( \tNFa {\Delta}, \tNFa r, k, \nu)$ is given by
 \begin{align}
  C_\stab^{\real,m} ( \tNFa {\Delta}, \tNFa r, k, \nu)^2   &=   1 - \! \Bparens{ \Csyma {\tNFa {\Delta}}  + m^{\frac 1 2} \Parens{ C_{\low}^2 + \CBand(k, \nu)^2  \Parens{ C_{\tot}^2 - C_{\low}^2 } }^{\frac 1 2}\! }^2  \label{eq:thm-StabRealSquare-2} \\ 
  C_{\low} &:=  \max_{\xi\in  \Xi  } \eFd (x) , \qquad C_{\tot} :=  \max_{\xi\in \mR} \eFd (x) \nonumber 
 \end{align}
\end{theorem}
\vspace{.5em}
\begin{proof}
 Let $\varphi \in \BSpacem 1 \cap  L^2(\Omega, \mR)$. Since $\norm{ \varphi|_{\Omega_{< d}}} \leq \norm{ \varphi }^2$, we then have by \cref{thm:LeakageRealSquare}:
 \begin{align}
   \Norm{\cD(\varphi)|_{\compl K}} - \Csyma {\tNFa {\Delta}}  \Norm{\varphi}   &\leq \bnorm{ \hat p^{\leak}_{ \tNF, \tNFa {\Delta} } \cdot \cF(\varphi) }  = \bbparens{ \sum_{j = 1}^m  \bip{ \cF(\varphi) }{ \babs{ \hat p^{\leak}_{ \tNF, \tNFa {\Delta}, j }}^2 \cdot \cF(\varphi) } }^{\frac 1 2} \label{eq:thm-StabRealSquare-pf1}
 \end{align}
 with  quasi-1D functions $\hat p^{\leak}_{ \tNF, \tNFa {\Delta}, j }(\bxi) = \hat p^{\leak}_{ \tNF, \tNFa {\Delta} }( \be_j \cdot \bxi )$ as defined in the proof of \cref{thm:LeakageRealSquare}. Let $\Xi_j := \mR^{j-1} \times \Xi \times \mR^{m-j}$ and $\Xi_{r,j} := (\tNF^{1/2}/\nu) \cdot \Xi_j =  \mR^{j-1} \times [- \pi/ r; \pi/ r ] \times \mR^{m-j}$ for $1\leq j \leq m$. Then it holds that $\max_{\bxi \in (\nu\Xi_{r,j})} | \hat p^{\leak}_{ \tNF, \tNFa {\Delta}, j }(\bxi)|  = C_{\low}$ and $\max_{\bxi \in \mR^m} | \hat p^{\leak}_{ \tNF, \tNFa {\Delta}, j }(\bxi)|  = C_{\tot}$ and hence, by a derivation completely analogously as in \eqref{eq:thm-StabRealInterval-pf1},
 \begin{align}
  \bip{ \cF(\varphi) }{ \abs{ \hat p^{\leak}_{ \tNF, \tNFa {\Delta}, j } }^2 \cdot \cF(\varphi) } - C_{\low}^2 &\norm{\varphi}^2 \leq (C_{\tot}^2 - C_{\low}^2)  \bnorm{ \cF(\varphi)|_{\compl{(\nu\Xi_{r,j})}} }^2 \nnl
   &\leq \CBand(k, \nu)^2 (C_{\tot}^2 - C_{\low}^2)  \norm{ \varphi }^2  \label{eq:thm-StabRealSquare-pf2}
 \end{align}
 for all $1 \leq j \leq m$, where \cref{thm:SplineQuasiBandlimitND} has been applied. Bounding the right-hand side of \eqref{eq:thm-StabRealSquare-pf1} via \eqref{eq:thm-StabRealSquare-pf2} and rearranging as in the proof of \cref{thm:StabRealInterval} yields the assertion.
 \end{proof}
\vspace{.5em}

 \begin{figure}[hbt!]
 \centering
 \includegraphics[width=\textwidth]{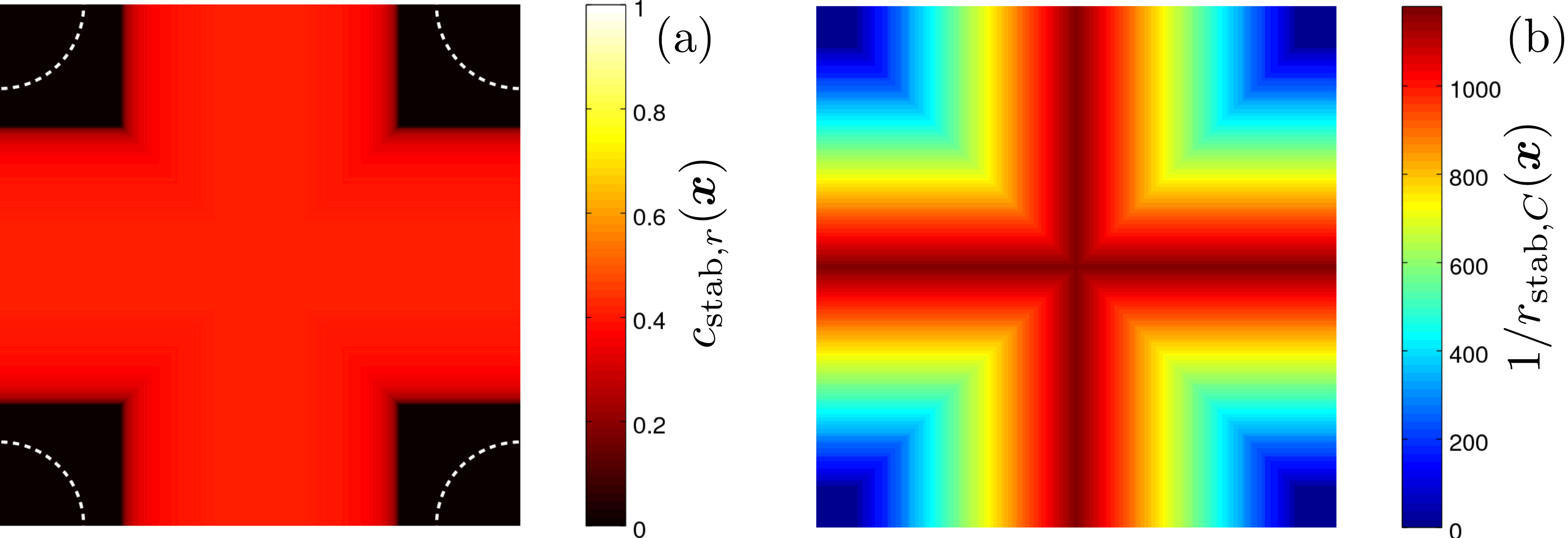}
 \caption{Same plot as \cref{fig:StabilityResolutionComplex}, yet for the \emph{real-valued} setting of \cref{IP:Real}(a). Local stability constant (a) and -resolution (b) have been computed according to \cref{thm:StabRealSquare} via \eqref{eq:LocalStabReal} and \eqref{eq:LocalResoReal}.}
 \label{fig:StabilityResolutionReal}
\end{figure}

\vspace{.5em}
\subsection{Application: resolution estimates} \label{SS:RealResolutionEstimates}

Analogously as for the complex-valued case in \secref{SS:SplinesResolutionEstimates}, we can use \cref{thm:StabRealSquare} to assess the resolution within the \emph{real-valued} setting of \cref{IP:Real}(a).

For illustration, we consider exactly the same setting as in \secref{SS:SplinesResolutionEstimates}, 
but express the local stability constant and resolution via the improved bound \eqref{eq:thm-StabRealInterval-1}, exploiting real-valuedness:
\begin{align}
 c_{\stab, r}(\bx) &:= \sup \big\{C_\stab^{\real,m} ( \tNFa {\Delta}, \tNFa r, k, \nu) : \bx \in {\textstyle \bigcup_{j = 1}^m} S_{d,j} \big\} \label{eq:LocalStabReal} \\
 r_{\stab, C}(\bx) &:= \inf \left\{r > 0 : c_{\stab, r}(\bx) \geq C \right\} \label{eq:LocalResoReal}
\end{align}
with $S_{d,j}$ as defined in \cref{thm:LeakageRealSquare}. $c_{\stab, r}$ and $1/r_{\stab, C}$ are plotted in \cref{fig:StabilityResolutionReal}(a),(b).
 
According to \cref{fig:StabilityResolutionReal}(a), stable reconstruction is guaranteed within the entire FoV except for square-shaped neighborhoods around the corners of $K$. The width of the unstable region is about 1.5 times $\pi / (\tNF r)$ -- the value that is to be expected from the analysis \secref{SSS:GaussianResoEstimatesReal}. Likewise, the local resolutions in \cref{fig:StabilityResolutionReal}(b) are qualitatively in good agreement with the results from the wave-packet-analysis in \secref{SSS:GaussianResoEstimatesReal}, compare \cref{fig:WavepacketResolution}(b).

\vspace{.5em}
\section{Extension to the phaseless case: application to linearized \HIGHLIGHT{XPCI}} \label{S:Phaseless}

So far, the analysis has been limited to the case where the full complex-valued propagated wave field $\cD(h)(\bx)$ -- including the phase -- is measured at each point $\bx\in K$ of the FoV. In the following, we outline how the results can be extended to the case of phaseless data. We   consider the inverse problems \cref{IP:Complex}(b) and \cref{IP:Real}(b) that model image-reconstruction in \HIGHLIGHT{XPCI} within the linear CTF-regime. On the contrary, analyzing the \emph{nonlinear} problems \cref{IP:Complex}(c) and \cref{IP:Real}(c) is beyond reach as stability is an open problem for these even in the case of a full FoV $K = \mR^m$.

\vspace{.5em}
\subsection{Leakage estimates}

As a first step, we aim to bound the amount of data that is leaked outside a square field-of-view within the setting of \cref{IP:Complex}(b) and \cref{IP:Real}(b). This is fairly simple as the measured data, $\sim 2\Re( \cD(h) )$, relates to Fresnel-propagation simply by the \emph{pointwise} real-part  and $|\Re(z)| \leq |z|$ for all $z\in\mC$. This yields the following bound:

\vspace{.5em}
\begin{theorem}[Leakage bound for linearized \HIGHLIGHT{XPCI} data] \label{thm:LeakagePhaseless}
 Let $K \subset \mR^m$ be measurable and $\cT, \cS_\alpha$ be the forward maps from \cref{IP:Complex}(b) and \cref{IP:Real}(b). Then it holds that 
 \begin{align}
   \Norm{ \cT(h)|_{\compl K} } &\leq 2 \Norm{ \cD(h)|_{\compl K} } \MTEXT{for all} h \in L^2(\mR^m),  \label{eq:thm-LeakagePhaseless-1}  \\
   \Norm{ \cS_\alpha(\varphi)|_{\compl K} } &\leq 2 \Norm{ \cD(\varphi)|_{\compl K} } \MTEXT{for all} \varphi \in L^2(\mR^m, \mR).  \label{eq:thm-LeakagePhaseless-2}
 \end{align}
\end{theorem}
\vspace{.5em}

The gist of \cref{thm:LeakagePhaseless} is simple: it states that the leaked part of \HIGHLIGHT{XPCI} data, $\cT(h)|_{\compl K}$, cannot contain more information than the corresponding phased Fresnel-data $\cD(h)|_{\compl K}$. Despite its simplicity, however, this result has important consequences:
by  \cref{thm:LeakagePhaseless}, literally \emph{any} of the leakage estimate of the preceding sections induces a bound for the phaseless case.

\vspace{.5em}
\subsection{Stability estimates} 

Using the simple insight from \cref{thm:LeakagePhaseless}, we may derive stability estimates for \HIGHLIGHT{ phase contrast imaging} with finite detectors. To this end, we combine leakage estimates with the stability results for \HIGHLIGHT{XPCI} with \emph{infinite} FoVs from \cref{thm:WellPosed}:

\vspace{.5em}
\begin{theorem}[Stability estimate for linearized \HIGHLIGHT{XPCI} with square detector] \label{thm:StabilityPhaseless}
 Let $K = \SquareDom$ and $\Omega \subset \SquareDomMin \Delta $ for some $0 < \Delta < \frac 1 2$. Let $T \in \{ \cT, \cS_\alpha \}$ and $h \in L^2(\Omega)$, where $h$ is assumed to be real-valued if $T= \cS_\alpha$. Furthermore, let $C_{\textup{stab}}^{\textup{IP$\ast$}}(\Omega,\tNF,(\alpha))$ denote the stability constant of $T$ for a full FoV from \cref{thm:WellPosed}.
 Then it holds that
 \begin{align}
   \Norm{ T(h)|_{K} }^2  &\geq \Norm{ T(h)  }^2  - 4 \bnorm{ \hat p^{\leak}_{\tNF, \tNFa {\Delta}}  \cdot \cF(h)}^2  \label{eq:thm-StabilityPhaseless-1} \\
   &\geq C_{\textup{stab}}^{\textup{IP$\ast$}}(\Omega,\tNF,(\alpha))^2 \bnorm{ h  }^2 - 4 \bnorm{ \hat p^{\leak}_{\tNF, \tNFa {\Delta}}  \cdot \cF(h)}^2. \label{eq:thm-StabilityPhaseless-2} 
 \end{align}
 If $h \in \BSpace$ is moreover a B-spline and $\nu \geq 1$, then \eqref{eq:thm-StabilityPhaseless-2} further implies that
  \begin{align}
   \Norm{ T(h)|_{K} }^2  
   &\geq  \Parens{ C_{\textup{stab}}^{\textup{IP$\ast$}}(\Omega,\tNF,(\alpha))^2 - 4 \bparens{ 1 - \CStab  (\tNFa {\Delta}, \tNFa r, k, \nu )^{2m} }  } \Norm{ h  }^2  \label{eq:thm-StabilityPhaseless-4}
 \end{align}
 where the notation is as in \cref{thm:StabSpline}.
\end{theorem}
\vspace{.5em}
\begin{proof}
 The first inequality, \eqref{eq:thm-StabilityPhaseless-1}, is obtained by bounding $\Norm{ T(h)|_{\compl K} }$ via \cref{thm:LeakagePhaseless,thm:LeakageSquareSimple} and using that $\Norm{ T(h)|_{K} }^2 = \Norm{ T(h) }^2 - \Norm{ T(h)|_{\compl K} }^2 $. \eqref{eq:thm-StabilityPhaseless-2} then follows from \eqref{eq:thm-StabilityPhaseless-1} by estimating $\Norm{ T(h) }$ via \cref{thm:WellPosed}. The bound
 \eqref{eq:thm-StabilityPhaseless-4} is obtained analogously if $\norm{ \cD(h)|_{\compl K} }^2 =  \norm{ h }^2 - \norm{ \cD(h)|_{K} }^2$ is estimated via \cref{thm:StabSpline} instead of \cref{thm:LeakageSquareSimple}.
\end{proof}
\vspace{.5em}

While the right-hand side of \eqref{eq:thm-StabilityPhaseless-4} is clearly the simplest of all bounds in  \cref{thm:StabilityPhaseless}, it is also the most pessimistic. The reason is that both the full-FoV-contrast $\norm{T(h)}$ \emph{and} the leaked part $\norm{T(h)|_{\compl K} }$ are bounded via \emph{worst-case} estimates. Hence, the bound \eqref{eq:thm-StabilityPhaseless-4} is likely to be far from sharp since, otherwise, some $h \in \BSpace \cap L^2(\Omega)$ would have to both  {minimize}  $\norm{T(h)}$ and  {maximize}  $\norm{T(h)|_{\compl K} } $. However, as shown in \cite{MaretzkeHohage2017SIAM},  $\norm{T(h)}$ is minimized by \emph{low-frequency} modes, whereas the leakage estimates are in terms of \emph{high-pass} filters.

Despite its lossiness, we demonstrate that the bound \eqref{eq:thm-StabilityPhaseless-4} may indeed guarantee stability in practically relevant settings. To this end, the required stability constant for an infinite FoV $C_{\textup{stab}}^{\textup{IP}\ast}$ is approximated numerically, which can be done to high accuracy for ball- or square-domains $\Omega$.
Let us first consider \cref{IP:Complex}(b). This problem is excessively ill-conditioned \cite{MaretzkeHohage2017SIAM} even for a full FoV, except for settings with very small object-domains. For such a case, we show that stability also holds with finite detectors:

\vspace{.5em}
\begin{example}[Stability estimate for \HIGHLIGHT{XPCI} of weak objects (\cref{IP:Complex}(b))] \label{ex:1}
 Let $\tNF = 2\cdot 10^3$ and $K = [-\frac 1 2 ; \frac 1 2]^2$. Let  $h \in \BSpace \cap L^2(\Omega)$ with support  $\Omega = [-\frac 1 {20}; \frac 1 {20}]^2$, resolution $1/r = 190$ and spline order $k = 7$. Then the bound \eqref{eq:thm-StabilityPhaseless-4} guarantees stability with
 \begin{align}
  \Norm{ \cT(h)|_{K} } \geq 0.12  \norm{h} \quad (C_{\textup{stab}}^{\textup{IP1(b)}}(\Omega,\tNF) \geq 0.328, \; \CStab  (\tNFa {\Delta}, \tNFa r, k, 1.2 ) \geq 0.988).
 \end{align}
 By \cref{res:ComplexReso}, 
 an \emph{upper} bound for the resolution is given by $1/r \lesssim 0.45\tNF/\pi \approx 290$.
\end{example}
\vspace{.5em}
\noindent Unfortunately, as $C_{\textup{stab}}^{\textup{IP1(b)}}(\Omega,\tNF)$ decays exponentially with the Fresnel number associated with the size of $\Omega$ \cite{MaretzkeHohage2017SIAM}, stability cannot be guaranteed for larger object-domains $\Omega$ or $\tNF$.

The situation is better for \cref{IP:Real}(b), i.e.\ for the reconstruction of \emph{homogeneous objects} as introduced in \secref{SSS:MathForwardModels}. Of particular  relevance are non-absorbing, pure  {phase objects}: 


\vspace{.5em}
\begin{example}[Stability estimate for \HIGHLIGHT{XPCI} of weak phase objects (\cref{IP:Real}(b): $\alpha = 0$)]  \label{ex:2}
 Let $\tNF = 5\cdot 10^3$ and $K = [-\frac 1 2 ; \frac 1 2]^2$. Let  $\varphi \in \BSpace \cap L^2(\Omega, \mR)$ with   $\Omega = \{ \bx \in \mR^2: |\bx| \leq \frac 1 {10} \}$, resolution $1/r = 350$ and $k = 7$. Then the bound \eqref{eq:thm-StabilityPhaseless-4} guarantees stability with 
 \begin{align}
  \Norm{ \cS_0 ( \varphi ) |_{K} } \geq 0.05 \norm{\varphi} \quad (C_{\textup{stab}}^{\textup{IP2(b)}}(\Omega,\tNF, 0) \geq 0.151, \; \CStab  (\tNFa {\Delta}, \tNFa r, k, 1.25 ) \geq 0.997).
 \end{align}
  By \cref{res:RealReso}, 
 an \emph{upper} bound for the resolution is given by $1/r \lesssim 0.5\tNF/\pi \approx 800$.
\end{example}
\vspace{.5em}

Yet, the full-FoV stability constant $C_{\textup{stab}}^{\textup{IP2(b)}}(\Omega,\tNF, 0)$ decays like $\tNF^{-1}$ for $\tNF\to\infty$, which is still too fast for \eqref{eq:thm-StabilityPhaseless-4} to guarantee stability at larger Fresnel numbers. This is different when the imaged sample is also known to be slightly absorbing,
in which case the asymptotics improve to $C_{\textup{stab}}^{\textup{IP2(b)}}(\Omega,\tNF, \alpha) \gtrsim \tNF^{-1/2}$ \cite{MaretzkeHohage2017SIAM}. This enables stability guarantees for reconstructions at \emph{optical} resolutions as fine as the native resolution of typical detectors. In such a setting, the finite FoV is no longer a limiting factor for the performance of the imaging setup. We consider an example for a sample satisfying $\mu = 0.1\phi$, i.e.\ for $10\,\%$ absorption (see \secref{SSS:MathForwardModels}):

\vspace{.5em}
\begin{example}[Stability estimate for \HIGHLIGHT{XPCI} of  homogeneous objects (\cref{IP:Real}(b): $\alpha = \arctan(\frac 1 {10})$)]  \label{ex:3}
 Let $\tNF = 4\cdot 10^4$ and $K = [-\frac 1 2 ; \frac 1 2]^2$.  Let  $\varphi \in \BSpace \cap L^2(\Omega, \mR)$ with  $\Omega = \{ \bx \in \mR^2: |\bx| \leq \frac 1 {4} \}$, resolution $1/r = 2000$ and $k = 7$. Then the bound \eqref{eq:thm-StabilityPhaseless-4} guarantees stability with 
 \begin{align}
  \Norm{ \cS_\alpha ( \varphi ) |_{K} } \geq 0.08 \norm{\varphi} \quad (C_{\textup{stab}}^{\textup{IP2(b)}}(\Omega,\tNF, \alpha) \geq 0.147, \; \CStab  (\tNFa {\Delta}, \tNFa r, k, 1.25 ) \geq 0.998).
 \end{align}
    By \cref{res:RealReso}, an \emph{upper} bound for the resolution is given by $1/r \lesssim 2^{-3/2}\tNF/\pi \approx 4500$.
\end{example}
\vspace{.5em}


\vspace{.5em}
\subsection{Improved estimates for real-valued objects}

In principle, the improved leakage bounds for the real-valued setting from \secref{S:Real} apply to the CTF-based reconstruction of homogeneous objects, \cref{IP:Real}(b). Unfortunately, the derived bounds are too pessimistic in this setting to enable stability estimates for practically relevant Fresnel numbers. However, note that numerical simulations (not shown) indicate that the larger stability regions for the real-valued case, shown in \cref{fig:WavepacketResolution,fig:StabilityResolutionReal}, indeed seem to carry over to the phaseless \HIGHLIGHT{XPCI}-setting.

\vspace{.5em}
\section{Conclusions} \label{S:Conclusions}

We have studied locality of wave-propagation in the Fresnel- (or paraxial) regime in order to quantify the effects of a finite detector on the stability of X-ray \HIGHLIGHT{ phase contrast imaging} (\HIGHLIGHT{XPCI}). The analysis shows that locality depends on spatial frequencies, i.e.\ the finer the features of some object $h$ the more delocalized it is upon Fresnel-propagation $h \mapsto \cD(h)$. As a consequence, \emph{truncated} diffraction-data, as measured by any real-world detector, introduces a spatially varying resolution limit within the field-of-view: features of the imaged object finer than some limiting length-scale $r_{\stab}$ may induce a signal in the diffraction-pattern that essentially \emph{leaks} out the detection-domain $K $ upon propagation and thus
cannot be stably reconstructed from the data. On the contrary, Lipschitz-stability estimates hold for images that comply with the resolution limit, as has been proven for multi-variate B-splines. \HIGHLIGHT{The decisive property of B-splines for this result is that they are \emph{quasi band-limited} functions. Notably, the obtained estimates on their concentration in Fourier-space (\cref{thm:SplineQuasiBandlimit,thm:SplineQuasiBandlimitND}) may be of interest beyond the specific inverse problems considered this work.}

The stability results do not only hold for the (hypothetical) case where full complex-valued Fresnel-data $\cD(h)|_K$ is measured, but have also been extended to the \emph{phaseless} setting of \HIGHLIGHT{XPCI} in the linear CTF-regime. However, as the (possibly complicated) interplay between the instabilities due to a finite FoV  and those due to the missing phase  is not taken into account, the derived estimates for the phaseless case are expected to be highly non-optimal.

The maximum resolution for a square detector is found to be $1/ r_{\stab} \approx \F$, \HIGHLIGHT{in accordance with the numerical aperture of the lensless imaging setup \cite{Nugent2010coherent,LatychenskaiaEtAl2012_HoloMeetsCDI(withHoloResolution)},} where $\F = b^2 /(\lambda d)$ is the  Fresnel number associated with the detector's aspect-length~$b$ ($\lambda$: wavelength, $d$: propagation-distance). Hence, if $\F$ is smaller than the number of detector-pixels along one dimension, the finite FoV bottlenecks the achievable resolution. For complex-valued images to be reconstructed, the optimal resolution is moreover attained only in the very center of the FoV. Interestingly, this situation is much worse than for the standard \HIGHLIGHT{XPCI} case of homogeneous objects, that boils down to reconstructing a \emph{real-valued} image. In the latter case, maximum resolution $\approx\! \F$ can be achieved in large parts of the FoV, except for the detector-corners. 

The analysis of this work may be readily extended. For once, all results can be adjusted to non-square object- and detection-domains at the cost of a more involved notation. Moreover, it is straightforward to extend the derived locality-bounds to multiple diffraction-patterns acquired at different Fresnel numbers $\tNF_1, \tNF_2, \ldots$, which is a typical setting in \HIGHLIGHT{XPCI}. However, the larger amount of data is not too useful in view of a finite detector because, according to this work's analysis, features that leak outside the FoV for the largest Fresnel number are lost in \emph{all} diffraction patterns.
Finally, the estimates obtained within the Fresnel-regime may be generalized to propagation within the full Helmholtz equation, by combining them with bounds on the deviation from the paraxial limit. Thereby, the results might be applied to a large range of scattering experiments that give rise to \emph{approximately} paraxial wave-fields.

\vspace{.5em}
 \section*{Acknowlegdments}

The author thanks Johannes Hagemann, Malte Vassholz, Thorsten Hohage and Tim Salditt for inspiring discussions.
Financial support by Deutsche Forschungsgemeinschaft DFG through SFB 755 - Nanoscale
Photonic Imaging is gratefully acknowledged.

%
%

\vspace{1em}
\appendix

\section{Fresnel-propagation and frequency shifts} \label{App:FresnelPropFreqShift}

\begin{proof}[Proof of \cref{lem:FresnelPropFreqShift}]
 By the alternate form of the Fresnel propagator \eqref{eq:FresnelPropAltForm}, we have
 \begin{align}
   \E^{\I m \pi /4} \tNF ^{-\frac m 2 }  \cD( \be_{\ba} \cdot f )  = \nF \cdot \cF( \nF \cdot \be_{\ba} \cdot f ) ( \tNF (\cdot)). \label{eq:lem-FresnelPropFreqShift-1} 
 \end{align}
 Moreover, it holds  for all $\bx \in \mR^m$
 \begin{align} 
  \nF \cdot \be_{\ba} (\bx) &= \exp\left( \I \left( \frac{ \tNF \bx^2 } 2  + \ba \cdot \bx \right) \right) = \exp\left(  - \frac{ \I    \ba^2  }{ 2 \tNF }   \right) \exp\left( \frac{ \I    \tNF (\bx + \ba / \tNF)^2  } 2   \right) \nnl
  &= \mF(\ba)  \cdot \nF \left( \bx + \ba / \tNF \right) = \mF(\ba)  \cdot T_{\ba/\tNF}( \nF )(\bx), \label{eq:lem-FresnelPropFreqShift-2} 
 \end{align}
 Since $(T_{\bt})^{-1} = T_{-\bt}$ and $\cF\big( T_{\bt} (g) \big) = \be_{\bt} \cdot \cF (g)$ for any $\bt \in \mR^m$, $g \in L^2(\mR^m)$, we thus have
 \begin{align}
 \nF \cdot \cF( \nF \cdot \be_{\ba} \cdot f ) ( \tNF (\cdot)) &=  \mF(\ba) \cdot \nF   \cdot \cF \left( T_{\ba/\tNF}(\nF) \cdot f \right) ( \tNF (\cdot)) \nnl
  &=  \mF(\ba) \cdot \nF   \cdot \cF \left( T_{\ba/\tNF}\left( \nF \cdot T_{-\ba/\tNF}(f) \right) \right) ( \tNF (\cdot)) \nnl
  &=  \mF(\ba) \cdot \nF   \cdot \be_{\ba/\tNF}  ( \tNF (\cdot)) \cdot  \cF \left( \nF \cdot T_{-\ba/\tNF}(f)  \right) ( \tNF (\cdot))  \nnl
  &=    \mF(\ba) \cdot \be_{\ba}  \cdot \left( \nF \cdot  \cF \left( \nF \cdot T_{-\ba/\tNF}(f)  \right) ( \tNF (\cdot)) \right)  \nnl
  &=    \mF(\ba) \cdot \be_{\ba}  \cdot \E^{\I m \pi /4} \tNF ^{-\frac m 2 } \cD \left( T_{-\ba/\tNF}(f)  \right) \label{eq:lem-FresnelPropFreqShift-3} 
 \end{align}
 By comparing to \eqref{eq:lem-FresnelPropFreqShift-1} and exploiting that $\cD$ commutes with translations as a convolution operator, we finally obtain 
 \begin{align*}
 \cD( \be_{\ba} \cdot f ) &=    \mF(\ba) \cdot \be_{\ba}  \cdot \cD \left( T_{-\ba/\tNF}(f)  \right)   =    \mF(\ba) \cdot \be_{\ba}  \cdot T_{-\ba/\tNF}\left( \cD  (f)  \right).  
 \end{align*}
\end{proof}
\vspace{1em}

\section{Quasi-band-limitation of B-splines} \label{App:SplineQuasiBandlimit}

\begin{proof}[Proof of \cref{thm:SplineQuasiBandlimit}]
 We prove the estimate \eqref{eq:thm-SplineQuasiBandlimit-1} for $h = \sum_{j \in \mZ} b_j B_{k}(\cdot /r - j - \bo ) \in \BSpacem 1$ with coefficients that vanish for all but finitely many entries, i.e.\ $(b_j) \in \ell^0(\mZ) := \{ (c_j)_{j \in \mZ} \subset \mC : \exists J \subset \mZ \textup{ finite s.t.\ } c_l = 0 \textup{ for } l \in \mZ \setminus J $. This is sufficient since such splines form an $L^2$-dense subspace of $\BSpacem 1$ (by denseness of $\ell^0(\mZ)$ in $ \ell^2(\mZ)$ and the Riesz-sequence property \eqref{eq:SplineRieszSeq}) and both sides of  \eqref{eq:thm-SplineQuasiBandlimit-1} are $L^2$-continuous in $h$.
 
 For the considered $h$, all sums of the form $\sum_{j\in\mZ} b_j (\ldots)$ are finite. By linearity and the behavior of the Fourier-transform under translations and dilations, this implies that
 \begin{align}
  \cF(h)(\xi)  &= \cF \bbparens {\sum_{j \in \mZ } b_{j} B_k(\cdot /r - j - \bo ) }(\xi)  = \sum_{j \in \mZ } b_{j} \cF \Parens { B_k(\cdot /r - j - \bo ) }(\xi) \nnl
  &= \underbrace{ \bbparens{ \exp \Parens{ - \I r \xi  \bo }  \sum_{j \in \mZ} b_{j} \exp \Parens{ - \I r \xi  j }   } }_{=: \hat h_{\textup{per}}(r\xi)}  r \cF(B_k)(r\xi) \MTEXT{for all} \xi \in \mR. \label{eq:thm-SplineQuasiBandlimit-pf1} 
 \end{align}
 From \eqref{eq:thm-SplineQuasiBandlimit-pf1}, it can be readily seen that the function $\hat h_{\textup{per}}$ is $2\pi$-periodic, i.e.\ $\hat h_{\textup{per}}(\xi + 2\pi l) = \hat h_{\textup{per}}(\xi)$ for all $\xi \in \mR, l \in \mZ$.
 
 In order to prove the estimate \eqref{eq:thm-SplineQuasiBandlimit-1}, we decompose the Fourier-domain: with $\bar \nu := 1+ 2 \ceil{(\nu-1)/2}$ as defined in the assumptions, it holds that
 \begin{align}
  \compl{(\nu \Xi_ r)} = \bparens{ (\bar \nu \Xi_ r) \setminus (\nu \Xi_ r) } \cup \bigcup_{n = 1 +  \ceil{(\nu-1)/2}}^\infty \Parens{ \Xi_ r + \frac{2\pi } r n} \cup  \Parens{ \Xi_ r - \frac{2\pi} r  n}  , 
 \end{align}
 where the union is mutually disjoint except for intersections of Lebesgue-measure zero. Accordingly, the squared $L^2$-norm over $\compl{(\nu \Xi_ r)}$ can be written as a sum
  \begin{align}
  \Norm{ \cF(h) |_{\compl{(\nu \Xi_ r)}}}^2 &= \Norm{ \cF(h) |_{(\bar \nu \Xi_ r) \setminus (\nu \Xi_ r)}}^2  + \sum_{n = 1 +  \ceil{(\nu-1)/2} }^\infty  \Parens{ \bnorm{\cF(h)|_{\Xi_ r + \frac{ 2 \pi} r n }}^2 + \bnorm{\cF(h)|_{\Xi_ r - \frac{2 \pi} r n }}^2   }   \label{eq:thm-SplineQuasiBandlimit-pf2} 
 \end{align}
 
 We first consider the squared norms in the second summand on the right-hand-side of \eqref{eq:thm-SplineQuasiBandlimit-pf2}. By the $2\pi$-periodicity of $\hat h_{\textup{per}}$, we have
 \begin{align}
  r^{-1} \bnorm{\cF(h)|_{\Xi_ r + \frac{ 2\pi} r l }}^2 \! &= r \! \int_{\Xi_ r + \frac{2\pi} r l} \Abs{ \hat h_{\textup{per}}(r\xi) }^2 \! \Abs{ \cF(B_k)(r\xi)}^2 \D \xi = \! \int_{(2\pi l-1)\pi}^{(2\pi l+1)\pi}  \! \Abs{ \hat h_{\textup{per}}( \xi) }^2 \! \Abs{ \cF(B_k)( \xi)}^2 \D \xi \nnl
  &=  \int_{-\pi}^{\pi}  \Abs{ \hat h_{\textup{per}}( \xi) }^2 \Abs{ \cF(B_k)( \xi + 2\pi l)}^2 \D \xi \label{eq:thm-SplineQuasiBandlimit-pf3} 
 \end{align}
for all $l \in \mZ$. Hence, we obtain for all $n \in \mN$
   \begin{align}
    \!\bnorm{\cF(h)|_{(\Xi_ r + \frac{ 2 \pi n} r ) \cup (\Xi_ r - \frac{ 2 \pi n} r )}}^2 &=  r\! \int_{-\pi}^\pi \! \Babs{ \hat h_{\textup{per}}(\xi)  }^2 \! \Parens{ \Abs{ \cF(B_k)( \xi + 2\pi n)}^2 \! + \! \Abs{ \cF(B_k)( \xi - 2\pi n)}^2 } \D \xi  \nnl 
    &\leq r c_{k, n}  \int_{-\pi}^\pi \Babs{ \hat h_{\textup{per}}(\xi)  }^2   \Abs{ \cF(B_k)( \xi )}^2\, \D \xi =  c_{k, n} \Norm{\cF(h)|_{\Xi_ r  }}^2   \label{eq:thm-SplineQuasiBandlimit-pf4}  \\ 
    c_{k, n} := \sup_{\xi \in [-\pi;\pi]} w_{k, n}(\xi), &\qquad w_{k, n}(\xi) :=  \frac{ \Abs{ \cF(B_k)( \xi + 2\pi n)}^2 + \Abs{ \cF(B_k)( \xi - 2\pi n)}^2 }{\Abs{ \cF(B_k)( \xi )}^2}. \nonumber %
  \end{align}
  
 We aim to explicitly compute the coefficients $c_{k, n}$. To this end, we use the known Fourier transform of $B_k$, $\cF(B_k)(\xi) = (2\pi)^{-1/2} \sinc(\xi/2)^{k+1}$ for all $\xi \in \mR$ where $\sinc(x) := \sin(x)/x$. 
 As the function $\sin^2$ is $\pi$-periodic, it holds that
  \begin{align}
  \frac{ \Abs{ \cF(B_k)( \xi + 2\pi l)}^2 }{ \Abs{ \cF(B_k)( \xi ) }^2 } &=  \frac{ \sin(\xi/2 + \pi l)^{2(k+1)}}{\sin(\xi/2)^{2(k+1)}}\cdot \frac{(\xi/2)^{2(k+1)}}{ (\xi/2 + \pi l)^{2(k+1)}} = \frac{\xi^{2(k+1)}}{ (\xi + 2\pi l)^{2(k+1)}}  \label{eq:thm-SplineQuasiBandlimit-pf5}  
 \end{align}
 for all $\xi\in [-\pi; \pi]$, $l \in \mZ$. 
  Accordingly, the weight-function $w_{k, n}$ is given by
  \begin{align}
   w_{k, n}(\xi) &= \frac{\xi^{2(k+1)}}{ (\xi + 2\pi n)^{2(k+1)}}  + \frac{\xi^{2(k+1)}}{ (\xi - 2\pi n)^{2(k+1)}}  \nnl
   \Rightarrow \; c_{k, n} &= \sup_{\xi \in [-\pi;\pi]} w_{k, n}(\xi) = w_{k, n}(\pm \pi) = \frac 1 {(2n-1)^{2(k+1)}} + \frac 1 {(2n+1)^{2(k+1)}}, \label{eq:thm-SplineQuasiBandlimit-pf6}  
  \end{align}
  for all $n \in \mN$, where the second line follows from the fact that $w_{k, n}: [-\pi; \pi] \to \mR$ is even and attains its maximum at the boundary as a convex function.
 
 Now it remains to bound the first term on the right-hand side of \eqref{eq:thm-SplineQuasiBandlimit-pf2}. By definition, it holds that $\bar \nu \geq \nu$, where equality holds if and only if $\nu \in 2\mN - 1$, in which case the considered term vanishes. Hence, we restrict to $\bar \nu > \nu$.  By transforming the integration variable and exploiting periodicity analogously as in \eqref{eq:thm-SplineQuasiBandlimit-pf3}, we obtain
 \begin{align}
    r^{-1}  \Norm{\cF(h)|_{(\bar\nu \Xi_ r) \setminus (\nu \Xi_r) }}^2 &= \bbparens{  \int_{-\bar \nu \pi}^{-\nu\pi} + \int_{\nu\pi}^{\bar \nu \pi}  }   \Abs{ \hat h_{\textup{per}}( \xi) }^2 \Abs{ \cF(B_k)( \xi)}^2 \D \xi \nnl
    &=  \int_{-\pi}^{\tilde \nu \pi}  \Abs{ \hat h_{\textup{per}}( \xi -(\bar \nu -1)\pi) }^2 \Abs{ \cF(B_k)( \xi)}^2 \D \xi \nnl
    &\qquad + \int_{- \tilde \nu \pi}^\pi  \Abs{ \hat h_{\textup{per}}( \xi) }^2 \Abs{ \cF(B_k)( \xi+ (\bar \nu -1)\pi))}^2 \D \xi \nnl
    &\leq  c_{k,0} r^{-1} \Norm{\cF(h)|_{\Xi_ r }}^2, \qquad c_{k,0} := \!\sup_{\xi \in [-\pi; \pi]} w_{k, 0}(\xi) \label{eq:thm-SplineQuasiBandlimit-pf7} \\ 
   w_{k, 0}(\xi) &:=  \begin{cases}
                                                                                   \frac{ \Abs{ \cF(B_k)( \xi - (\bar \nu -1)\pi))}^2 }{ \Abs{ \cF(B_k)( \xi)}^2 }  &\textup{for } \xi \leq - |\tilde \nu| \pi \\
                                                                                   \frac{ \Abs{ \cF(B_k)( \xi - (\bar \nu -1)\pi))}^2 }{ \Abs{ \cF(B_k)( \xi)}^2 } + \frac{ \Abs{ \cF(B_k)( \xi + (\bar \nu -1)\pi))}^2 }{ \Abs{ \cF(B_k)( \xi)}^2 }  &\textup{for } -\tilde \nu < \xi < \tilde \nu  \\
                                                                                   \frac{ \Abs{ \cF(B_k)( \xi + (\bar \nu -1)\pi))}^2 }{ \Abs{ \cF(B_k)( \xi)}^2 }   &\textup{for } \xi > |\tilde \nu| \pi \\
                                                                                   0 &\textup{else}
                                                                                   \end{cases},
 \nonumber
 \end{align}
 where $\tilde \nu = \bar \nu - \nu - 1 \in (-1 ; 1)$ has been inserted. Since $(\bar \nu -1)\pi$ is necessarily an integer-multiple of $2\pi$, we may again use the relation \eqref{eq:thm-SplineQuasiBandlimit-pf5}  to simplify $w_{k, 0}$:
 \begin{align}
  w_{k, 0}(\xi) &:=  \begin{cases}
                      \frac{\xi^{2(k+1)}}{ (\xi - (\bar \nu -1)\pi)^{2(k+1)}}  &\textup{for } \xi \leq - |\tilde \nu| \pi \\
                      \frac{\xi^{2(k+1)}}{ (\xi - (\bar \nu -1)\pi)^{2(k+1)}} + \frac{\xi^{2(k+1)}}{ (\xi + (\bar \nu -1)\pi)^{2(k+1)}} &\textup{for } -\tilde \nu < \xi < \tilde \nu  \\
                      \frac{\xi^{2(k+1)}}{ (\xi + (\bar \nu -1)\pi)^{2(k+1)}}   &\textup{for } \xi > |\tilde \nu| \pi \\
                      0 &\textup{else}
                     \end{cases}, \label{eq:thm-SplineQuasiBandlimit-pf8}  
 \end{align}
  The function $w_{k, 0}$ can be readily seen to be smooth and convex on each of the intervals $[-\pi; -|\tilde \nu| \pi)$, $(-\tilde \nu \pi; \tilde \nu  \pi)$ and $(|\tilde \nu| \pi; \pi]$. Consequently, the supremum over $[-\pi; \pi]$ is attained at one of the six boundary points of these intervals. By the symmetry $w_{k, 0}(-\xi) = w_{k, 0}(\xi)$, it is furthermore sufficient to consider non-negative values of $\xi$.
  
  We first consider the case $\tilde \nu \in (-1;0]$. Then the interval $(-\tilde \nu \pi; \tilde \nu  \pi)$ is empty and the definition of $w_{k,0}$ simplifies accordingly so that $c_{k,0}$ can be computed as
  \begin{align}
   c_{k,0} &= \max \left\{  \lim_{\xi \searrow |\tilde \nu| \pi} w_{k, 0}(\xi) , \;  w_{k, 0}(\pi) \right\}  = \max \left\{  \frac{|\tilde \nu|}{ \Parens{ (\bar\nu-1) + |\tilde \nu| }   } , \; \frac 1 {\bar \nu }  \right\}^{ 2(k+1)}  = \frac 1 {\bar \nu ^{ 2(k+1)}}   \label{eq:thm-SplineQuasiBandlimit-pf9}  
  \end{align}
  for all $ \bar \nu \in 2\mN -1$.
 On the other hand, if $\tilde \nu \in (0;1)$, then also the interior domain-part $(-\tilde \nu \pi; \tilde \nu  \pi)$ has to be considered in the computation of the supremum:
  \begin{align}
   &c_{k,0} = \max \left\{  \lim_{\xi \nearrow \tilde \nu \pi} w_{k, 0}(\xi), \; \lim_{\xi \searrow \tilde \nu \pi} w_{k, 0}(\xi) , \;  w_{k, 0}(\pi) \right\}  \nnl
   &= \max \left\{  \frac{\tilde \nu^{ 2(k+1)}}{ \Parens{ (\bar\nu-1) + \tilde \nu } ^{ 2(k+1)} }+   \frac{\tilde \nu^{ 2(k+1)}}{ \Parens{ (\bar\nu-1) - \tilde \nu } ^{ 2(k+1)} } , \; \frac{\tilde \nu^{ 2(k+1)}}{ \Parens{ (\bar\nu-1) + \tilde \nu } ^{ 2(k+1)} } , \; \frac 1 {\bar \nu ^{ 2(k+1)}}  \right\} \nnl
   &\stackrel{(\nu = \bar \nu - 1 - \tilde \nu)}= \frac 1 {\bar \nu ^{ 2(k+1)}} +  \underbrace{ \max \left\{  \frac{\max\{\tilde \nu, 0\}^{ 2(k+1)}}{ \Parens{ \nu + 2\tilde \nu } ^{ 2(k+1)} }+   \frac{\max\{\tilde \nu, 0\}^{ 2(k+1)}}{   \nu   ^{ 2(k+1)} } - \frac 1 {\bar \nu ^{ 2(k+1)}} , \; 0  \right\} }_{c_{\band,0} ( k, \nu )}.  \label{eq:thm-SplineQuasiBandlimit-pf10}  
  \end{align}
  By comparing to \eqref{eq:thm-SplineQuasiBandlimit-pf9}, it can be seen that equality between the left-hand side and the bottom line of \eqref{eq:thm-SplineQuasiBandlimit-pf10} remains valid for $\tilde \nu \in (-1;0]$, i.e.\ holds true in general.

  By inserting \eqref{eq:thm-SplineQuasiBandlimit-pf4}, \eqref{eq:thm-SplineQuasiBandlimit-pf6}, \eqref{eq:thm-SplineQuasiBandlimit-pf7} and \eqref{eq:thm-SplineQuasiBandlimit-pf10} into \eqref{eq:thm-SplineQuasiBandlimit-pf2}, we finally arrive at
 \begin{align}
     \Norm{\cF(h)|_{\compl{(\nu\Xi_ r)}}}^2 &\leq \bbparens{c_{\band,0} ( k, \nu ) + \frac 1 {\bar \nu ^{ 2(k+1)}} +  \sum_{n = 1 +  \ceil{(\nu-1)/2}}^\infty c_{k,n}  } \Norm{\cF(h)|_{ \Xi_ r }}^2 \nnl
     &=   \bbparens{  c_{\band, 0} ( k, \nu )  + \! \sum_{n =  \ceil{(\nu-1)/2}}^\infty \frac 2 { \Parens{ 2n + 1 }^{2(k+1)} }  }   \Norm{\cF(h)|_{ \Xi_ r }}^2 \nnl
     &=   c_{\band} ( k, \nu )  \Norm{\cF(h)|_{ \Xi_ r }}^2 \stackrel{ \Xi_ r \subset \nu \Xi_ r}\leq  c_{\band} ( k, \nu )^2    \Norm{\cF(h)|_{ \nu \Xi_ r }}^2 \label{eq:thm-SplineQuasiBandlimit-pf11}  
 \end{align}
 The assertion now follows by exploiting that $\Norm{\cF(h)}^2 = \Norm{\cF(h)|_{\compl{(\nu\Xi_ r)}}}^2 + \Norm{\cF(h)|_{\nu\Xi_ r}}^2$:
  \begin{align}
    \CBand(k,\nu)^2 \Norm{\cF(h)}^2 &= \CBand(k,\nu)^2 \Norm{\cF(h)|_{ \nu\Xi_ r}}^2 + \CBand(k,\nu)^2  \Norm{\cF(h)|_{\compl{(\nu\Xi_ r)}}}^2 \nnl
    &\stackrel{\eqref{eq:thm-SplineQuasiBandlimit-pf11}}\geq \CBand(k,\nu)^2 \Parens{ \frac 1 {c_{\band} ( k, \nu )} + 1 } \Norm{\cF(h)|_{\compl{(\nu\Xi_ r)}}}^2  = \Norm{\cF(h)|_{\compl{(\nu\Xi_ r)}}}^2 
 \end{align}
 \vspace{.5em}
 
 \paragraph{Negative result for $\nu < 1$:} Now let $\nu < 1$. Then, by the theory of Fourier series, there exists a sequence $(b_j)_{j \in \mZ}$ such that
 \begin{align}
  \hat b_\nu (\xi) :=  \exp \Parens{ - \I r \xi  \bo }  \sum_{j \in \mZ} b_{j} \exp \Parens{ - \I r \xi  j } = \begin{cases}
                                                                               1 &\textup{if } \xi \in \Xi_r \setminus (\nu \Xi_r) \\
                                                                               0 &\textup{if } \xi \in \nu \Xi_r
                                                                              \end{cases}, \quad \xi \in [-\pi; \pi] 
 \end{align}
 in an $L^2$-sense. If we define $h := \sum_{j\in \mZ} b_j B_k(\cdot/r - j - \bo)$ as the corresponding B-spline, then the periodic part of $\cF(h)$  in \eqref{eq:thm-PrincipalLeakage-1} is given by $\hat h _{\textup{per}} = \hat b_\nu$. Hence, it follows that $\supp(\cF(h)) \cap [-\pi;\pi] = \supp(\hat b_\nu) \subset \compl{(\nu \Xi_r)}$, i.e.\ $\cF(h) = \cF(h)|_{\compl{(\nu \Xi_r)}}$.
 The constructed example shows that no non-trivial bound of the form \eqref{eq:thm-SplineQuasiBandlimit-1} may hold true for $\nu < 1$.
 \end{proof}
\vspace{1em}

\bibliographystyle{siamplain}
\bibliography{/home/simon/SharedObjects/literature.bib}

\end{document}